\documentclass[a4paper]{amsart}
\usepackage[utf8]{inputenc}
\usepackage[T1]{fontenc}
\usepackage{amsfonts}
\usepackage{amsmath}
\usepackage{amssymb}
\usepackage{amstext}
\usepackage{amsthm}
\usepackage[shortlabels]{enumitem}
\usepackage{geometry}
\usepackage{fullpage}
\usepackage{hyperref}
\usepackage{cleveref}
\usepackage{mathtools}

\usepackage{tikz}                           
\usepackage{ifthen}							
\usetikzlibrary{arrows}						
\usepackage{tikz-cd}                        
\usetikzlibrary{decorations.pathmorphing}   
\usetikzlibrary{decorations.markings}		
\usetikzlibrary{decorations.pathreplacing}	
\usetikzlibrary{calc}		
\usetikzlibrary{3d}							
\usetikzlibrary{matrix,arrows}

\usepackage{ifthen}
\usepackage{footmisc}

\usepackage{shuffle}
\usepackage{pifont} 
\usepackage{caption} 

\usepackage{some_packages}
\usepackage{some_macros}

\newtheorem{theorem}{Theorem}[section]
\newtheorem{lemma}[theorem]{Lemma}

\newtheorem{corollary}[theorem]{Corollary}
\newtheorem{proposition}[theorem]{Proposition}

\theoremstyle{definition}
\newtheorem{definition}[theorem]{Definition}
\newtheorem{example}[theorem]{Example}

\theoremstyle{remark}
\newtheorem{remark}[theorem]{Remark}
\numberwithin{equation}{section}

\title{The Loehr-Remmel bijection in the $n\times kn$ grid and sandpiles}

\author{Michele D'Adderio}
\address{Universit\`a di Pisa\\Dipartimento di Matematica\\ Largo Bruno Pontecorvo 5, 56127 Pisa\\ Italy}\email{michele.dadderio@unipi.it}

\author{Alessio Sgubin}
\address{LIGM, Universit\'e Gustave-Eiffel\\ CNRS, ENPC, ESIEE-Paris\\ 5 Boulevard Descartes,
	Champs-sur-Marne, 77454 Marne-la-Vall\'ee cedex 2\\ France}\email{alessio.sgubin@univ-eiffel.fr}

\begin{document}

\begin{abstract}
	We extend the $\mathsf{pmaj}$ statistic of Loehr and Remmel to labelled Dyck paths in the $n \times kn$ grid, and give a simple description of their bijection, proving that it sends the bistatistic $(\mathsf{dinv},\mathsf{area})$ to $(\mathsf{area}, \mathsf{pmaj})$. As a corollary we prove a new combinatorial formula for $\nabla^k e_n$ ($k \geq 1$). Moreover, we provide an explicit description of the recurrent configurations of the sandpile model on a family of graphs $G_{\mu, \nu}^{(k)}$, indexed by an integer $k \geq 1$ and two compositions $\mu$ and $\nu$: at $k = 1$ these are the clique-independent graphs of D'Adderio et al. Finally, we define a $\mathsf{delay}$ statistic on these configurations, and we show that, together with the usual level statistic, it can be used to provide a new combinatorial interpretation of the polynomials $\langle \nabla^k e_n,e_\mu h_\nu \rangle$ from the $(n,kn)$-shuffle theorem. At $k = 1$ we recover the main results of D'Adderio et al.
\end{abstract}

\maketitle

\section{Introduction}

In a recent breakthrough~\cite{CarlssonMellitShuffle} Carlsson and Mellit gave a positive solution to the long-standing \emph{shuffle conjecture}~\cite{HHLRU-2005}, which states a combinatorial formula for the Frobenius characteristic of the so-called diagonal harmonics. More precisely, this theorem provides the monomial expansion of the symmetric function $\nabla e_n$, where $e_n$ is the elementary symmetric function of degree $n$ in the variables $x_1,x_2,\dots$, and $\nabla$ is the famous \emph{nabla} operator introduced by Bergeron and Garsia in the 90's (cf.~\cite{Bergeron-Garsia-Haiman-Tesler-Positivity-1999}). In this formula, to each \emph{labelled Dyck path} in the $n\times n$ grid corresponds a monomial, where the variables $x_1,x_2,\dots$ keep track of the labels, while the variables $q$ and $t$ keep track of the bistatistic ($\dinv$, $\area$).

The bistatistic ($\dinv$, $\area$) has been successfully extended to other combinatorial objects related to Dyck paths, providing an impressive proliferation of theorems and conjectures (see e.g.~\cite{Carlsson-Mellit-Compositio,BHMPS_any_line,BHMPS_Loehr-Warrington,IraciRomero_DeltaTheta,BergeronHaglungIraciRomero_supernabla,IraciPagariaPaolini_falling_stars,KimLeeOh_science_fiction,KimOh_ext_SciFi,DAdderioIraciVandenWyngaerd_Theta,DAdderio_Mellit_Delta,DILRV_tiered_trees}). 

In particular, in~\cite{Mellit_Rational} Mellit proved a \emph{rational} extension of the \emph{shuffle theorem}, which encompasses a formula for $\nabla^k e_n$ ($k\geq 1$) in terms of labelled Dyck paths in the $n\times kn$ grid.

\smallskip

The statistic $\dinv$, discovered by Haiman, was originally defined only on ``unlabelled'' Dyck paths, and together with $\area$ provided a combinatorial interpretation of the famous $(q,t)$-Catalan $\langle \nabla e_n,e_n\rangle$ (here $\langle - ,-\rangle$ denotes the Hall scalar product on symmetric functions). At the same time Haglund defined a statistic $\bounce$ on Dyck paths, that together with $\area$ provided another interpretation of the $(q,t)$-Catalan. These two interpretations can be proved to be the same by an explicit bijection, usually called $\zeta$, due to Haglund and Loehr~\cite{Haglund_Conjectured,Haglund_Loehr_Conjectured}, sending a Dyck path in the $n\times n$ grid into another one whose bistatistic $(\area,\bounce)$ coincides with the bistatistic $(\dinv,\area)$ of the original one. Both the statistic $\bounce$ and the bijection have been extended first by Loehr~\cite{Loehr_HigherCatalan} to Dyck paths in the $n\times kn$ grid, and then to more general Dyck paths in a rectangular grid, the maps going under the collective name of \emph{sweep maps}, see e.g.~\cite{ArmstrongLoehrWarrington_Sweep,ThomasWilliams_Sweeping}. Cf.\ also \cite{LoehrNiese2014} where an interpretation of $\nabla^k e_n$ in terms of a bistatistic $(\area',\bounce)$ is provided. 

\smallskip

In~\cite{LoehrRemmel_pmaj} Loehr and Remmel introduced a new statistic $\pmaj$ for labelled Dyck paths in the $n\times n$ grid, as an extension of the $\bounce$ to the labelled objects. In this way they provided an alternative combinatorial interpretation of $\nabla e_n$: their formula is in terms of the same objects, but using the bistatistic ($\area$, $\pmaj$). In fact, they proved that the two combinatorial formulas coincide, by defining a remarkable bijection on $(n,n)$-parking functions (identified with Dyck paths with a standard labelling) sending the bistatistic $(\dinv,\area)$ to $(\area, \pmaj)$. This bijection is an extension of the original $\zeta$ map. See~\cite{DAdderio_Sgubin_Expo} for an expository article on this bijection.

Unlike for $(\dinv,\area)$, there are not many extensions of the bistatistic $(\area,\pmaj)$ to ``labelled'' objects that are known, even conjecturally (but cf.~\cite[Sections~2.3~and~2.4]{DAdderioIraciVandenWyngaerd_bible}). Actually, in the same \cite{LoehrRemmel_pmaj}, Loehr and Remmel define their bijection $G$ on more general objects, including the labelled Dyck paths in the $n\times kn$ grid, and they propose a \emph{pmaj} for those object, but they remark that their statistic ``does not seem to have a particularly simple direct definition (not relying on the bijection $G$)''.

\medskip

Recently, in~\cite{DDILLV} the authors proved a new combinatorial interpretation of $\nabla e_n$ in terms of the famous \emph{sandpile model}, first introduced in~\cite{BakTangWiesenfeld} by Bak, Tang and Wiesenfeld in the context of ``self-organized criticality'' in statistical mechanics (see~\cite{Klivans_Sandpiles} for a nice introductory monograph). Their formula uncovered a surprising bridge between this heavily studied combinatorial dynamical system and the $q$,$t$-combinatorics related to Madonald polynomials, which has the shuffle theorem of Carlsson and Mellit as one of its cornerstones. In~\cite{DDILLV} the labelled Dyck paths in the $n\times n$ grid are mapped bijectively to \emph{sorted recurrent configurations} of suitable graphs; in this correspondence, the statistic $\pmaj$ is sent to a statistic $\del$ on those configurations.

\medskip

The first main contribution of the present article is a simple direct extension of the statistic $\pmaj$ of Loehr and Remmel in~\cite{LoehrRemmel_pmaj} to $(n,kn)$-parking functions ($k\geq 1$), identified with labelled Dyck paths in the $n\times kn$ grid with a standard labelling. It turns out that this statistic coincides with the one proposed in \cite{LoehrRemmel_pmaj}. Indeed, we provide also a direct description of the Loehr-Remmel bijection for $(n,kn)$-parking functions ($k\geq 1$), extending the description in \cite{DAdderio_Sgubin_Expo} in the case $k=1$, and we prove that it maps the bistatistics $(\dinv,\area)$ to $(\area,\pmaj)$.
As a corollary, using the rational shuffle theorem of Mellit~\cite{Mellit_Rational}, we get a new formula for the symmetric function $\nabla^k e_n$ in terms of $(n,kn)$-parking functions with the bistatistic $(\area,\pmaj)$, which can be written as
\[
	\nabla^k e_n=\sum_{D\in \PF_{n,kn}}q^{\area(D)}t^{\pmaj(D)}L_{n,\Des(\wrow(D)^{-1})}.
\]
We refer to Section~\ref{sec:SFapplications} for the missing definitions (cf.\ Corollary~\ref{cor:identitypmaj}). 
Indeed, our arguments and notation will follow closely the expository~\cite{DAdderio_Sgubin_Expo}, which concerns the case $k=1$.

Our second main contribution is an extension of the main results in~\cite{DDILLV}: indeed we provide a bijection between labelled Dyck paths in the $n\times kn$ grid and sorted recurrent configurations of suitable graphs; again, in this correspondence, the statistic $\pmaj$ is sent to a statistic $\del$ on those configurations. As a result we get a new formula for $\nabla^k e_n$ in terms of sorted recurrent configurations of the sandpile model with the bistatistic $(\area,\pmaj)$. A more precise statement can be the following (cf.\ Corollary~\ref{cor:identitydelay}: given two compositions $\mu,\nu$ such that $|\mu| + |\nu| = n$, we have
\[
\langle \nabla^k e_n, e_\mu h_\nu \rangle = \sum_{c \in \SortRec_k(\mu;\nu)}q^{\lev(c)}t^{\del(c)}.
\]
We refer to Sections~\ref{sec:sandpile} and~\ref{sec:sortedrec} for the missing definitions. At $k=1$ we recover precisely the correspondence in~\cite{DDILLV}.\smallskip\\
We like to mention that our extension of the $\pmaj$ was actually inspired by the corresponding $\del$ in the sandpile model. 

\medskip

The rest of the present article is organized in the following way. In Section~\ref{sec:LDyck} we introduce all the needed combinatorial definitions related to labelled Dyck paths in order to state the formula for $\nabla^k e_n$ in terms of parking functions with statistics $\area$ and $\dinv$, and then define our statistic $\pmaj$. In Section~\ref{sec:phi_nk} we give our definition of the map $\phi_n^{(k)}$, giving the Loehr-Remmel bijection on $(n,kn)$-parking function, we prove that it is well defined and that it sends $\dinv$ to $\area$. In Section~\ref{sec:inverse_psi_nk} we define the inverse of $\phi_n^{(k)}$, which we denote $\psi_n^{(k)}$, and in Section~\ref{sec:psi_welldefined} we show that it is a well defined function on $(n,kn)$-parking functions. In Section~\ref{sec:is_the_inverse} we prove that $\phi_n^{(k)}$ and $\psi_n^{(k)}$ are inverse of each others, and we show that $\phi_n^{(k)}$ sends $\area$ to $\pmaj$. In Section~\ref{sec:SFapplications} we apply $\phi_n^{(k)}$ to deduce our formula for $\nabla^k e_n$ in terms of the bistatistic $(\area,\pmaj)$.

In Section~\ref{sec:sandpile} we introduce the notation and the basic notions of the sandpile model, together with our family $G_{\mu,\nu}^{(k)}$ of graphs. In Section~\ref{sec:sortedrec} we introduce the sorted recurrent configurations for $G_{\mu,\nu}^{(k)}$ and their statistics $\lev$ and $\del$. In Sections~\ref{sec:pfandsand} and~\ref{sec:bijSortedDyck} we construct the bijection between sorted recurrent configurations and labelled Dyck paths, deducing from it our formula for $\langle \nabla^k e_n,e_\mu h_\nu\rangle$ in terms of the bistatistic $(\lev,\del)$.

\medskip

\textbf{Acknowledgments.} We thank Nick Loehr for his useful feedback on the first version of the present article. D'Adderio is partially supported by PRIN 2022A7L229 ALTOP, and by INDAM research group GNSAGA.\\
Sgubin is Co-Funded by the European Union. Views and opinions expressed are however those of the author(s) only and do not necessarily reflect those of the European Union. Neither the European Union nor the granting authority can be held responsible for them.

\medskip

\textbf{AI disclosure.} No AI tool has been used in this work.

\section{Labelled Dyck paths}\label{sec:LDyck}

Fix $n,k\in \N$ such that $n, k\geq 1$. We will use the common notation $[n]:=\{1,2,\dots,n\}$.

\begin{definition}\label{def:cl_Dyckpath}
	A \emph{Dyck path} in the $n\times kn$ grid is a lattice path consisting of \emph{north steps} (going from a lattice point $(a,b)$ to the point $(a,b+1)$) and \emph{east steps} (going from a lattice point $(a,b)$ to $(a+1,b)$) going from $(0,0)$ to $(kn,n)$ that never goes below the \emph{main diagonal} $y=x/k$. We denote by $\Dyck_{n,kn}$ the set of all Dyck paths in the $n\times kn$ grid. 
	
	The \emph{diagram} of a Dyck path is its picture in the $n\times kn$ square grid that bounds it.  
\end{definition}
The diagram of a Dyck path in $\Dyck_{7,21}$ is shown in Figure~\ref{fig:Dyck}. The squares intersecting the main diagonal are shaded in yellow.
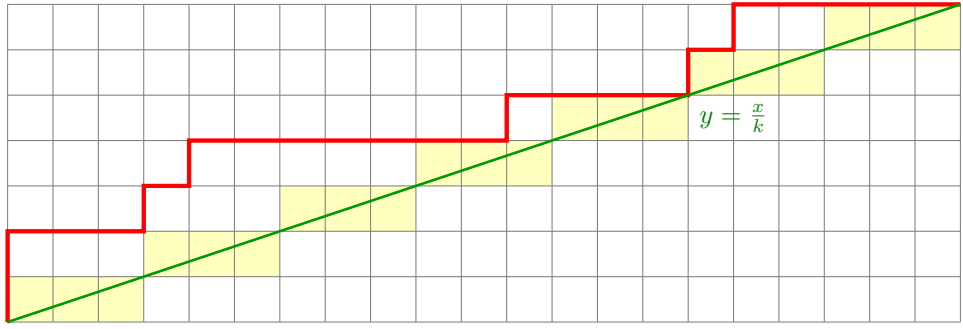
\begin{figure}[ht]
	\centering
	\begin{tikzpicture}[scale=0.6]
		\parkfunc{7}{3}{21,21,18,17,10,6,5}{$\ $,$\ $,$\ $,$\ $,$\ $,$\ $,$\ $}{1};
		\draw[red,line width = 1.6pt] (0,0)|-(3,2)|-(4,3)|-(11,4)|-(15,5)|-(16,6)|-(21,7);
		\draw[green!60!black, line width = 1pt] (0,0) -- (21,7);
		\node[green!50!black] at (16,4.5) {$y = \frac{x}{k}$};
	\end{tikzpicture}
	\caption{A Dyck path in $\Dyck_{7,21}$.}\label{fig:Dyck}
\end{figure}

Given a Dyck path $D\in \Dyck_{n,kn}$, we will refer to the horizontal and vertical sequences of unit squares in its diagram as the \emph{rows} and the \emph{columns} of $D$, respectively. We number columns increasingly from left to right, and rows from bottom to top, both starting from $1$.

\begin{definition}\label{def:cl_labelledDyck}
	A \emph{labelled} Dyck path in the $n\times kn$ grid is a Dyck path $D \in \Dyck_{n,kn}$ where each north step is labelled by a positive integer, so that the labels of north steps on the same vertical line have increasing labels when read from bottom to top. If the set of $n$ labels of $D$ is precisely $[n]$ we will call $D$ a \emph{$(n,kn)$-parking function}. The \emph{diagram} of a labelled Dyck path is the diagram of the associated Dyck path with the label of each vertical step appearing in the unit square to its right. For every $i\in [n]$ we denote by $\ell_i(D)$ the label of the vertical step in row $i$, and by $f_D(i)$ the column in which the label $i$ appears in the diagram of $D$. Finally, we denote by $\LDyck_{n,kn}$ the set of all labelled Dyck paths in the $n\times kn$ grid, and by $\PF_{n,kn}$ the set of $(n,kn)$-parking functions. 
\end{definition}

\begin{remark}\label{rem:PFnkn}
	The name ``parking function'' comes from the following easy observation.\\
	If $D$ is a $(n,kn)$-parking functions according to our definition, then the functions $i\mapsto f_D(i)$ are usually called $(n,kn)$-parking functions in the literature.\smallskip\\
	Indeed, it is easy to check that for $D\in \PF_{n,kn}$, the function $f_D:[n]\to [kn]$ satisfies
	\[
		\# f_D^{-1}([1+k(i-1)])\geq i\qquad \text{for every }i\in [n],
	\]
	and viceversa any such function occurs as $f_D$ for some $D\in \PF_{n,kn}$.\smallskip\\
	Since $D\mapsto f_D$ is a bijection between $\PF_{n,kn}$ and these objects, we simply identify them, as it is customary for $k=1$. Indeed, for $k=1$, the functions $i\mapsto f_D(i)$ are precisely the usual parking functions of size $n$.
\end{remark}
The diagram of a labelled Dyck path $D$ in $\PF_{7,12}\subseteq \LDyck_{7,21}$ is shown in Figure~\ref{fig:labelled_Dyck}; its labels are $\ell_1(D)=2$, $\ell_2(D)=4$, $\ell_3(D)=6$, $\ell_4(D)=7$, $\ell_5(D)=1$, $\ell_6(D)=5$, and $\ell_7(D)=3$, while $f_D(1)=12$, $f_D(2)=1$, $f_D(3)=17$, $f_D(4)=1$, $f_D(5)=16$, $f_D(6)=4$, and $f_D(7)=5$.

\begin{figure}[ht]
	\centering
	\begin{tikzpicture}
	\fill[blue!20!white] (0.0,0.6) -- (1.8,0.6) -- (1.8,1.2) -- (0.0,1.2) -- cycle;
	\fill[blue!20!white] (1.8,1.2) -- (3.6,1.2) -- (3.6,1.8) -- (1.8,1.8) -- cycle;
	\fill[blue!20!white] (2.4,1.8) -- (5.4,1.8) -- (5.4,2.4) -- (2.4,2.4) -- cycle;
	\fill[blue!20!white] (6.6,2.4) -- (7.2,2.4) -- (7.2,3.0) -- (6.6,3.0) -- cycle;
	\fill[blue!20!white] (9.6,3.6) -- (10.8,3.6) -- (10.8,4.2) -- (9.6,4.2) -- cycle;
	\parkfunc{7}{3}{21,21,18,17,10,6,5}{2,4,6,7,1,5,3}{.6}
	\node[blue] at (14, .3) {$w_1(D) = 0$};
	\node[blue] at (14, .9) {$w_2(D) = 3$};
	\node[blue] at (14,1.5) {$w_3(D) = 3$};
	\node[blue] at (14,2.1) {$w_4(D) = 5$};
	\node[blue] at (14,2.7) {$w_5(D) = 1$};
	\node[blue] at (14,3.3) {$w_6(D) = 0$};
	\node[blue] at (14,3.9) {$w_7(D) = 2$};
	\node[blue] at (14,4.5) {$\warea(D)$};
	\end{tikzpicture}
	\caption{A labelled Dyck path $D$ in $\PF_{7,21}\subseteq \LDyck_{7,21}$ and its area word. The unit squares contributing to the area are shaded in blue.}\label{fig:labelled_Dyck}
\end{figure}
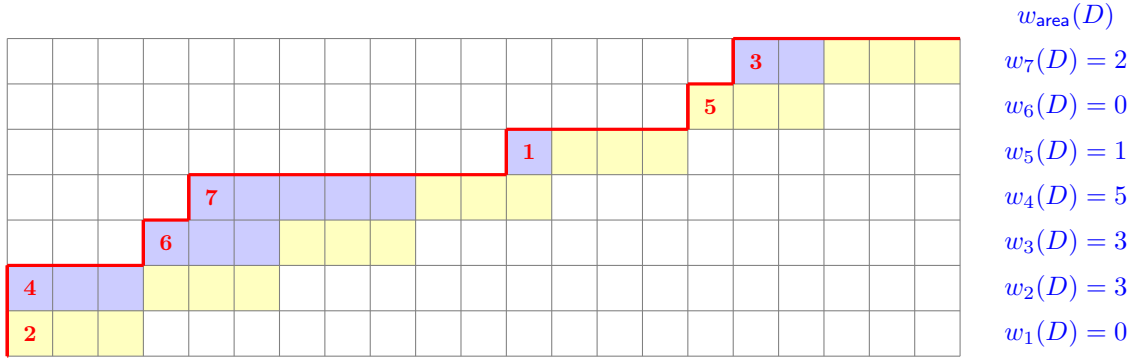

\subsection{The statistic $\area$}
The statistic $\area$ is classical.

\begin{definition}\label{def:gen_area}
Given $D \in \Dyck_{n,kn}$, we define its \emph{area word} as the $n$-tuple
\[w_{\area}(D) := (w_1(D),\dots,w_n(D))\]
where for each $i \in [n]$
	\[
	w_i(D) := \#\text{ of integer squares in the $i^{\text{th}}$-row between the path and the main diagonal}.
	\]
Then the \emph{area} of $D \in \Dyck_{n,k}$ is
	\[
	\area(D) := w_1(D) + \dots + w_n(D).
	\]
Given $D \in \LDyck_{n,kn}$, its \emph{area word} and its \emph{area} are defined as the ones of the corresponding Dyck path (i.e.\ ignoring the labels). 
\end{definition}
\begin{example}\label{ex:area}
The area word of the labelled Dyck path $D$ in Figure~\ref{fig:labelled_Dyck} is shown in the picture, i.e.\ $\warea(D)=(0,3,3,5,1,0,2)$, hence its area is $\area(D)=0+3+3+5+1+0+2=14$.	
\end{example}

\subsection{The statistic $\dinv$}\label{sec:dinv}

For the statistic $\dinv$ we will use a reformulation of the geometric description given in~\cite{hicks_dinv_descr}, which consists of two contributions, $\tdinv$ and $\dinvcorr$.
\begin{definition}\label{def:gen_tdinv}
	Given a labelled Dyck path $D \in \LDyck_{n,kn}$, 
	for every $i\in [n]$ we set
	\[
		a_{\ell_i(D)}(D):=w_i(D).
	\]
	For example, for $D$ in Figure~\ref{fig:labelled_Dyck} we have $a_1(D)=w_5(D)=1$, $a_2(D)=w_1(D)=0$, $a_3(D)=w_7(D)=2$, $a_4(D)=w_2(D)=3$, $a_5(D)=w_6(D)=0$, $a_6(D)=w_3(D)=3$, and $a_7(D)=w_4(D)=5$.

	\smallskip

	Now we say that a pair $(i,j)\in [n]\times [n]$ of labels is a \emph{temporary diagonal inversion} of $D$ if any of the following two conditions is satisfied:
	\begin{enumerate}[label=(\Alph*)]
		\item\label{it:gen_dinv_condA} $f_D(i)<f_D(j)$, $i<j$ and $a_j(D)-k<a_i(D)\leq a_j(D)$.
		\item\label{it:gen_dinv_condB} $f_D(i)<f_D(j)$, $i>j$ and $a_j(D)<a_i(D)\leq a_j(D)+k$.
	\end{enumerate} 

	We denote by $\Tdinv(f)$ the multiset\footnote{Notice that $\Tdinv(D)$ will always be a set, but we consider it a multiset since later we will add it to the multiset $\Dinvcorr(D)$.} of all temporary diagonal inversions of $D$, and we define the $\tdinv$ of $D$ as the cardinality of $\Tdinv(D)$, i.e. \[\tdinv(D):=\# \Tdinv(D).\]
\end{definition}
This definition is better understood by the next geometric remark.

\begin{definition}
	Given a north step of a Dyck path $D \in \Dyck_{n,kn}$ with endpoints $(a,b)$ and $(a,b+1)$ we define its \emph{main shadow} as the region of $\R^2$ whose points $(x,y)$ satisfy the inequalities
	\begin{equation}\label{eq:mainshadow}
		\begin{cases}
			x \leq a\\
			\frac{1}{k}(x - a) < y - b \leq \frac{1}{k}(x - a) + 1
		\end{cases}\, ,
	\end{equation}
	while its \emph{upper shadow} is the region of $\R^2$ whose points $(x,y)$ satisfy the inequalities
	\begin{equation}\label{eq:uppershadow}
		\begin{cases}
			x \leq a\\
			\frac{1}{k}(x - a) < y - b-1 \leq \frac{1}{k}(x - a) + 1
		\end{cases} \, .
	\end{equation}
	These shadows are shown in Figure~\ref{fig:def_tdinv}; notice that they are parallel to the main diagonal $y=x/k$, and they are half open, as indicated by the dotted borders.
\end{definition}
\begin{figure}[ht]
	\centering
	\begin{tikzpicture}[scale=0.8]
		\node at (0,0) {
			\begin{tikzpicture}[scale=0.8]
				\node at (-2.8,3) {\bfseries Condition (A):};
				\node at (-2.8,2.3) {$i < j$};
				\node[green!60!black,scale=.8] at (-2.8,.3) {Main shadow of $j$};
				\node[orange!80!black,scale=.8] at (-2.8,1) {Upper shadow of $j$};
				\dyckpath{5}{3}{15,13,12,7,6}
				\fill[opacity=.3, green!60!black] (5.6,2.8) -- (5.6,2.1) -- (-.3,0.135) -- (-.3,.835) -- cycle;
				\fill[opacity=.2, orange!80!black] (5.6,3.5) -- (5.6,2.8) -- (-.7,.7) -- (-.7,1.4) -- cycle;
				\draw[green!60!black] (-.3,.835) -- (5.6,2.8);
				\draw[green!60!black, dashed] (-.3,.135) -- (5.6,2.1);
				\draw[very thick, green!60!black] (5.6,2.1) -- (5.6,2.8) node [midway, right] {\footnotesize{$j$}};
				\draw[orange!80!black] (-.7,1.4) -- (5.6,3.5);
				\draw[orange!80!black, dashed] (-.7,.7) -- (5.6,2.8);
				\draw[orange!80!black] (5.6,2.8) -- (5.6,3.5);
				\node[blue, scale=1.2] at (1.4,1.35) {\small{\textbullet}};
				\node[blue, scale=1.2] at (0,.65) {\small{\textbullet}};
				\draw[very thick, blue] (1.4,.7) -- (1.4,1.4) node [midway, right] {\footnotesize{$i$}};
				\draw[very thick, blue] (0,0) -- (0,.7) node [midway, right] {\footnotesize{$i$}};
			\end{tikzpicture}
		};
		\node at (0,-4.2) {
			\begin{tikzpicture}[scale=0.8]
				\node at (-2.8,3) {\bfseries Condition (B):};
				\node at (-2.8,2.3) {$i > j$};
				\node[green!60!black,scale=.8] at (-2.8,.3) {Main shadow of $j$};
				\node[orange!80!black,scale=.8] at (-2.8,1) {Upper shadow of $j$};
				\dyckpath{5}{3}{15,13,12,7,6}
				\fill[opacity=.3, green!60!black] (5.6,2.8) -- (5.6,2.1) -- (-.3,0.135) -- (-.3,.835) -- cycle;
				\fill[opacity=.2, orange!80!black] (5.6,3.5) -- (5.6,2.8) -- (-.7,.7) -- (-.7,1.4) -- cycle;
				\draw[green!60!black] (-.3,.835) -- (5.6,2.8);
				\draw[green!60!black, dashed] (-.3,.135) -- (5.6,2.1);
				\draw[very thick, green!60!black] (5.6,2.1) -- (5.6,2.8) node [midway, right] {\footnotesize{$j$}};
				\draw[orange!80!black] (-.7,1.4) -- (5.6,3.5);
				\draw[orange!80!black, dashed] (-.7,.7) -- (5.6,2.8);
				\draw[orange!80!black] (5.6,2.8) -- (5.6,3.5);
				\node[blue, scale=1.2] at (2.1,2.05) {\small{\textbullet}};
				\draw[very thick, blue] (2.1,1.4) -- (2.1,2.1) node [midway, right] {\footnotesize{$i$}};
			\end{tikzpicture}
		};
	\end{tikzpicture}
	\caption{Graphical depiction of conditions (A) and (B) in the definition of $\Tdinv(D)$.}\label{fig:def_tdinv}
\end{figure}
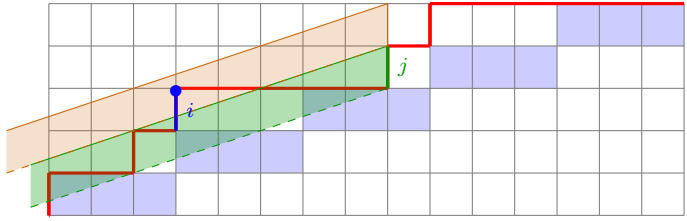

\begin{remark}\label{rem:gen_tdinv_graphic}
Conditions $(A)$ and $(B)$ correspond to the intersection of the north endpoint of the north step labelled $i$ with the main shadow and the upper shadow of (the north step labelled) $j$, respectively: see Figure~\ref{fig:def_tdinv}.
\end{remark}

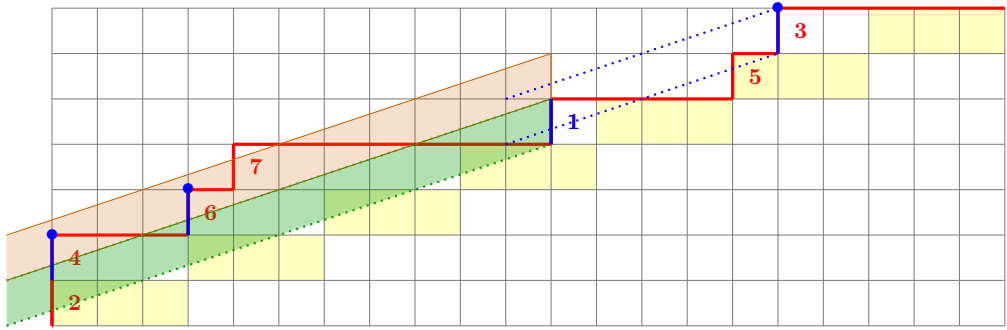
\begin{figure}[htbp]
	\centering
	\begin{tikzpicture}
		\parkfunc{7}{3}{21,21,18,17,10,6,5}{2,4,6,7,1,5,3}{.6}
		\draw[green!60!black, dotted, thick] (-.6,0) -- (6.6,2.4);
		\draw[green!60!black] (6.6,2.4) -- (6.6,3.0) -- (-.6,.6);
		\fill[green!60!black, opacity=.3] (-.6,0) -- (6.6,2.4) -- (6.6,3.0) -- (-.6,.6);
		\draw[orange!80!black, dotted, thick] (-.6,.6) -- (6.6,3);
		\draw[orange!80!black] (6.6,3) -- (6.6,3.6) -- (-.6,1.2);
		\fill[orange!80!black, opacity=.2] (-.6,.6) -- (6.6,3) -- (6.6,3.6) -- (-.6,1.2);
		\draw[blue, thick, dotted] (6,2.4) -- (9.6,3.6);
		\draw[blue, thick, dotted] (6,3.0) -- (9.6,4.2) -- (9.6,3.6);
		\draw[very thick, blue] (6.6,2.4) -- (6.6,3.0);
		\node[blue,font=\bfseries, scale={6/7}] at (6.9,2.7) {1};
		\draw[very thick, blue] (0,.6) -- (0,1.2);
		\node[blue] at (0,1.2) {\textbullet};
		\draw[very thick, blue] (1.8,1.2) -- (1.8,1.8);
		\node[blue] at (1.8,1.8) {\textbullet};
		\draw[very thick, blue] (9.6,4.2) -- (9.6,3.6);
		\node[blue] at (9.6,4.2) {\textbullet};
	\end{tikzpicture}
	\caption{Illustration of the temporary diagonal inversions involving the north step labelled $1$.}\label{fig:example_tdinv}
\end{figure}

\begin{example}
	For the labelled Dyck path $D$ in Figure~\ref{fig:labelled_Dyck} we have\footnote{We denote the multisets with double curly brackets.\label{fn:multisets}} (cf.\ Figure~\ref{fig:example_tdinv})
	\[\Tdinv(D)=\left\{\!\!\left\{ \begin{array}{c}
		(1, 3), (2, 3), (2, 5), (4, 1), (4, 3), (4, 6),\\
		(4, 7), (6, 1), (6, 3), (6, 5), (6, 7), (7, 3)
	\end{array} \right\}\!\!\right\}.\]
\end{example}

\begin{definition}\label{def:gen_EAST_labels}
	Given a labelled Dyck path $D\in \LDyck_{n,kn}$, for every east step of its diagram with endpoints $(c,d)$ and $(c+1,d)$ we define its \emph{main shadow} as the region of $\R^2$ whose points $(x,y)$ satisfy the inequalities:
	\[
		\begin{cases}
			y \leq d\\
			\frac{1}{k}(x - c) < y - d < \frac{1}{k}(x - c - 1).
		\end{cases}
	\]
	Notice that the main shadow of every east step of $D$ will intersect at least one north step of $D$ to its west.

	The \emph{east step labelling} of $D$ consists of labelling every east step of $D$ with the label of the first north step encountered by its shadow, moving from the east step towards west.
\end{definition}

\begin{example}\label{exa:gen_EAST_labels}
	In Figure~\ref{fig:gen_dinvcompA} we computed the east step labelling of the labelled Dyck path $D$ in Figure~\ref{fig:labelled_Dyck}.
\end{example}

\begin{remark}\label{rmk:east_step_labelling}
Observe that for every label $i$ of a north step of $D$ there are exactly $k$ east steps of $D$ that gets labelled with $i$. Cf.\ Remark~\ref{rem:dlambda}.
\end{remark}

\begin{definition}\label{def:gen_dinvcorr_set}
	Given $D\in \LDyck_{n,kn}$, consider its east step labelling. We define the \emph{multiset of diagonal inversion corrections}, denoted $\Dinvcorr(D)$, as the multiset of all pairs $(\lambda, \mu)$ where $\lambda$ is the label of a east step of $D$ entirely\footnote{Recall that the main shadow of a north step is half open.} contained in the main shadow of a north step of $D$ labelled $\mu$. We define $\dinvcorr(D)$ as the cardinality of the multiset $\Dinvcorr(D)$.
\end{definition}

\begin{example}\label{exa:gen_def_dinvcorrmulti}
	Consider the labelled Dyck path $D$ of Figure~\ref{fig:labelled_Dyck}, whose east step labelling is shown in Figure~\ref{fig:gen_dinvcompA}. We compute$^{\ref{fn:multisets}}$
	\[\Dinvcorr(D) = \left\{\!\!\left\{ \begin{array}{c}
		(1, 5), (1, 5),  (2, 5), (2, 5), (2, 1), (6, 1),  (4, 1), \\
		(1, 3), (6, 3), (6, 3), (4, 3), (4, 3), (4, 6), (4, 6)
	\end{array} \right\}\!\!\right\},\]
	hence $\dinvcorr(D)=14$.
\end{example}
\begin{figure}[hptb]
	\centering
	\begin{tikzpicture}
		\draw[opacity=0] (-.6,0) -- (-.6,1);    
		\parkfunc{7}{3}{21,21,18,17,10,6,5}{2,4,6,7,1,5,3}{.6}
		
		\node[green!60!black,font=\bfseries,scale={6/7}] at (  .3,1.4) {4};
		\node[green!60!black,font=\bfseries,scale={6/7}] at (  .9,1.4) {4};
		\node[green!60!black,font=\bfseries,scale={6/7}] at ( 1.5,1.4) {4};
		\node[green!60!black,font=\bfseries,scale={6/7}] at ( 2.1,2.0) {6};
		\node[green!60!black,font=\bfseries,scale={6/7}] at ( 2.7,2.6) {7};
		\node[green!60!black,font=\bfseries,scale={6/7}] at ( 3.3,2.6) {7};
		\node[green!60!black,font=\bfseries,scale={6/7}] at ( 3.9,2.6) {7};
		\node[green!60!black,font=\bfseries,scale={6/7}] at ( 4.5,2.6) {6};
		\node[green!60!black,font=\bfseries,scale={6/7}] at ( 5.1,2.6) {6};
		\node[green!60!black,font=\bfseries,scale={6/7}] at ( 5.7,2.6) {2};
		\node[green!60!black,font=\bfseries,scale={6/7}] at ( 6.3,2.6) {2};
		\node[green!60!black,font=\bfseries,scale={6/7}] at ( 6.9,3.2) {1};
		\node[green!60!black,font=\bfseries,scale={6/7}] at ( 7.5,3.2) {1};
		\node[green!60!black,font=\bfseries,scale={6/7}] at ( 8.1,3.2) {1};
		\node[green!60!black,font=\bfseries,scale={6/7}] at ( 8.7,3.2) {2};
		\node[green!60!black,font=\bfseries,scale={6/7}] at ( 9.3,3.8) {5};
		\node[green!60!black,font=\bfseries,scale={6/7}] at ( 9.9,4.4) {3};
		\node[green!60!black,font=\bfseries,scale={6/7}] at (10.5,4.4) {3};
		\node[green!60!black,font=\bfseries,scale={6/7}] at (11.1,4.4) {3};
		\node[green!60!black,font=\bfseries,scale={6/7}] at (11.7,4.4) {5};
		\node[green!60!black,font=\bfseries,scale={6/7}] at (12.3,4.4) {5};
		\draw[blue!50, dashed] (0,0.6) -- (9*6/10,4*6/10);
		\draw[blue!50, dashed] (4*6/10,3*6/10) -- (7*6/10,4*6/10);
		\draw[blue!50, dashed] (0,0) -- (21*6/10,7*6/10);
		\draw[blue!50, dashed] (11*6/10,4*6/10) -- (14*6/10,5*6/10);
		\draw[blue!50, dashed] (16*6/10,6*6/10) -- (19*6/10,7*6/10);
	\end{tikzpicture}
	\caption{The east step labelling of $D \in \PF_{7,21}\subseteq \LDyck_{7,21}$ from Figure~\ref{fig:labelled_Dyck}.}\label{fig:gen_dinvcompA}
\end{figure}
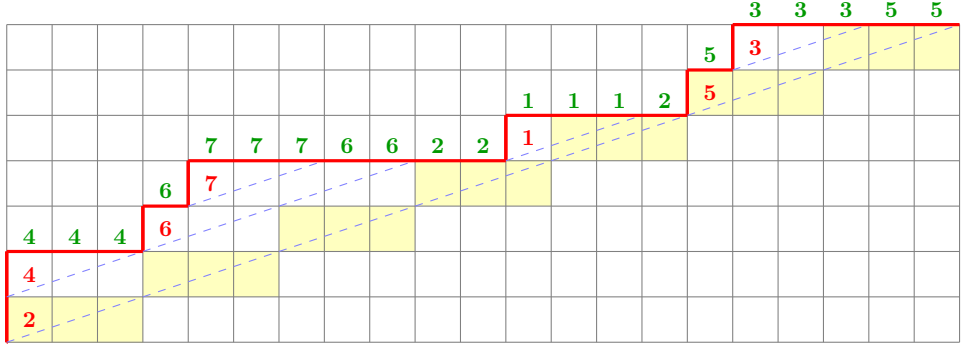

\begin{definition}
	Given $D\in \LDyck_{n,kn}$, we define the multiset $\Dinv(D)$ of diagonal inversions of $D$ as the sum of the multisets $\Tdinv(D)$ and $\Dinvcorr(D)$. We define the $\dinv$ of $D$ as the cardinality of $\Dinv(D)$, i.e.\ 
	\[\dinv(D):=\tdinv(D)+\dinvcorr(D).\]
\end{definition}

\begin{example}\label{ex:dinv}
	For the labelled Dyck path $D$ in Figure~\ref{fig:labelled_Dyck}, we have$^{\ref{fn:multisets}}$
	\begin{equation}\label{eq:gen_dinv_set_exa}
	\Dinv(f) = \left\{\!\!\!\left\{ \begin{array}{c}
		(2, 5), (2, 5), (2, 5),\\
		(1, 5), (1, 5), (2, 1),\\
		(1, 3), (2, 3), (1, 3),\\
		(4, 1), (4, 3), (4, 1), (4, 3), (4, 3),\\
		(6, 1), (6, 3), (4, 6), (6, 5), (6, 1), (6, 3), (6, 3), (4, 6), (4, 6),\\
		(4, 7), (6, 7), (7, 3)
	\end{array} \right\}\!\!\!\right\}\, ,
	\end{equation}
	so $\dinv(D)=\tdinv(D)+\dinvcorr(D)=12+14=26$.
\end{example}

\begin{definition}
	Given $D\in \LDyck_{n,kn}$, we define its \emph{diagonal reading word} $\wdiag(D)$ as the word obtained by reading the labels of the north steps of $D$ by sweeping its diagram with diagonals parallel to $y=x/k$, from bottom to top, reading the labels along the same diagonal from left to right. Notice that $\wdiag(D)$ is a permutation in $\mathfrak{S}_n$ (in one-line notation).
\end{definition}
For example, for the labelled Dyck path $D$ in Figure~\ref{fig:labelled_Dyck} we have $\wdiag(D)=2513467$.

\smallskip

Recall that given $S\subseteq [n-1]$, the \emph{(Gessel) fundamental quasisymmetric function} $L_{n,S}(x)$ is defined as
\[L_{n,S}=L_{n,S}(x):=\mathop{\sum_{1\leq i_1\leq i_2\leq \dots \leq i_n}}_{i_j=i_{j+1}\implies j\notin S}x_{i_1}x_{i_2}\dots x_{i_n}.\]
Also, we denote by $e_n$ the elementary symmetric function of degree $n$, and by $\nabla$ the famous \emph{nabla operator} from~\cite{Bergeron-Garsia-Haiman-Tesler-Positivity-1999}.\\
Finally, the \emph{descent set} of a permutation $\sigma\in \mathfrak{S}_n$ is 
\[
	\Des(\sigma):=\{i\mid \sigma(i)>\sigma(i+1)\}\subseteq [n-1].
\]

\smallskip

We can now state a special case of the so called \emph{rational shuffle theorem} of Mellit.
\begin{theorem}[{\cite{Mellit_Rational}}]\label{thm:rationalMellit}
	We have	
	\[
		\nabla^k e_n=\sum_{D\in \PF_{n,kn}}q^{\dinv(D)}t^{\area(D)}L_{n,\Des(\wdiag(D)^{-1})}.
	\]
\end{theorem}

\subsection{The statistic $\pmaj$}

Our definition of the $\pmaj$ statistic for labelled Dyck paths in $\LDyck_{n,kn}$ is new for $k>1$, while for $k=1$ we recover the original one defined in~\cite{LoehrRemmel_pmaj}. It turns out that it coincides with the \emph{pmaj} proposed in~\cite{LoehrRemmel_pmaj}.

\begin{definition}\label{def:gen_pmaj}
	Consider a labelled Dyck path $D \in \LDyck_{n,kn}$. We define a word in the alphabet $[n]$, denoted  $\sigma_{\pmaj}(D) := \sigma_1\sigma_2\dots \sigma_{kn}$, such that every letter appears $k$ times.\\
	Let $B_0' := \varnothing$ be the empty multiset$^{\ref{fn:multisets}}$ 
	and $\sigma_0 := n+1$. Then iterate the following steps for $m = 1,2,\dots,kn$.
	\begin{enumerate}[label=\roman*)]
		\item Set\footnote{We will use the exponential notation for the elements of multisets, indicating the multiplicities with the exponents: for example we will write $\{\!\!\{2^3,4,5^2\}\!\!\}$ for $\{\!\!\{2,2,2,4,5,5\}\!\!\}$.} $B_m := B_{m-1}' + \big\{\!\!\big\{ \lambda^k  \mid  f_D(\lambda)=m \big\}\!\!\big\}=B_{m-1}' + \big\{\!\!\big\{ \underbrace{\lambda, \lambda,\dots,\lambda}_{\text{$k$ times}} \mid  f_D(\lambda)=m \big\}\!\!\big\}$.
		\item Consider the multiset $X_m := \{\!\!\{a \in B_m \ | \ a < \sigma_{m-1}\}\!\!\}$. There are two cases: if $X_m \neq \varnothing$ then set $\sigma_m := \max(X_m)$, otherwise set $\sigma_m := \max(B_m)$.
		\item Remove an occurrence of $\sigma_m$ from the multiset $B_m$ to obtain $B_m' := B_m - \{\!\!\{\sigma_m\}\!\!\}$.
	\end{enumerate}
	It is easy to check that this algorithm is well defined: from the condition on $D$ of being a labelled Dyck path we have that $B_m\neq \varnothing$ for every $m\in [kn]$.\\
	Now for each label $\lambda \in [n]$ define its \emph{pmaj contribution} as
	\[
	p_\lambda(D) := \#\text{ of weak ascents before the leftmost occurrence of $\lambda$ in $\sigma_{\pmaj}(D)$}.
	\] 
	Finally, we can define the \emph{pmaj statistic} for $D$ by summing these contributions:
	\[
	\pmaj(D) := p_1(D)+p_2(D)+\dots +p_n(D).
	\]
\end{definition}

\begin{example}\label{exa:gen_pmaj}
	In Figure~\ref{fig:gen_pmaj} we record in a table the steps of the algorithm that computes the $\pmaj$ of the labelled Dyck path $D$ in Figure~\ref{fig:gen_exa_imgphi}: we have $\pmaj(D)=14$ (and $\area(D)=26$).
	\begin{figure}[ht]
		\centering
		\begin{tikzpicture}
			\node[scale=.8] at (0,0) {
				\begin{tikzpicture}
					\parkfunc{7}{3}{21,21,18,15,15,14,6}{2,5,1,3,6,4,7}{.7}
					\node[green!40!black] at ( -1.0,-.5) {$\sigma_{\pmaj}(D)$};
					\node[green!40!black] at ( 0.35,-.5) {5};
					\node[green!40!black] at ( 1.05,-.5) {2};
					\node[green!40!black] at ( 1.75,-.5) {5};
					\node[green!40!black] at ( 2.45,-.5) {2};
					\node[green!40!black] at ( 3.15,-.5) {1};
					\node[green!40!black] at ( 3.85,-.5) {5};
					\node[green!40!black] at ( 4.55,-.5) {3};
					\node[green!40!black] at ( 5.25,-.5) {2};
					\node[green!40!black] at ( 5.95,-.5) {1};
					\node[green!40!black] at ( 6.65,-.5) {6};
					\node[green!40!black] at ( 7.35,-.5) {4};
					\node[green!40!black] at ( 8.05,-.5) {3};
					\node[green!40!black] at ( 8.75,-.5) {1};
					\node[green!40!black] at ( 9.45,-.5) {6};
					\node[green!40!black] at (10.15,-.5) {4};
					\node[green!40!black] at (10.85,-.5) {3};
					\node[green!40!black] at (11.55,-.5) {7};
					\node[green!40!black] at (12.25,-.5) {6};
					\node[green!40!black] at (12.95,-.5) {4};
					\node[green!40!black] at (13.65,-.5) {7};
					\node[green!40!black] at (14.35,-.5) {7};
					\draw[very thick,blue] ( 1.4,-.2) -- ( 1.4,-.8);
					\draw[very thick,blue] ( 3.5,-.2) -- ( 3.5,-.8);
					\draw[very thick,blue] ( 6.3,-.2) -- ( 6.3,-.8);
					\draw[very thick,blue] ( 9.1,-.2) -- ( 9.1,-.8);
					\draw[very thick,blue] (11.2,-.2) -- (11.2,-.8);
					\draw[very thick,blue] (13.3,-.2) -- (13.3,-.8);
					\draw[very thick,blue] (14.0,-.2) -- (14.0,-.8);
					\node[blue] at (-1.0,-1.2) {$p_{\lambda}(D)$};
					\node[blue] at ( 0.35,-1.2) {0};
					\node[blue] at ( 1.05,-1.2) {0};
					\node[blue] at ( 3.15,-1.2) {1};
					\node[blue] at ( 4.55,-1.2) {2};
					\node[blue] at ( 6.65,-1.2) {3};
					\node[blue] at ( 7.35,-1.2) {3};
					\node[blue] at (11.55,-1.2) {5};
				\end{tikzpicture}
			};
		\end{tikzpicture}
		\caption{The diagram of $\phi_{7}^{(3)}(D)$ for $D$ in Figure~\ref{fig:labelled_Dyck}, its $\sigma_{\pmaj}(D)$ word (the bars indicate the weak ascents), and its $\pmaj$ contributions.}\label{fig:gen_exa_imgphi}
	\end{figure}
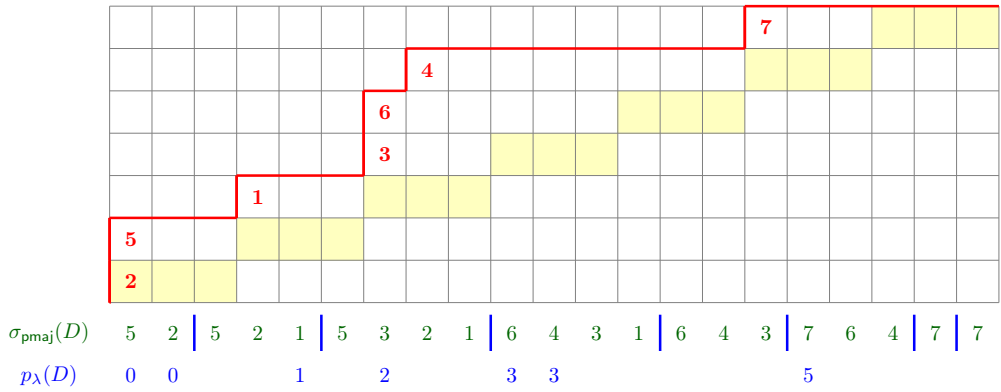
	\begin{figure}[ht]
		\centering
		\begin{tikzpicture}
			\node[scale=.75, below] at (-1,0) {
				\renewcommand{\arraystretch}{1.2}
				\begin{tabular}{l|l l l l|}
					$m$ & $B_m$ & $X_m$ & $\sigma_m$ & $B_m'$\\
					\hline
					1 &           $\{\!\!\{2^3,5^3\}\!\!\}$ & $\{\!\!\{2^3,5^3\}\!\!\}$ & 5 &           $\{\!\!\{2^3,5^2\}\!\!\}$\\
					2 &             $\{\!\!\{2^3,5^2\}\!\!\}$ &       $\{\!\!\{2^3\}\!\!\}$ & 2 &             $\{\!\!\{2^2,5^2\}\!\!\}$\\
					3 &               $\{\!\!\{2^2,5^2\}\!\!\}$ &     $\varnothing$ & 5 &               $\{\!\!\{2^2,5\}\!\!\}$\\
					4 &           $\{\!\!\{1^3,2^2,5\}\!\!\}$ &   $\{\!\!\{1^3,2^2\}\!\!\}$ & 2 &           $\{\!\!\{1^3,2,5\}\!\!\}$\\
					5 &             $\{\!\!\{1^3,2,5\}\!\!\}$ &       $\{\!\!\{1^3\}\!\!\}$ & 1 &             $\{\!\!\{1^2,2,5\}\!\!\}$\\
					6 &               $\{\!\!\{1^2,2,5\}\!\!\}$ &     $\varnothing$ & 5 &               $\{\!\!\{1^2,2\}\!\!\}$\\
					7 &     $\{\!\!\{1^2,2,3^3,6^3\}\!\!\}$ & $\{\!\!\{1^2,2,3^3\}\!\!\}$ & 3 &     $\{\!\!\{1^2,2,3^2,6^3\}\!\!\}$\\
					8 & $\{\!\!\{1^2,2,3^2,4^3,6^3\}\!\!\}$ &       $\{\!\!\{1^2,2\}\!\!\}$ & 2 & $\{\!\!\{1^2,3^2,4^3,6^3\}\!\!\}$\\
					9 &   $\{\!\!\{1^2,3^2,4^3,6^3\}\!\!\}$ &         $\{\!\!\{1^2\}\!\!\}$ & 1 &   $\{\!\!\{1,3^2,4^3,6^3\}\!\!\}$\\
					10 &     $\{\!\!\{1,3^2,4^3,6^3\}\!\!\}$ &     $\varnothing$ & 6 &     $\{\!\!\{1,3^2,4^3,6^2\}\!\!\}$\\
					11 &       $\{\!\!\{1,3^2,4^3,6^2\}\!\!\}$ & $\{\!\!\{1,3^2,4^3\}\!\!\}$ & 4 &       $\{\!\!\{1,3^2,4^2,6^2\}\!\!\}$\\
									\end{tabular}
			};
			\node[scale=.75, below] at (6.8,0) {
				\renewcommand{\arraystretch}{1.2}
				\begin{tabular}{l|l l l l|}
					$m$ & $B_m$ & $X_m$ & $\sigma_m$ & $B_m'$\\
					\hline
					12 &         $\{\!\!\{1,3^2,4^2,6^2\}\!\!\}$ &       $\{\!\!\{1,3^2\}\!\!\}$ & 3 &         $\{\!\!\{1,3,4^2,6^2\}\!\!\}$\\
					13 &           $\{\!\!\{1,3,4^2,6^2\}\!\!\}$ &           $\{\!\!\{1\}\!\!\}$ & 1 &           $\{\!\!\{3,4^2,6^2\}\!\!\}$\\
					14 &             $\{\!\!\{3,4^2,6^2\}\!\!\}$ &     $\varnothing$ & 6 &             $\{\!\!\{3,4^2,6\}\!\!\}$\\
					15 &               $\{\!\!\{3,4^2,6\}\!\!\}$ &       $\{\!\!\{3,4^2\}\!\!\}$ & 4 &               $\{\!\!\{3,4,6\}\!\!\}$\\
					16 &           $\{\!\!\{3,4,6,7^3\}\!\!\}$ &           $\{\!\!\{3\}\!\!\}$ & 3 &           $\{\!\!\{4,6,7^3\}\!\!\}$\\
					17 &             $\{\!\!\{4,6,7^3\}\!\!\}$ &     $\varnothing$ & 7 &             $\{\!\!\{4,6,7^2\}\!\!\}$\\
					18 &               $\{\!\!\{4,6,7^2\}\!\!\}$ &         $\{\!\!\{4,6\}\!\!\}$ & 6 &               $\{\!\!\{4,7^2\}\!\!\}$\\
					19 &                 $\{\!\!\{4,7^2\}\!\!\}$ &           $\{\!\!\{4\}\!\!\}$ & 4 &                 $\{\!\!\{7^2\}\!\!\}$\\
					20 &                   $\{\!\!\{7^2\}\!\!\}$ &     $\varnothing$ & 7 &                   $\{\!\!\{7\}\!\!\}$\\
					21 &                     $\{\!\!\{7\}\!\!\}$ &     $\varnothing$ & 7 &             $\varnothing$
				\end{tabular}
			};
		\end{tikzpicture}
		\caption{Computation of the $\pmaj$ statistic for the labelled Dyck path in Figure~\ref{fig:gen_exa_imgphi}.}\label{fig:gen_pmaj}
	\end{figure}
\end{example}

\section{The map $\phi_n^{(k)}$} \label{sec:phi_nk}

In this section we define a bijection $\phi_n^{(k)}:\PF_{n,kn}\to \PF_{n,kn}$, which turns out to be the Loehr-Remmel bijection in~\cite{LoehrRemmel_pmaj}. Our definition will be an extension of the definition of the map for $k=1$ given in~\cite{DAdderio_Sgubin_Expo}: while being an expository article about the original bijection of Loehr and Remmel, the alternative definition given in~\cite{DAdderio_Sgubin_Expo} of $\phi_n=\phi_n^{(1)}$ was new, and it inspired the extension that we are presenting in this section. Also, our arguments and notation about $\phi_n^{(k)}$ will follow closely~\cite{DAdderio_Sgubin_Expo}.

To define $\phi_n^{(k)}$ we need one more definition.

\begin{definition}
	Consider $D\in \LDyck_{n,kn}$ and let $\wdiag(D)=\sigma_1\sigma_2\dots \sigma_n$ be its diagonal reading word. We define the \emph{dinv contribution} of a label $\lambda\in [n]$ by
	\[
		d_\lambda(D):=\#\{\!\!\{(\sigma_i,\sigma_j)\in \Dinv(D)\mid \sigma_{\max(i,j)}=\lambda\}\!\!\}.
	\]
	Clearly we have $\sum_{\lambda\in [n]}d_\lambda(D)=\dinv(D)$.
\end{definition}
For example, for the labelled Dyck path $D$ in Figure~\ref{fig:labelled_Dyck}, $\wdiag(D)=2513467$, and we compute $d_2(D)=0$, $d_5(D)=3$, $d_1(D)=3$, $d_3(D)=3$, $d_4(D)=5$, $d_6(D)=9$, and $d_7(D)=3$: cf.\ Equation~\eqref{eq:gen_dinv_set_exa}, where each row shows the pairs contributing to a $\lambda\in [7]\setminus\{2\}$ ($2$ is the first letter of $\wdiag(D)$ hence necessarily $d_2(D)=0$).

\begin{definition}\label{def:gen_phi}
	Given $D\in \PF_{n,kn}$, define $\phi_n^{(k)}(D)$ as the unique $(n,kn)$-parking function in $\PF_{n,kn}$ such that 
	\[
		f_{\phi_n^{(k)}(D)}(\sigma_i)=k(i-1)-d_{\sigma_i}(D)+1\, ,
	\]
	where $\wdiag(D)=\sigma_1\sigma_2\dots \sigma_n$.
\end{definition}
This definition is better understood by an example.
\begin{example}
	Consider the labelled Dyck path $D$ in Figure~\ref{fig:labelled_Dyck} (here $n=7$ and $k=3$). We already computed $\wdiag(D)=2513467$, and its dinv contributions $d_2(D)=0$, $d_5(D)=3$, $d_1(D)=3$, $d_3(D)=3$, $d_4(D)=5$, $d_6(D)=9$, and $d_7(D)=3$. The definition is simply saying that label $2$ has to occur in $\phi_7^{(3)}(D)$ in column $3(1-1)-d_2(D)+1=1$, label $5$ in column $3(2-1)-d_5(D)+1=1$, label $1$ in column $3(3-1)-d_1(D)+1=4$, label $3$ in column $3(4-1)-d_3(D)+1=7$, label $4$ in column $3(5-1)-d_4(D)+1=8$, label $6$ in column $3(6-1)-d_6(D)+1=7$, and label $7$ in column $3(7-1)-d_7(D)+1=16$. The diagram of $\phi_7^{(3)}(D)$ is shown in Figure~\ref{fig:gen_exa_imgphi}.
\end{example}
One of the main results of this article is the following theorem. 
\begin{theorem}\label{thm:main_bijection}
	The correspondence $D\mapsto \phi_{n}^{(k)}(D)$ is a well-defined bijection of $\PF_{n,kn}$ into itself such that
	\begin{equation}\label{eq:dinvarea_areapmaj}
		\area(\phi_{n}^{(k)}(D))=\dinv(D)\quad \text{ and }\quad \pmaj(\phi_{n}^{(k)}(D))=\area(D).
	\end{equation}
\end{theorem}
For example, Equation~\eqref{eq:dinvarea_areapmaj} is easily checked for $D$ in Figure~\ref{fig:labelled_Dyck} (cf.\ Examples~\ref{ex:area},~\ref{ex:dinv} and~\ref{exa:gen_pmaj}).
\medskip

The proof of Theorem~\ref{thm:main_bijection} will keep us busy until the end of Section~\ref{sec:is_the_inverse}.

First of all we need to show that $\phi_n^{(k)}$ is well defined.
\begin{proposition}
	For every $n,k\in \mathbb{N}$ such that $n,k\geq 1$, $\phi_n^{(k)}$ is a well-defined function from $\PF_{n,kn}$ into itself.
\end{proposition}
\begin{proof}
	Let $\wdiag(D)=\sigma_1\sigma_2\dots \sigma_n$. Observe that by definition any pair $(i,j)$ occurs at most once in $\Tdinv(D)$ and at most $k-1$ times in $\Dinvcorr(D)$. Hence, for every $i\in [n]$ there are at most $k(i-1)$ pairs of the form $(\sigma_j,\sigma_i)$ with $j<i$ occurring in $\Dinv(D) = \Tdinv(D)+\Dinvcorr(D)$. Therefore
	\[
		0 \leq d_{\sigma_i}(D) \leq k(i-1),
	\]
	which implies
	\[
		1 \leq f_{\phi_n^{(k)}(D)}(\sigma_i) \leq 1+k(i-1).
	\]
	We deduce that
	\[
		\{\sigma_1,\sigma_2,\dots,\sigma_i\}\subseteq f_{\phi_n^{(k)}(D)}^{-1}([1+k(i-1)]),
	\]
	therefore
	\[
		\# f_{\phi_n^{(k)}(D)}^{-1}([1+k(i-1)])\geq i.
	\]
	By Remark~\ref{rem:PFnkn} we conclude that $\phi_n^{(k)}(D)\in \PF_{n,kn}$ exists and it is unique.
\end{proof}

There is another property that is easy to deduce from our definition of $\phi_n^{(k)}$.

\begin{proposition}
	For every $D \in \PF_{n,kn}$ we have
	\begin{equation}\label{eq:cl_phi_dinv-area}
		\dinv(D) = \area(\phi_n^{(k)}(D)).
	\end{equation}
\end{proposition}
\begin{proof}
	Given a parking function $D \in \PF_{n,kn}$, let $\wdiag(D)=\sigma_1\sigma_2\dots\sigma_n$.
	
	Observe that, for every $D\in \LDyck_{n,kn}$, $f_D(\lambda)$ is the number of the column of the diagram of $D$ containing the label $\lambda\in [n]$, we have
	\begin{equation}\label{eq:cl_intepr_phif}
		f_D(\lambda) = \#\left\{\begin{array}{c}
			\text{unit squares to the left of}\\
			\text{the north step labelled }\lambda
		\end{array}\right\} + 1 \qquad \text{for all }\lambda \in [n].
	\end{equation}
	The number of all the whole unit squares above the main diagonal in a diagram is $k\binom{n}{2}$, therefore by definition of the area statistic
	\begin{align*}
		\area(\phi_n^{(k)}(D)) &= k\binom{n}{2} - \sum_{i=1}^{n}\big(f_{\phi_n^{(k)}(D)}(\sigma_i)-1\big) = \sum_{i=1}^{n}k(i-1) - \sum_{i=1}^{n}(k(i-1) - d_{\sigma_i}(D))\\
		&= \sum_{i=1}^{n}d_{\sigma_i}(D) = \#\Dinv(D) = \dinv(D).
	\end{align*}
	This concludes the proof.
\end{proof}

\section{Definition of the inverse map $\psi_n^{(k)}$}\label{sec:inverse_psi_nk}

In this section we construct a map $\psi_n^{(k)}: \PF_{n,kn} \to \PF_{n,kn}$, whose definition is iterative, hence much more involved than the one of $\phi_n^{(k)}$. In later sections we will prove that $\psi_n^{(k)}$ is the inverse of $\phi_n^{(k)}$, and as a corollary we will deduce that $\area(D) = \pmaj(\phi_n^{(k)}(D))$ for all $D \in \PF_{n,kn}$. 

Since it is quite involved, we will illustrate it via an example.

\begin{example}\label{exa:cl_constr_psi}
	Consider the parking function $g \in \PF_{7,21}$ in Figure~\ref{fig:gen_exa_imgpsi}.
	First of all, we compute the word $\sigma_{\pmaj}(D)$, and for each label $\lambda$ we record its $\pmaj$ contribution $p_\lambda:=p_\lambda(D)$, together with its \emph{coarea contribution} $c_\lambda=c_\lambda(D):=f_D(\lambda)-1$, i.e.\ the number of squares to the left of $\lambda$ in the diagram of $D$: cf.\ the table on the right in Figure~\ref{fig:gen_exa_imgpsi}.
	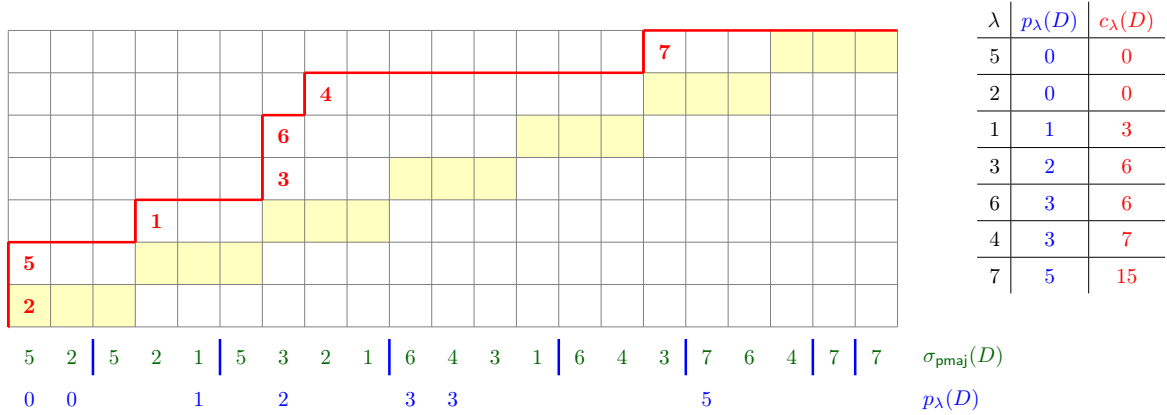
\begin{figure}[ht]
		\centering
		\begin{tikzpicture}
			\node[scale=.8] at (0,0) {
				\begin{tikzpicture}
					\parkfunc{7}{3}{21,21,18,15,15,14,6}{2,5,1,3,6,4,7}{.7}
					\node[green!40!black, right] at (15,-.5) {$\sigma_{\pmaj}(D)$};
					\node[green!40!black] at ( 0.35,-.5) {5};
					\node[green!40!black] at ( 1.05,-.5) {2};
					\node[green!40!black] at ( 1.75,-.5) {5};
					\node[green!40!black] at ( 2.45,-.5) {2};
					\node[green!40!black] at ( 3.15,-.5) {1};
					\node[green!40!black] at ( 3.85,-.5) {5};
					\node[green!40!black] at ( 4.55,-.5) {3};
					\node[green!40!black] at ( 5.25,-.5) {2};
					\node[green!40!black] at ( 5.95,-.5) {1};
					\node[green!40!black] at ( 6.65,-.5) {6};
					\node[green!40!black] at ( 7.35,-.5) {4};
					\node[green!40!black] at ( 8.05,-.5) {3};
					\node[green!40!black] at ( 8.75,-.5) {1};
					\node[green!40!black] at ( 9.45,-.5) {6};
					\node[green!40!black] at (10.15,-.5) {4};
					\node[green!40!black] at (10.85,-.5) {3};
					\node[green!40!black] at (11.55,-.5) {7};
					\node[green!40!black] at (12.25,-.5) {6};
					\node[green!40!black] at (12.95,-.5) {4};
					\node[green!40!black] at (13.65,-.5) {7};
					\node[green!40!black] at (14.35,-.5) {7};
					\draw[very thick,blue] ( 1.4,-.2) -- ( 1.4,-.8);
					\draw[very thick,blue] ( 3.5,-.2) -- ( 3.5,-.8);
					\draw[very thick,blue] ( 6.3,-.2) -- ( 6.3,-.8);
					\draw[very thick,blue] ( 9.1,-.2) -- ( 9.1,-.8);
					\draw[very thick,blue] (11.2,-.2) -- (11.2,-.8);
					\draw[very thick,blue] (13.3,-.2) -- (13.3,-.8);
					\draw[very thick,blue] (14.0,-.2) -- (14.0,-.8);
					\node[blue, right] at (15,-1.2) {$p_{\lambda}(D)$};
					\node[blue] at ( 0.35,-1.2) {0};
					\node[blue] at ( 1.05,-1.2) {0};
					\node[blue] at ( 3.15,-1.2) {1};
					\node[blue] at ( 4.55,-1.2) {2};
					\node[blue] at ( 6.65,-1.2) {3};
					\node[blue] at ( 7.35,-1.2) {3};
					\node[blue] at (11.55,-1.2) {5};
				\end{tikzpicture}
			};
		\node[right, scale=.8] at (6,1) {
			\renewcommand{\arraystretch}{1.4}
			\begin{tabular}{r|c|c}
				$\lambda$ & \textcolor{blue}{$p_\lambda(D)$} & \textcolor{red}{$c_\lambda(D)$} \\ \hline
				5 & \textcolor{blue}{0} & \textcolor{red}{0} \\ \hline
				2 & \textcolor{blue}{0} & \textcolor{red}{0} \\ \hline
				1 & \textcolor{blue}{1} & \textcolor{red}{3} \\ \hline
				3 & \textcolor{blue}{2} & \textcolor{red}{6} \\ \hline
				6 & \textcolor{blue}{3} & \textcolor{red}{6} \\ \hline
				4 & \textcolor{blue}{3} & \textcolor{red}{7} \\ \hline
				7 & \textcolor{blue}{5} &  \textcolor{red}{15}
			\end{tabular}
		};
		\end{tikzpicture}
		\caption{The $(7,21)$-parking function considered in Example~\ref{exa:gen_pmaj}, with its $\pmaj$ contributions and its co-area contributions.}\label{fig:gen_exa_imgpsi}
	\end{figure}\\
	The idea of the construction is to build up the labelled Dyck path $\psi_7^{(3)}(D)$ iteratively, by inserting the labels $[7]$ in the order they first occur in $\sigma_{\pmaj}(D) = \sigma_1\sigma_2 \dots \sigma_{21}$ reading it from left to right: say $i_1<i_2<\dots <i_7$ are the positions in $\sigma_{\pmaj}(D)$ of these occurrences.\\
	We start with the empty $0\times 0$ labelled Dyck path. For every $m = 1,2,\dots,7$, we will choose the end point of a step in the $(m-1)\times k(m-1)$ labelled Dyck path obtained at the previous iterative step, where we will insert a north step labelled $\sigma_{i_m}$ immediately followed by $k$ east steps, in such a way that
	\begin{itemize}
		\item we obtain a $m\times k m$ labelled Dyck path,
		\item the row-area contribution of $\sigma_{i_m}$ in $\psi_7^{(3)}(D)$ is equal to $p_{\sigma_{i_m}}(D)$, and
		\item the label $\sigma_{i_m}$ ``avoids'' to create exactly $c_{\sigma_{i_m}}(D)$ diagonal inversions; in other words, in the new $m\times km$ labelled Dyck path the label $\sigma_{i_m}$ does not create a diagonal inversion with exactly $c_{\sigma_{i_m}}(D)$ many labels (counted with multiplicities).
	\end{itemize}
	\medskip 
	In our example, this iterative construction consists of $7$ iterative steps, shown in Figure~\ref{fig:gen_exa_psi2}:
	 \begin{itemize}
	 	\item \underline{Step 1}. Just construct the $1\times 3$ labelled Dyck path where the only north step is labelled $\sigma_{i_1} =\sigma_1= 5$.
	 	\item \underline{Step 2}. Since label $\sigma_{i_2} =\sigma_2 = 2$ has $\pmaj$ contribution $p_2(D) = 0$, the corresponding north step in the labelled path will have $0$ area contribution. Hence, we consider the main diagonal, highlighted with a dotted line in the picture. Intersecting this diagonal with the labelled path we get $2$ intersection points shown in blue. Both are possible spots in which to add the label $\sigma_2 = 2$, but they would ``avoid'' a different number of diagonal inversions: in the figure, the number of ``avoided'' diagonal inversions is denoted in green (and it is computed by the formula $k(m-1)-d$ where $d$ is the blue sum ``{\color{blue}$a + b$}'' next to it: $\color{blue}a$ is the number of $\tdinv$ pairs the insertion in that spot would create, while $\color{blue}b$ is the number of $\dinvcorr$ pairs).\\ 	
	 	Since $c_2 = 0$, at the blue point labelled $0$ we insert a north step labelled $\sigma_2 = 2$ followed by $k=3$ east step. In this way we obtain the second labelled Dyck path in Figure~\ref{fig:gen_exa_psi2}.
		\item \underline{Steps $3,4$ and $5$}. These steps follow the same rule from Step 2, the only difference being that, instead of the main diagonal, since $\sigma_{i_3}=\sigma_5=1$, $\sigma_{i_4}=\sigma_7=3$, $\sigma_{i_5}=\sigma_{10}=6$, and $p_1(D)=1$, $p_3(D) = 2$, $p_6(D)=3$, we need to intersect the path with the lines $y = \frac{1}{3}(x + p_{1}(D))=\frac{1}{3}(x + 1)$, $y = \frac{1}{3}(x + p_{3}(D))=\frac{1}{3}(x + 2)$ and $y = \frac{1}{3}(x + p_{6}(D))=\frac{1}{3}(x + 3)$ respectively, and we need to pick the spots that avoid $c_1(D)=3$, $c_3(D)=6$ and $c_6(D)=6$ diagonal inversions, respectively.  Observe that at every step, reading the possible insertion spots from left to right, the ``avoided'' dinv increases by $1$ at each spot: this is an important property that we will prove in general.
		\item \underline{Step 6}. Now $\sigma_{i_6}=\sigma_{11} = 4$, $p_4(D)=3$ and $c_4(D)=7$. But observe that intersecting the path with the line $y = \frac{1}{3}(x + p_{4}(D))=\frac{1}{3}(x + 3)$ we get $5$ intersections, but one of them, denoted with a blue $\times$, is not available, since inserting a vertical step labelled $4$ right above the label $5$ would violate the condition of being a labelled Dyck path. So those intersection points should be ignored, and we have only $4$ possible spots. Of course we insert at the spot ``avoiding'' $7$ diagonal inversions.
		\item \underline{Step 7}. This is analogous to the previous one, leading to the labelled Dyck path appearing in Figure~\ref{fig:labelled_Dyck}: this will be our $\psi_7^{(3)}(D)$. 
	\end{itemize}
	
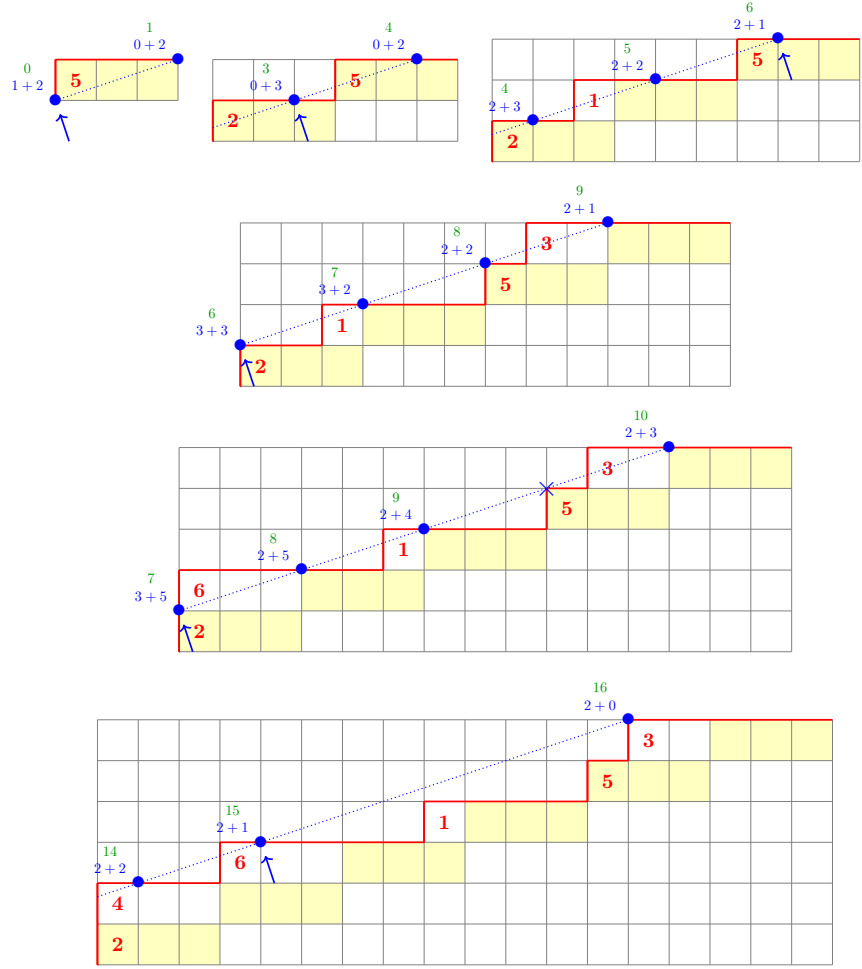
\begin{figure}[htbp!]
	\centering
	\resizebox{.9\width}{.9\height}{
		\begin{tikzpicture}
			\pgfmathsetmacro{\scaleDINV}{1}
			\node[scale=.6] at (-5.3,0.5) {\begin{tikzpicture}
					\parkfunc{1}{3}{3}{5}{1}
					\draw[blue, dotted, thick] (0,0) -- (3,1);
					\coordinate (C) at (0,0);
					\draw[blue, very thick, ->] ($(C) + (.3333,-1)$) -- ($(C) + (.1111,-.3333)$);
					\node[scale=2, blue] at (0,0) {\textbullet};
					\node[scale=\scaleDINV, above left] at (0,0) {$\begin{array}{c}{\color{green!60!black}0}\\ {\color{blue}1+2}\end{array}$};
					\node[scale=2, blue] at (3,1) {\textbullet};
					\node[scale=\scaleDINV, above left] at (3,1) {$\begin{array}{c}{\color{green!60!black}1}\\ {\color{blue}0+2}\end{array}$};
			\end{tikzpicture}};
			\node[scale=.6] at (-1.8,0.5) {\begin{tikzpicture}
					\parkfunc{2}{3}{6,3}{2,5}{1}
					\coordinate (C) at (2,1);
					\draw[blue, very thick, ->] ($(C) + (.3333,-1)$) -- ($(C) + (.1111,-.3333)$);
					\draw[blue, dotted, thick] (0,0.3333) -- (5,2);
					\node[scale=2, blue] at (2,1) {\textbullet};
					\node[scale=\scaleDINV, above left] at (2,1) {$\begin{array}{c}{\color{green!60!black}3}\\ {\color{blue}0+3}\end{array}$};
					\node[scale=2, blue] at (5,2) {\textbullet};
					\node[scale=\scaleDINV, above left] at (5,2) {$\begin{array}{c}{\color{green!60!black}4}\\ {\color{blue}0+2}\end{array}$};
			\end{tikzpicture}};
			\node[scale=.6] at (3.1,0.5) {\begin{tikzpicture}
					\parkfunc{3}{3}{9,7,3}{2,1,5}{1}
					\coordinate (C) at (7,3);
					\draw[blue, very thick, ->] ($(C) + (.3333,-1)$) -- ($(C) + (.1111,-.3333)$);
					\draw[blue, dotted, thick] (0,0.6666) -- (7,3);
					\node[scale=2, blue] at (1,1) {\textbullet};
					\node[scale=\scaleDINV, above left] at (1,1) {$\begin{array}{c}{\color{green!60!black}4}\\ {\color{blue}2+3}\end{array}$};
					\node[scale=2, blue] at (4,2) {\textbullet};
					\node[scale=\scaleDINV, above left] at (4,2) {$\begin{array}{c}{\color{green!60!black}5}\\ {\color{blue}2+2}\end{array}$};
					\node[scale=2, blue] at (7,3) {\textbullet};
					\node[scale=\scaleDINV, above left] at (7,3) {$\begin{array}{c}{\color{green!60!black}6}\\ {\color{blue}2+1}\end{array}$};
			\end{tikzpicture}};
			\node[scale=.6] at (0,-2.5) {\begin{tikzpicture}
					\parkfunc{4}{3}{12,10,6,5}{2,1,5,3}{1}
					\coordinate (C) at (0,1);
					\draw[blue, very thick, ->] ($(C) + (.3333,-1)$) -- ($(C) + (.1111,-.3333)$);
					\draw[blue, dotted, thick] (0,1) -- (9,4);
					\node[scale=2, blue] at (0,1) {\textbullet};
					\node[scale=\scaleDINV, above left] at (0,1) {$\begin{array}{c}{\color{green!60!black}6}\\ {\color{blue}3+3}\end{array}$};
					\node[scale=2, blue] at (3,2) {\textbullet};
					\node[scale=\scaleDINV, above left] at (3,2) {$\begin{array}{c}{\color{green!60!black}7}\\ {\color{blue}3+2}\end{array}$};
					\node[scale=2, blue] at (6,3) {\textbullet};
					\node[scale=\scaleDINV, above left] at (6,3) {$\begin{array}{c}{\color{green!60!black}8}\\ {\color{blue}2+2}\end{array}$};
					\node[scale=2, blue] at (9,4) {\textbullet};
					\node[scale=\scaleDINV, above left] at (9,4) {$\begin{array}{c}{\color{green!60!black}9}\\ {\color{blue}2+1}\end{array}$};
			\end{tikzpicture}};
			\node[scale=.6] at (0,-6.1) {\begin{tikzpicture}
					\parkfunc{5}{3}{15,15,10,6,5}{2,6,1,5,3}{1}
					\coordinate (C) at (0,1);
					\draw[blue, very thick, ->] ($(C) + (.3333,-1)$) -- ($(C) + (.1111,-.3333)$);
					\draw[blue, dotted, thick] (0,1) -- (12,5);
					\node[scale=2, blue] at (0,1) {\textbullet};
					\node[scale=\scaleDINV, above left] at (0,1) {$\begin{array}{c}{\color{green!60!black}7}\\ {\color{blue}3+5}\end{array}$};
					\node[scale=2, blue] at (3,2) {\textbullet};
					\node[scale=\scaleDINV, above left] at (3,2) {$\begin{array}{c}{\color{green!60!black}8}\\ {\color{blue}2+5}\end{array}$};
					\node[scale=2, blue] at (6,3) {\textbullet};
					\node[scale=\scaleDINV, above left] at (6,3) {$\begin{array}{c}{\color{green!60!black}9}\\ {\color{blue}2+4}\end{array}$};
					\node[scale=2, blue] at (9,4) {$\times$};
					\node[scale=2, blue] at (12,5) {\textbullet};
					\node[scale=\scaleDINV, above left] at (12,5) {$\begin{array}{c}{\color{green!60!black}10}\\ {\color{blue}2+3}\end{array}$};
			\end{tikzpicture}};
			\node[scale=.6] at (0,-10.4) {\begin{tikzpicture}
					\parkfunc{6}{3}{18,18,15,10,6,5}{2,4,6,1,5,3}{1}
					\coordinate (C) at (4,3);
					\draw[blue, very thick, ->] ($(C) + (.3333,-1)$) -- ($(C) + (.1111,-.3333)$);
					\draw[blue, dotted, thick] (0,1.6666) -- (13,6);
					\node[scale=2, blue] at (1,2) {\textbullet};
					\node[scale=\scaleDINV, above left] at (1,2) {$\begin{array}{c}{\color{green!60!black}14}\\ {\color{blue}2+2}\end{array}$};
					\node[scale=2, blue] at (4,3) {\textbullet};
					\node[scale=\scaleDINV, above left] at (4,3) {$\begin{array}{c}{\color{green!60!black}15}\\ {\color{blue}2+1}\end{array}$};
					\node[scale=2, blue] at (13,6) {\textbullet};
					\node[scale=\scaleDINV, above left] at (13,6) {$\begin{array}{c}{\color{green!60!black}16}\\ {\color{blue}2+0}\end{array}$};
			\end{tikzpicture}};
	\end{tikzpicture}}
	\caption{The step-by-step construction of $\psi_{7}^{(3)}(D)$ for $D$ in Figure~\ref{fig:gen_exa_imgpsi}. The final step leads to the $(7,21)$-parking function appearing in Figure~\ref{fig:labelled_Dyck}.}\label{fig:gen_exa_psi2}
\end{figure}
\end{example}

\section{$\psi_n^{(k)}$ is well defined}\label{sec:psi_welldefined}

First of all we need to show that $\psi_n^{(k)}:\PF_{n,kn}\to \PF_{n,kn}$ is a well-defined function.

\smallskip

Consider a $(n,kn)$-parking function $D \in \PF_{n,kn}$ and compute $\sigma_{\pmaj}(D) = \sigma_1\sigma_2 \dots\sigma_{kn}$, and let $i_1<i_2<\dots <i_n$ be the positions in $\sigma_{\pmaj}(D)$ of the first occurrences of the elements of $[n]$ when reading $\sigma_{\pmaj}(D)$ from left to right. Recall that for each label $\lambda \in [n]$, $c_\lambda=c_\lambda(D) := f_D(\lambda) - 1$ denotes the co-area contribution of the row with north step labelled $\lambda$ in $D$, and $p_\lambda = p_\lambda(D)$ denotes the pmaj contribution of the label $\lambda$ in $D$.

We want to perform our iterative procedure as in Example~\ref{exa:cl_constr_psi}, hence, starting with the empty labelled Dyck path, we will perform $n$ insertion Steps. At Step $m$, we will look at the intersection of the diagonal $y=\frac{1}{k}(x+p_{\sigma_m}(D))$ with the $(m-1)\times k(m-1)$ labelled Dyck path obtained at Step $m-1$, and insert a north step labelled $\sigma_{i_m}$ immediately followed by $k$ east steps in a \emph{suitable point} of that diagonal, i.e.\ an end point of a step of the $(m-1)\times k(m-1)$ labelled Dyck path such that the insertion performed in that point gives a $m\times km$ labelled Dyck path. Notice that in this way its row-area contribution equals $p_{\sigma_{i_m}}(D)$. Moreover, we want to insert $\sigma_{i_m}$ in such a way that it avoids exactly $c_{\sigma_{i_m}}(D)$ diagonal inversions.

In order to show that $\psi_n^{(k)}$ is well defined we need the following lemma.

\begin{lemma}\label{lem:well_def}
	For every $m=1,2,\dots,n$, at Step $m$ there exists a unique suitable insertion point on which the insertion of label $\sigma_{i_m}$ avoids exactly $c_{\sigma_{i_m}}$ diagonal inversions.
\end{lemma}

The rest of this section is dedicated to the proof of this fundamental result. We need a few definitions.

Given $D \in \PF_{n,kn}$, consider the permutation $\sigma_{\pmaj}(D) = \sigma_1\sigma_2
\sigma_{kn}$, and insert a bar at the weakly ascent positions of this word (i.e.\ between $\sigma_{j}$ and $\sigma_{j+1}$ if $\sigma_j\leq \sigma_{j+1}$). The subwords separated by the bars are called \emph{runs}, and we denote them $R_0,R_1,R_2,\dots$ in increasing order from left to right. 
\begin{remark}\label{rem:Rpsigmam}
	Observe that, by construction (see Definition~\ref{def:gen_pmaj}), for every $\lambda\in [n]$, the run $R_{p_\lambda(D)}$ is the leftmost one for which $\lambda\in R_{p_\lambda(D)}$. In other words, in the notation above, $\sigma_{i_m}\in R_{p_{\sigma_{i_m}}(D)}$ for every $m\in [n]$. Moreover, $\lambda\in R_h$ if and only if $0\leq h-p_\lambda(D)<k$.
\end{remark}
\begin{example}\label{ex:sigmamajDruns}
	For the labelled Dyck path $D$ in Example~\ref{exa:cl_constr_psi}, we reproduce here the word $\sigma_{\pmaj}(D)$, that appears right below the diagram of $D$ in Figure~\ref{fig:gen_exa_imgpsi}: 
	\begin{figure}[ht]
		\centering
		\begin{tikzpicture}
			\node[scale=.9] at (0,0) {\begin{tikzpicture}
				\node[green!40!black] at ( -1.0,-.5) {$\sigma_{\pmaj}(D)$};
				\node[green!40!black] at ( 0.35,-.5) {5};
				\node[green!40!black] at ( 1.05,-.5) {2};
				\node[green!40!black] at ( 1.75,-.5) {5};
				\node[green!40!black] at ( 2.45,-.5) {2};
				\node[green!40!black] at ( 3.15,-.5) {1};
				\node[green!40!black] at ( 3.85,-.5) {5};
				\node[green!40!black] at ( 4.55,-.5) {3};
				\node[green!40!black] at ( 5.25,-.5) {2};
				\node[green!40!black] at ( 5.95,-.5) {1};
				\node[green!40!black] at ( 6.65,-.5) {6};
				\node[green!40!black] at ( 7.35,-.5) {4};
				\node[green!40!black] at ( 8.05,-.5) {3};
				\node[green!40!black] at ( 8.75,-.5) {1};
				\node[green!40!black] at ( 9.45,-.5) {6};
				\node[green!40!black] at (10.15,-.5) {4};
				\node[green!40!black] at (10.85,-.5) {3};
				\node[green!40!black] at (11.55,-.5) {7};
				\node[green!40!black] at (12.25,-.5) {6};
				\node[green!40!black] at (12.95,-.5) {4};
				\node[green!40!black] at (13.65,-.5) {7};
				\node[green!40!black] at (14.35,-.5) {7};
				\draw[very thick,blue] ( 1.4,-.2) -- ( 1.4,-.8);
				\draw[very thick,blue] ( 3.5,-.2) -- ( 3.5,-.8);
				\draw[very thick,blue] ( 6.3,-.2) -- ( 6.3,-.8);
				\draw[very thick,blue] ( 9.1,-.2) -- ( 9.1,-.8);
				\draw[very thick,blue] (11.2,-.2) -- (11.2,-.8);
				\draw[very thick,blue] (13.3,-.2) -- (13.3,-.8);
				\draw[very thick,blue] (14.0,-.2) -- (14.0,-.8);
				\node[blue] at (-1.0,-1.2) {$p_{\lambda}(D)$};
				\node[blue] at ( 0.35,-1.2) {0};
				\node[blue] at ( 1.05,-1.2) {0};
				\node[blue] at ( 3.15,-1.2) {1};
				\node[blue] at ( 4.55,-1.2) {2};
				\node[blue] at ( 6.65,-1.2) {3};
				\node[blue] at ( 7.35,-1.2) {3};
				\node[blue] at (11.55,-1.2) {5};
			\end{tikzpicture}};
		\end{tikzpicture}
	\end{figure}

	The blue bars indicate the weak ascents of the word $\sigma_{\pmaj}(D)$, hence the runs are $R_0=5\: 2$, $R_1=5\: 2\: 1$, $R_2=5\: 3\: 2\: 1$, $R_3=6\: 4\: 3\: 1$, $R_4=6\: 4\: 3$, $R_5=7\: 6\: 4$, $R_6=7$, and $R_7=7$.
\end{example}
Now add the letter $\sigma_0 = 0$ at the beginning of the word $\sigma_{\pmaj}(D) = \sigma_1\sigma_2\dots\sigma_{kn}$, so that we now have a new ascent and a new run: $R_{-1} = 0$.  
\begin{remark}\label{rem:im}
	Observe that, by construction, for every $m\in [n]$ we have
	\[
		i_m=\sum_{j=-1}^{h-1}\# R_j +\#\{\lambda\in R_h\mid \lambda >\sigma_{i_m}\}.
	\]
\end{remark}

For every $m \in [n]$, if $\sigma_{i_m}$ occurs in $R_h$, then we set
\begin{equation}\label{eq:cl_def_um}
	u_m :=\#\{\lambda\in R_h\mid \lambda >\sigma_{i_m}\}+\#\{\lambda\in R_{h-1}\mid \lambda\leq \sigma_{i_m}\}.
\end{equation}

\begin{example}
	Consider $\sigma_\pmaj(D)$ from Example~\ref{ex:sigmamajDruns}. Then $n=7$,
	\[\begin{array}{lll}
		i_1=1 	\qquad	& 	\sigma_{i_1}=\sigma_1=5	\qquad   & 	u_1=1 \\ 
		i_2=2 		& 	\sigma_{i_2}=\sigma_2=2	   & 	u_2=2 \\ 
		i_3=5 		& 	\sigma_{i_3}=\sigma_5=1	   & 	u_3=2 \\ 
		i_4=7		& 	\sigma_{i_4}=\sigma_7=3    & 	u_4=3 \\ 
		i_5=10 		& 	\sigma_{i_5}=\sigma_{10}=6 & 	u_5=4 \\ 
		i_6=11 		& 	\sigma_{i_6}=\sigma_{11}=4 & 	u_6=4 \\ 
		i_7=16 		& 	\sigma_{i_7}=\sigma_{17}=7 & 	u_7=3 
	\end{array}\]
\end{example}

For every $m\in [n]$, $i_m$ and $u_m$ provide a range for $f_D(\sigma_{i_m})$.
\begin{lemma}\label{lem:cl_co_area}
	In the notation above, for all $m \in [n]$ we have
	\begin{equation}\label{eq:cl_dis_tech}
		i_m - u_m \leq c_{\sigma_{i_m}}(D) \leq i_m - 1\, ,
	\end{equation}
	equivalently
	\begin{equation*}
		i_m - u_m + 1 \leq f_D(\sigma_{i_m}) \leq i_m.
	\end{equation*}
\end{lemma}
\begin{proof}
	We prove the two inequalities separately. 
	
	To show that $i_m - u_m + 1 \leq f_D(\sigma_{i_m})$ there are two sub-cases.
	\begin{itemize}
		\item if $\sigma_{i_m}$ occurs in $R_0$, then $i_m=u_m = m$, and the inequality becomes $1 \leq f_D(\sigma_{i_m})$ which is always true;
		\item if $\sigma_{i_m}$ occurs in $R_h$ with $h > 0$, then suppose that $f_D(\sigma_{i_m}) \leq i_m - u_m$: by Remark~\ref{rem:im} in this case the algorithm in Definition~\ref{def:gen_pmaj} would have inserted the label $\sigma_{i_m}$ earlier in the construction of $\sigma_{\pmaj}(D)$, which is absurd.
	\end{itemize}
	
	To prove $f_D(\sigma_{i_m}) \leq i_m$, observe that if $f_D(\sigma_{i_m})>i_m$, then the algorithm in Definition~\ref{def:gen_pmaj} would insert $\sigma_{i_m}$ later in the construction of $\sigma_\pmaj(D)$, which is absurd. 
	
	This concludes the proof.
\end{proof}

We can now state a more precise lemma.
\begin{lemma}\label{lem:cl_avoid_dinv}
	Let $D \in \PF_{n,kn}$ and consider the iterative procedure to construct $\psi_n^{(k)}(D)$.
	For every $m \in [n]$ there are exactly $u_m$ suitable insertion points on the labelled Dyck path of size $(m-1)\times k(m-1)$, and the insertion procedure of $\sigma_{i_m}$ would avoid respectively $i_m - u_m$, $i_m - u_{m} + 1$, $\dots$, $i_m-1$ diagonal inversions, reading these points from left to right.
\end{lemma}
\begin{proof}
	The key observation is that the suitable points along the diagonal $y=\frac{1}{k}(x+p_{\sigma_{i_m}})$ are precisely the points where our insertion procedure would sit $\sigma_{i_m}$ on top of a smaller label on the diagonal $y=\frac{1}{k}(x+p_{\sigma_{i_m}}-1)$, or right after\footnote{In the case $p_{\sigma_{i_m}} = 0$ it might also be at the beginning of the path: this will be counted by the run $R_{-1} = 0$.} a (necessarily bigger) label on the diagonal $y=\frac{1}{k}(x+p_{\sigma_{i_m}})$. Since by Remark~\ref{rem:Rpsigmam} we have $\sigma_{i_m}\in R_{p_{\sigma_{i_m}}}$, this gives precisely $u_m$ suitable points. 
	
	First we want to show that inserting $\sigma_{i_m}$ in the rightmost suitable point avoids $i_m-1$ diagonal inversions.
	
	To do so, observe that $i_m-1$ is the number of letters to the left of $\sigma_{i_m}$ in $\sigma_\pmaj(D)$, hence if for every $j\in [m-1]$ we denote by $\gamma_m(\sigma_{i_j})$ the number of occurrences of the letter $\sigma_{i_j}$ to the left of $\sigma_{i_m}$ in $\sigma_\pmaj(D)$, then clearly
	\[i_m-1=\sum_{j=1}^{m-1}\gamma_m(\sigma_{i_j}).\]
	
	Therefore, to prove our claim, it suffices to show that inserting $\sigma_{i_m}$ in the rightmost suitable point, for every $j\in [m-1]$ there are exactly $k-\gamma_m(\sigma_{i_j})$ diagonal inversions of the form $(a,b)$ in $\Dinv(D)$ with $\{a,b\}=\{\sigma_{i_j},\sigma_m\}$. This can be done with a case-by-case analysis, which can be carried over with the help of the tables in the Appendix~\ref{app:dinvtable}: we leave this straightforward, though tedious, task to the reader.
	
	A similar case-by-case analysis, that we also omit, shows that going from a suitable insertion point to the one immediately to its left increases the dinv precisely by $1$, and this proves the last statement, and hence the lemma.
\end{proof}

We can finally prove our first lemma.
\begin{proof}[Proof of Lemma~\ref{lem:well_def}]
	It follows immediately by combining Lemma~\ref{lem:cl_co_area} and Lemma~\ref{lem:cl_avoid_dinv}.
\end{proof}

\section{\texorpdfstring{$\psi_n^{(k)}$}{} is the inverse of \texorpdfstring{$\phi_n^{(k)}$}{}} \label{sec:is_the_inverse}

In order to prove that $\psi_n^{(k)}$ is the inverse of $\phi_n^{(k)}$, the following remark is crucial.
\begin{remark}\label{rem:dlambda}
	Given $D\in \PF_{n,kn}$, in our iterative procedure defining $\psi_n^{(k)}(D)$, we claim that the insertions of $\sigma_{m'}$ for $m'>m$ do not affect the number of labels $\sigma_j$ with $j<m$ which $\sigma_m$ creates a diagonal inversion with. This is clear, since the insertion of $\sigma_{m'}$ does not change their relative positions for the temporary diagonal inversions, i.e.\ conditions (A) and (B) in Definition~\ref{def:gen_tdinv} (cf.\ Figure~\ref{fig:def_tdinv}), nor for the diagonal inversion corrections of Definition~\ref{def:gen_dinvcorr_set} (cf.\ Figure~\ref{fig:gen_dinvcompA}, but also Remark~\ref{rmk:east_step_labelling}).
\end{remark}

\begin{theorem}
	The function $\psi_n^{(k)}$ is the inverse of $\phi_n^{(k)}$.
\end{theorem}
\begin{proof}
	Since these are maps from the finite set $\PF_{n,kn}$ into itself, it is enough to prove that $\phi_n^{(k)}\circ \psi_n^{(k)}$ is the identity function of $\PF_{n,kn}$.\\
	Fix $D \in \PF_{n,kn}$. We want to show that $f_{\phi_n^{(k)} \circ \psi_n^{(k)}(D)}(\lambda) = f_D(\lambda)$ for any given label $\lambda \in [n]$.\\
	Consider the word $\sigma_{\pmaj}(D) = w_1w_2\dots w_n$ of $D$. Set $D' = \psi_n^{(k)}(D) \in \PF_{n,kn}$, and consider the reading word $\wdiag(D') = v_1v_2\dots v_n$ of $D'$.\\
	By construction of $D' = \psi_n^{(k)}(D)$, for every $m \in [n]$ we have
	\begin{equation}\label{eq:cl_pres_area}
		p_{w_m}(D) = a_{w_m}(D').
	\end{equation}
	Fix a label $\lambda = w_m = v_j$ for some $m,j \in [n]$. Using Remark~\ref{rem:dlambda}, by construction of $\psi_n^{(k)}$ we have:
	\begin{equation*}
		f_{D}(\lambda) - 1 = \#\text{ of avoided diagonal inversions between $\lambda = w_m$ and $w_1,\dots,w_{m-1}$ in $f$}.
	\end{equation*}
	On the other hand, by the definition of $\phi_n^{(k)}$ (cf. Definition~\ref{def:gen_phi}):
	\begin{align*}
		f_{\phi_{n}^{(k)}\circ \psi_{n}^{(k)}(D)}(\lambda) - 1 &= f_{(\phi_{n}^{(k)}(D'))}(\lambda) - 1\\
		&= \#\text{ avoided diagonal inversions between $\lambda = v_j$ and $v_1,\dots,v_{j-1}$ in $f$}.
	\end{align*}
	Therefore, the following result would prove the equality $f_D = f_{\phi_{n}^{(k)}\circ \psi_{n}^{(k)}(D)}$: if a label $\mu$ precedes $\lambda$ in one of the two words but not in the other, the pair $\lambda$ then $\mu$ avoids no diagonal inversion, i.e.~they contribute exactly $k$ diagonal inversions.\\
	By definition, the letters in $\sigma_{\pmaj}(D)$ are ordered increasingly with respect to their $\pmaj$ contribute in $D$ and the letters in $\wdiag(D')$ with respect to the $\area$ contribute in $D'$.\\
	Thus, Equation~\ref{eq:cl_pres_area} implies that two labels $\mu$ and $\lambda$ that swap relative position in $\sigma_\pmaj(D)$ and $\wdiag(D')$ must satisfy:
	\[
		a_\mu(D') = p_\mu(D) = p_\lambda(D) = a_\lambda(D').
	\]
	We have two subcases:
	\begin{itemize}
		\item \ul{If $i' < i$ and $j' > j$}: thus $\mu$ precedes $\lambda$ in $\sigma_\pmaj(D)$. Since $\mu$ and $\lambda$ contribute the same to $\pmaj$, the inequality $\mu > \lambda$ must hold. Conversely, since $\lambda$ precedes $\mu$ in $\wdiag(D')$ but they contribute the same $\area$, $f_D(\lambda) < f_D(\mu)$.\\
		By looking at Tables~\ref{tab:poss_tdinvs} and~\ref{tab:poss_dinvcorrs} from Appendix~\ref{app:dinvtable}, the pair $\{\mu,\lambda\}$ falls into case $(A1)$ with $f_D(\lambda) < f_D(\mu)$.\\
		Therefore they contribute $1\ \tdinv$ and $k-1\ \dinvcorr$.
		\item \ul{If $i' > i$ and $j' < j$} by an analogous argument we obtain that $\lambda > \mu$. By looking at the same tables, the pair $\{\lambda,\mu\}$ falls into case $(A2)$ with $f(\lambda) > f(\mu)$.\\
		Thus, these labels contribute $1\ \tdinv$ and $k-1\ \dinvcorr$.
	\end{itemize}
	We have proven the claim in both instances. The maps $\phi_{n}^{(k)}$ and $\psi_{n}^{(k)}$ are inverse of one another.
\end{proof}

\begin{corollary}
	For every $D\in \PF_{n,kn}$ we have that 
	\[
		\pmaj(\phi_{n}^{(k)}(D))=\area(D).
	\]
\end{corollary}
\begin{proof}
	For every $D'\in \PF_{n,kn}$, by the very construction of $\psi_{n}^{(k)}(D')$ we clearly have $\area(\psi_{n}^{(k)}(D'))=\pmaj(D')$. Applying this to $D'=\phi_{n}^{(k)}(D)$ and using the fact that $\phi_{n}^{(k)}$ and $\psi_{n}^{(k)}$ are inverse of each other, we conclude.
\end{proof}
The above corollary (finally) completes the proof of Theorem~\ref{thm:main_bijection}.

\section{Applications to symmetric functions}\label{sec:SFapplications}

To state our main corollary, we need our last definitions.
\begin{definition}
	Given $D\in \LDyck_{n,kn}(D)$, we define its \emph{row word} as the word $\wrow(D)$ obtained by reading the labels in the diagram of $D$ by rows, from bottom to top.
\end{definition}
For example, the labelled Dyck path $D$ in Figure~\ref{fig:gen_exa_imgphi} has $\wrow(D)=2513647$.

\smallskip

Consider two compositions $\mu, \nu$ such that $|\mu| + |\nu| = n$. They parametrize some shuffle products defined in~\cite{DDILLV} as follows.\\
Construct the sets $K_{\mu_1} = \{n, n-1, \dots, n-\mu_1+1\}$, $K_{\mu_2} = \{n-\mu_1, n-\mu_1-1, \dots, n-\mu_1-\mu_2+1\}$, $\dots$, and the sets $I_{\nu_1} = \{1,2,\dots,\nu_1\}$, $I_{\nu_2} = \{\nu_1+1, \nu_1+2, \dots, \nu_1+\nu_2\}$ and so on.\\
Let $\uparrow\!\! K_{\mu_i}$ be the word obtained by ordering the elements of the set increasingly and $\downarrow\!\! I_{\nu_j}$ in decreasing order, so for example $\uparrow\!\! K_{\mu_1} = (n-\mu_1+1)\dots(n-1)n$ and $\downarrow\!\! I_{\nu_1} = \nu_1(\nu_1-1)\dots21$. Then, we define the \emph{shuffle product} of $\mu$ and $\nu$ as:
\[
	W(\mu;\nu) := \uparrow\!\! K_{\mu_1}\ \shuffle\ \uparrow\!\! K_{\mu_2}\ \shuffle\ \dots\ \downarrow\!\! I_{\nu_1}\ \shuffle\ \downarrow\!\! I_{\nu_2}\ \shuffle\ \dots\ .
\]
This allows us to define some subsets of $\PF_{n,kn}$ parametrized by the two compositions:
\[
	\PF_{n,kn}(\mu;\nu) := \{f_D \in \PF_{n,kn}\ | \ \wrow(D) \in W(\mu;\nu)\}.
\]

\smallskip

The following version of the shuffle theorem for $\nabla^k e_n$ can be deduced from Theorem~\ref{thm:main_bijection} and Theorem~\ref{thm:rationalMellit} with little more work. See Section~\ref{sec:dinv} to recall the few missing definitions.
\begin{corollary}\label{cor:identitypmaj}
	We have	
	\[
		\nabla^k e_n=\sum_{D\in \PF_{n,kn}}q^{\area(D)}t^{\pmaj(D)}L_{n,\Des(\wrow(D)^{-1})}.
	\]
\end{corollary}

The proof of this result relies on the following observation.
\begin{lemma}\label{lem:descents}
	Consider $D \in \PF_{n,kn}$ and its image $D' := \phi_{n}^{(k)}(D)$. Then a label $\lambda \in [n-1]$ precedes (respectively appears later than) $\lambda+1$ in $\wdiag(D)$ if and only if the same holds in $\wrow(D')$.\\
	Therefore, $\Des(\wdiag(D)^{-1}) = \Des(\wrow(D')^{-1})$.
\end{lemma}
\begin{proof}
	Consider $\lambda \in [n-1]$ such that $\lambda$ appears before $\lambda+1$ in $\wdiag(D) = w_1\dots w_n$. In other words, $\lambda = w_i$ and $\lambda+1 = w_j$ for $1 \leq i < j \leq n$.\\
	We claim that the number of avoided diagonal inversions in $D$ between $\lambda+1 = w_j$ and $\mu := w_l$ for any $1 \leq l < i$ is greater than or equal to the avoided diagonal inversions between $\lambda=w_i$ and $\mu=w_l$. By construction, this would mean that in $D'$ the north step labelled $\lambda$ appears to the west of the step labelled $\lambda+1$. Equivalently, $\lambda$ appears before $\lambda+1$ in $\wrow(D')$.\\
	To prove the claim, we look separately at the avoided temporary diagonal inversions and diagonal inversion corrections.
	\begin{itemize}
		\item \ul{Temporary diagonal inversions}: the contribute of each pair of labels to this statistic can only be $0$ or $1$ by definition. Thus, it suffices to prove that if $\mu$ and $\lambda+1$ contribute to $\tdinv$ then $\mu$ and $\lambda$ do it too.\\
		Suppose that the pair $\mu$, $\lambda+1$ creates $\tdinv$. Then their relative position is described by one of the following cases in Appendix~\ref{app:dinvtable}, Table~\ref{tab:poss_tdinvs}: (A2) with\footnote{The additional condition is given by the fact that $\mu$ appears before $\lambda + 1$ in $\wdiag(D)$ under the assumptions of the claim.} $f_D(\lambda) > f_D(\mu)$, (B2) or (C2).\\
        Since $\lambda$ and $\lambda + 1$ are consecutive labels, this implies that also $\mu < \lambda$ must hold. Moreover, the fact that $\mu$ appears before $\lambda$ and $\lambda+1$ in $\wdiag(D)$ means that the labels $\mu$, $\lambda$ must also satisfy one of the cases: (A2) with $f_D(\lambda) > f_D(\mu)$, (B2) or (C2).\\
        This always implies that $\mu$ and $\lambda$ contribute to $\tdinv$.
		\item \ul{Diagonal inversion corrections} under the assumption that the labels in $\wdiag(D)$ appear in the order $\mu$, $\lambda$, $\lambda+1$, we get that $m_1 := a_{\lambda}(D) - a_{\mu}(D) \geq 0$, $m_2 := a_{\lambda + 1}(D) - a_{\lambda}(D)\geq 0$ and $a_{\lambda + 1}(D) - a_{\mu}(D) = m_1 + m_2$.\\
        By looking at each case in Appendix~\ref{app:dinvtable}, Table~\ref{tab:poss_dinvcorrs}, we observe that the $\dinvcorr$ contributes do not depend on the label values but only on the relative position of the north steps. Moreover, the $\dinvcorr$ contribute decreases when the value of $m$ increases.\\
        Thus, since $0 \leq m_1 \leq m_1 + m_2$, the only case in which $(\mu,\lambda + 1)$ creates more diagonal inversion corrections is when $m_2 = 0$ and $f_D(\lambda + 1) < f_D(\mu) < f_D(\lambda)$. But this would imply that $\lambda + 1$ appears before $\lambda$ in $\wdiag(D)$, contradicting our hypothesis. 
	\end{itemize}
	Since the claim holds for the partial statistics, it holds for their sum $\dinv$.
	\smallskip\\
	Notice that $\lambda$ and $\lambda+1$ can be exchanged in the claim and its proof and they would still hold. In fact, we never compare the values of $\lambda$ and $\lambda+1$ themselves, but only the their value with some other label $\mu$ (and in such case, $\mu < \lambda$ if and only if $\mu < \lambda+1$ and $\lambda < \mu$ if and only if $\lambda+1 < \mu$).
	\smallskip\\
	This shows the ``only if'' in the first statement. By contraposition, the ``if'' statement is also true.
	\medskip\\
	Finally, the second statement directly follows from the first one, since the descents of a word are exactly the labels $\lambda \in [n-1]$ such that $\lambda+1$ appears before $\lambda$.
\end{proof}

\begin{proof}[Proof of Corollary~\ref{cor:identitypmaj}]
	Using Theorem~\ref{thm:rationalMellit}, it is enough to prove that:
	\[
		\sum_{D\in \PF_{n,kn}}q^{\dinv(D)}t^{\area(D)}L_{n,\Des(\wdiag(D)^{-1})} = \sum_{D'\in \PF_{n,kn}}q^{\area(D')}t^{\pmaj(D')}L_{n,\Des(\wrow(D')^{-1})}.
	\]
Associating each summand via $\phi_{n}^k(D) = D'$, the result follows because the statistics are preserved by Theorem~\ref{thm:main_bijection} and the descents sets coincide by Lemma~\ref{lem:descents}.
\end{proof}
From Corollary~\ref{cor:identitypmaj} and the usual \emph{superization} (cf.\ \cite[Chapter~6]{Haglund-Book-2008}), we immediately deduce the following corollary. In the following, $\langle-,-\rangle$ is the Hall scalar product of symmetric functions, $e_\mu=e_{\mu_1}e_{\mu_2}\dots$ are the elementary symmetric functions, and $h_\nu=h_{\nu_1}h_{\nu_2}\dots$ are the (complete) homogeneous symmetric functions.
\begin{corollary}\label{cor:nablaken_eh_pmaj}
	Consider $n,k \in \mathbb{N} \setminus \{0\}$ and two compositions $\mu,\nu$ such that $|\mu| + |\nu| = n$. Then:
	\[
	\langle \nabla^k e_n, e_\mu h_\nu \rangle = \sum_{D \in \PF_{n,kn}(\mu;\nu)}q^{\area(D)}t^{\pmaj(D)}.
	\]
\end{corollary}

\section{The sandpile model}\label{sec:sandpile}

In the second part of this work we extend the study of~\cite{DDILLV} on the sandpile model on \emph{simple} graphs by allowing multiple edges between the same vertices.

\begin{definition}
    Let $G = (V,E)$ be a finite loop-free undirected graph where the \emph{vertex set} is $V = \{0,1,2,\dots,n\}$ and $E$ is the \emph{edge multiset}, allowing for multiple edges between the same two vertices.
\end{definition}

A \emph{configuration} on the graph $G$ is a map $c: V \to \Z$ that assigns to every vertex a number of ``grains of sand''.\\
For any vertex $v \in V$ we say it is \emph{non-negative} if $c(v) \geq 0$. If $c(v) < \deg_G(v)$ we say that it is \emph{stable}, otherwise it is called \emph{unstable}.

Any vertex $v \in V$ can be \emph{toppled} (or \emph{fired}): intuitively, it ``donates'' its grains of sand to its neighbors, accordingly to the multiplicity of edges. Let $\epsilon(v,w)$ be the multiplicity of the edge $\{v,w\}$ in $E$. The result of a toppling of $v$ on a configuration $c$ is a new configuration $c'$ such that $c'(v) = c(v) - \deg_G(v)$ and for all $w \in \tilde{V}$ with $w \neq v$ we have $c'(w) = c(w) + \epsilon(v,w)$.\\
To keep track of the toppled vertex, we denote the new configuration $\phi_v(c) = c'$.

Let $0 \in V$ always be a distinguished vertex that we will call \emph{sink}. The others will be the \emph{non-sink} vertices and their set will be denoted $\tilde{V} := V \setminus \{1,2,\dots,n\}$.\\
We say that a configuration $c$ is \emph{non-negative} if its non-sink vertices are non-negative. Likewise, a configuration is \emph{stable} when its non-sink vertices are stable.
\begin{remark}
    Since we do not care about the value of the sink $0$ in a configuration, we can ignore it. Thus, from now on we will consider configurations restricted on the non-sink vertices of $G$.
\end{remark}

\definecolor{mycolor1}{RGB}{255,0,0}
\definecolor{mycolor2}{RGB}{120,200,30}
\definecolor{mycolor3}{RGB}{0,180,180}
\definecolor{mycolor4}{RGB}{128,0,255}

\begin{example}\label{exa:toppling}
    Consider the graph $G$ in Figure~\ref{fig:graph_fam} where the vertices are labelled $\{0\} \cup [13]$, the sink is $0$ and the multiplicities are shown using different line styles: continuous, dashed and dotted lines indicate respectively multiplicities $2$, $1$ and $1$ (c.f.~we follow the general convention introduced in Example~\ref{exa:conv_lines}).
    \begin{figure}[ht]
        \centering
        \begin{tikzpicture}
            \sandpile{2}{2}{4,1}{2,6}{{\color{mycolor1}$v_4^{\mu_1}$},{\color{mycolor1}$v_3^{\mu_1}$},{\color{mycolor1}$v_2^{\mu_1}$},{\color{mycolor1}$v_1^{\mu_1}$},{\color{mycolor2}$v_1^{\mu_2}$},{\color{mycolor3}$v_2^{\nu_2}$},{\color{mycolor3}$v_1^{\nu_2}$},{\color{mycolor4}$v_6^{\nu_1}$},{\color{mycolor4}$v_5^{\nu_1}$},{\color{mycolor4}$v_4^{\nu_1}$},{\color{mycolor4}$v_3^{\nu_1}$},{\color{mycolor4}$v_2^{\nu_1}$},{\color{mycolor4}$v_1^{\nu_1}$}}{13}{2}
        \end{tikzpicture}
        \caption{Representation of the graph $G_{\mu,\nu}^{(k)}$ where $\mu = ({\color{mycolor1}4},{\color{mycolor2}1})$, $\nu = ({\color{mycolor4}6},{\color{mycolor3}2})$ and $k = 3$.}\label{fig:graph_fam}
    \end{figure}\\
    Throughout, we will write down configurations on $G$ by listing the values clockwise: $c(13), c(12), \dots, c(1)$.\\
    For example, a non-negative stable configuration is $c = {\color{mycolor1} 6}\ {\color{mycolor1} 17}\ {\color{mycolor1} 18}\ {\color{mycolor1} 24}\ {\color{mycolor2} 2}\ {\color{mycolor3} 17}\ {\color{mycolor3} 14}\ {\color{mycolor4} 18}\ {\color{mycolor4} 18}\ {\color{mycolor4} 15}\ {\color{mycolor4} 12}\ {\color{mycolor4} 7}\ {\color{mycolor4} 2}$.
    Here, we list some topplings applied to the configuration $c$ of $G$:
    \begin{align*}
                       \phi_0(c) &= {\color{mycolor1} 7}\ {\color{mycolor1} 18}\ {\color{mycolor1} 19}\ {\color{mycolor1} 25}\ {\color{mycolor2} 3}\ {\color{mycolor3} 18}\ {\color{mycolor3} 15}\ {\color{mycolor4} 19}\ {\color{mycolor4} 19}\ {\color{mycolor4} 16}\ {\color{mycolor4} 13}\ {\color{mycolor4} 8}\ {\color{mycolor4} 3},\\
              \phi_{10}\phi_0(c) &= {\color{mycolor1} 9}\ {\color{mycolor1} 20}\ {\color{mycolor1} 21}\ {\color{mycolor1} 0}\ {\color{mycolor2} 5}\ {\color{mycolor3} 20}\ {\color{mycolor3} 17}\ {\color{mycolor4} 21}\ {\color{mycolor4} 21}\ {\color{mycolor4} 18}\ {\color{mycolor4} 15}\ {\color{mycolor4} 10}\ {\color{mycolor4} 5},\\
        \phi_6\phi_{10}\phi_0(c) &= {\color{mycolor1} 11}\ {\color{mycolor1} 22}\ {\color{mycolor1} 23}\ {\color{mycolor1} 2}\ {\color{mycolor2} 7}\ {\color{mycolor3} 22}\ {\color{mycolor3} 19}\ {\color{mycolor4} 1}\ {\color{mycolor4} 20}\ {\color{mycolor4} 19}\ {\color{mycolor4} 16}\ {\color{mycolor4} 11}\ {\color{mycolor4} 6}.
    \end{align*}
\end{example}
As in~\cite{DDILLV}, we will use the following (not standard) definition of recurrent configuration.
\begin{definition}\label{def:rec_conf}
    A configuration $c$ on $G$ is \emph{recurrent} if it is non-negative, stable and if there exists a visiting order of all vertices, starting at the sink $0$, such that applying consequently topplings on vertices in said order gives a sequence of non-negative configurations.\\
    In other words, there exists an indexing $\sigma_0=0, \sigma_1, \sigma_2, \dots, \sigma_n$ of all vertices of $G$ such that
    \[
            c,\;\; \phi_0(c),\;\; \phi_{\sigma_1}\phi_{0}(c),\;\; \phi_{\sigma_2}\phi_{\sigma_1}\phi_{0}(c),\;\; \dots,\;\; \phi_{\sigma_n}\dots\phi_{\sigma_1}\phi_{0}(c) = c \text{ are non-negative.}
    \]
    Observe that toppling every vertex once gives back the original configuration, since through each multiedge the same ``number of grains'' will travel in both directions.\\
    Finally, denote $\Rec(G)$ the set of all recurrent configurations of $G$.
\end{definition}

\begin{example}\label{exa:rec_conf}
    The configuration $c = {\color{mycolor1} 6}\ {\color{mycolor1} 17}\ {\color{mycolor1} 18}\ {\color{mycolor1} 24}\ {\color{mycolor2} 2}\ {\color{mycolor3} 17}\ {\color{mycolor3} 14}\ {\color{mycolor4} 18}\ {\color{mycolor4} 18}\ {\color{mycolor4} 15}\ {\color{mycolor4} 12}\ {\color{mycolor4} 7}\ {\color{mycolor4} 2}$ from Example~\ref{exa:toppling} is recurrent. The reader can check that the reading order $\sigma = 10\ 6\ 5\ 4\ 11\ 12\ 8\ 7\ 3\ 13\ 9\ 1$ works by continuing the sequence of topplings shown in the previous example and observing that at each step the resulting configuration is non-negative.
\end{example}

\begin{remark}\label{rem:reconf}
    Let $c$ be a recurrent configuration. Topple the sink and then $m < n$ distinct non-sink vertices. The resulting configuration will have an unstable vertex between those which were not toppled.
\end{remark}

From now on fix two compositions $\mu = (\mu_1,\mu_2,\dots,\mu_{\ell(\mu)})$ and $\nu = (\nu_1,\nu_2,\dots,\nu_{\ell(\nu)})$ be such that $n = |\mu| + |\nu|$ and a positive integer $k \in \N \setminus \{0\}$.\\
We will study recurrent configurations on a particular family of graphs indexed on these parameters.

\begin{definition}\label{def:graph_fam}
    We construct the graph $G_{\mu,\nu}^{(k)}$ by taking:
    \begin{itemize}
        \item $\ell(\mu)$ \emph{$k$-clique components}, which are distinct complete multigraphs $K_1,\dots,K_{\ell(\mu)}$ on respectively $\mu_1,\dots,\mu_{\ell(\mu)}$ vertices and edges of multiplicity $k$.
        \item $\ell(\nu)$ \emph{$(k-1)$-clique components}, which are distinct complete multigraphs $H_1,\dots,H_{\ell(\nu)}$ on respectively $\nu_1,\dots,\nu_{\ell(\nu)}$ vertices and edges of multiplicity $k-1$.
    \end{itemize}
    Then, connect with an edge of multiplicity $k$ every pair of vertices in different components.\\
    The vertices are indexed on $\{1,\dots,n\}$, numbering components in the order $H_1$, $H_2$, $\dots$, $H_{\ell(\nu)}$, $K_{\ell(\mu)}$, $\dots$, $K_2$, $K_1$. Finally, add an edge of multiplicity $1$ from every vertex to a new vertex $0$ which will be the sink.
\end{definition}

\begin{example}\label{exa:conv_lines}
    Consider compositions $\mu = (4,1)$ and $\nu = (6,2)$, such that $n = |\mu| + |\nu| = 13$, and $k = 2$. Then the graph $G_{\mu,\nu}^{(k)}$ is actually $G$ from Example~\ref{exa:toppling}, Figure~\ref{fig:graph_fam}.\\
    The different colors encode the components defined by $\mu$ and $\nu$ and the line styles represent edge multiplicity as follows:
    \begin{itemize}
        \item edge of multiplicity $k$: continuous line ``$\vcenter{\hbox{\begin{tikzpicture}
            \draw[red,thick] (0,0) -- (1,0);
        \end{tikzpicture}}}$''
        (or any other color) and
        ``$\vcenter{\hbox{\begin{tikzpicture}
            \draw[] (0,0) -- (1,0);
        \end{tikzpicture}}}$'',
        \item edge of multiplicity $k-1$: dashed line ``$\vcenter{\hbox{\begin{tikzpicture}
            \draw[red,thick,dashed] (0,0) -- (1,0);
        \end{tikzpicture}}}$'' (or any other color),
        \item edge of multiplicity $1$: dotted line ``$\vcenter{\hbox{\begin{tikzpicture}
            \draw[thick,dotted] (0,0) -- (1,0);
        \end{tikzpicture}}}$''.
    \end{itemize}
\end{example}

\begin{remark}
    In the case $k = 1$, the family of graphs $\big\{G_{\mu,\nu}^{(1)}\big\}_{\mu,\nu}$ coincides with  the family $\big\{\hat{G}_{\mu,\nu}\big\}_{\mu,\nu}$ studied in~\cite{DDILLV}.
\end{remark}

Consider the graph $G_{(1^n),\varnothing}^{(k)}$, i.e.~the ``$k$-complete'' graph on $n$ vertices to which we add the sink $0$ connected with multiplicity $1$ to all other vertices. Finding the recurrent configurations is straightforward.

\begin{lemma}\label{lem:rec_kcomplete}
    Consider a stable non-negative configuration $c$ on the graph $G = G_{(1^n),\varnothing}^{(k)}$.\\
    Then $c$ is recurrent if and only if the increasingly ordered sequence of values of $c$ is pointwise equal to or greater than the string $0,k,2k,\dots,(n-1)k$.
\end{lemma}
\begin{proof}
    Suppose that $i_1,i_2,\dots,i_n$ is the order of $\tilde{V} = [n]$ in which the values $c(i_1)$, $c(i_2)$, $\dots$, $c(i_n)$ are (weakly) increasing.\\
    If $c$ satisfies the pointwise inequality, then the toppling order $i_n, i_{n-1}, \dots, i_1$ verifies the recurrent property for $c$. In fact, after the toppling of the sink $0$ the ordered values of the configuration are greater or equal then $1,k+1,2k+1,\dots,(n-1)k+1$.\\
    When we topple $i_j$, that vertex has already received $(n-j)k+1$ grains by the previous topplings, so its value is $c(i_j) + (n-j)k+1 \geq (n-1)k+1 = \deg_G(i_j)$. Thus, after its topple the configuration is still non-negative.\\
    Conversely, suppose $c$ is recurrent. Suppose $i_1',i_2',\dots,i_n'$ is the toppling sequence that shows recurrence. After toppling $0$ in $c$, we have that $c(i_1')+1 \geq \deg_G(i_1') = (n-1)k + 1$, thus $c(i_1') \geq (n-1)k$. Now, after toppling $0$ and $i_1'$, the recurrence implies that $c(i_2') + 1 + k \geq (n-1)k + 1$ which means $c(i_2') \geq (n-2)k$. In general for all $j \in [n]$ the recurrence property implies
    \[
        c(i_j') + 1 + (j-1)k \geq (n-1)k + 1 \qquad \text{and thus} \qquad c(i_j') \geq (n-j)k.
    \]
    Re-ordering the values of $c$ in increasing order, the pointwise inequality of the theorem must hold. This concludes the proof.
\end{proof}

\begin{example}\label{exa:gen_lemma_increasing}
    Consider the graph $G = G_{(1^6),\varnothing}^{(3)}$ and the stable non-negative configurations $c = 3\ 15\ 14\ 8\ 14\ 2$ and $c' = 3\ 15\ 11\ 8\ 11\ 2$.
    In Table~\ref{tab:stableconf_1n}, we write the values correspondence $i \leftrightarrow c(i)$ but with columns reordered so that the values $c(i_j)$ in the second row are increasing.
    \begin{table}
        \begin{tabular}{r|cccccccr|cccccc}
            $i_j$ & 1 & 2 & 3 & 4 & 5 & 6 & $\qquad$ & $i_j$ & 6 & 1 & 4 & 3 & 5 & 2 \\
            $c(i)$ & 2 & 3 & 8 & 14 & 14 & 15 & & $c'(i_j)$ & 2 & 3 & 8 & 11 & \textbf{\color{red}11} & 15 \\
            $(j-1)k$ & 0 & 3 & 6 & 9 & 12 & 15 & & $(j-1)k$ & 0 & 3 & 6 & 9 & \textbf{\color{red}12} & 15
        \end{tabular}
        \caption{Examples of stable and non-stable configurations on $G_{(1^6),\varnothing}^{(3)}$}\label{tab:stableconf_1n}
    \end{table}\\
    Starting from configuration $c$ (on the left in Table~\ref{tab:stableconf_1n}), we can check that the toppling order $2$, $5$, $3$, $4$, $1$, $6$ works. As a matter of fact, the increasing values of the configuration are pointwise equal to or greater than the sequence $0$, $3$, $6$, $9$, $12$, $15$ (see the comparison between the second and third rows of the tables).\\
    Notice that the vertices of $G$ become unstable when their value is greater or equal than the number $(n-1)k + 1 = 16$, so the following topplings are all allowed:
    \[
        \begin{array}{rcccccc}
            c =                                             &  3 & 15 & 14 &  8 & 14 &  2 \\
            \phi_0(c) =                                     &  4 & \mathbf{\color{red}16} & 15 &  9 & 15 &  3 \\
            \phi_2\phi_0(c) =                               &  7 &  0 & {\color{red}18} & 12 & \mathbf{\color{red}18} &  6 \\
            \phi_5\phi_2\phi_0(c) =                         & 10 &  3 & \mathbf{\color{red}21} & 15 &  2 &  9 \\
            \phi_3\phi_5\phi_2\phi_0(c) =                   & 13 &  6 &  5 & \mathbf{\color{red}18} &  5 & 12 \\
            \phi_4\phi_3\phi_5\phi_2\phi_0(c) =             & \mathbf{\color{red}16} &  9 &  8 &  2 &  8 & 15 \\
            \phi_1\phi_4\phi_3\phi_5\phi_2\phi_0(c) =       &  0 & 12 & 11 &  5 & 11 & \mathbf{\color{red}18} \\
            \phi_6\phi_1\phi_4\phi_3\phi_5\phi_2\phi_0(c) = &  3 & 15 & 14 &  8 & 14 &  2
        \end{array}
    \]
    and $c$ is recurrent. In the example, the unstable vertices are highlighted in red and the bold ones are the vertices on which the toppling is performed.\smallskip\\
    However, by repeating the same argument with configuration $c'$, we notice that the increasing sequence $c'(i_j)$ is not pointwise greater than or equal to $0$, $3$, $6$, $9$, $12$, $15$: the inequality is not satisfied by $i_5 = 5$ since $11 = c'(5) = c(i_5) \not\geq 3(5 - 1) = 12$.\\
    Indeed, after toppling the sink $0$ and the first vertex of the sequence $i_j'$, there is no unstable vertex:
    \[
        \begin{array}{rcccccl}
            c =                                             &  3 & 15 & 11 &  8 & 11 &  2 \\
            \phi_0(c) =                                     &  4 & \mathbf{\color{red}16} & 12 &  9 & 12 &  3 \\
            \phi_2\phi_0(c) =                               &  7 &  0 & 15 & 12 & 15 &  6.
        \end{array}
    \]
    If there were one more grain on vertex $5$, i.e. $c'(5) = 12$ satisfied the pointwise inequality, then the vertex $5$ would have been unstable in $\phi_2\phi_0(c')$ and the configuration $c'$ recurrent.
\end{example}

\section{Sorted sandpiles and their statistics}\label{sec:sortedrec}

Given a composition $\mu=(\mu_1,\mu_2,\dots,\mu_r)$, set $\mu^\rev:=(\mu_r,\dots,\mu_2,\mu_1)$.

Consider the Young subgroup $\SG_\nu \times \SG_{\mu^{\rev}} < \SG_n$ acting on the configurations of $G_{\mu,\nu}^{(k)}$ by composition, i.e.~permuting in each component the values of every $c \in \Conf\big(G_{\mu,\nu}^{(k)}\big)$ by
\[
    \sigma \cdot c(i) := c(\sigma(i))
\]
for all $i \in [n]$ and $\sigma \in \SG_{\nu} \times \SG_{\mu^{\rev}}$. 

The orbits form equivalence classes of configurations on the graph $G_{\mu,\nu}^{(k)}$. Each class has a canonical representative: the unique configuration with increasing values in the $(k-1)$-components $\{H_j\}_j$ and decreasing values in the $k$-components $\{K_i\}_i$. Such representatives for stable non-negative configurations are called \emph{sorted configurations} and their set will be denoted $\Sort_k(\mu;\nu)$.

\begin{example}\label{exa:gen_sortconf}
    The configuration $c = {\color{mycolor1} 6}\ {\color{mycolor1} 17}\ {\color{mycolor1} 18}\ {\color{mycolor1} 24}\ {\color{mycolor2} 2}\ {\color{mycolor3} 17}\ {\color{mycolor3} 14}\ {\color{mycolor4} 18}\ {\color{mycolor4} 18}\ {\color{mycolor4} 15}\ {\color{mycolor4} 12}\ {\color{mycolor4} 7}\ {\color{mycolor4} 2}$ from Example~\ref{exa:toppling} is a sorted configuration on $G_{({\color{mycolor1}4},{\color{mycolor2}1}),({\color{mycolor4}6},{\color{mycolor3}2})}^{(2)}$.\\
    However, the configuration $c' := \phi_6\phi_{10}\phi_0(c) = {\color{mycolor1} 11}\ {\color{mycolor1} 22}\ {\color{mycolor1} 23}\ {\color{mycolor1} 2}\ {\color{mycolor2} 7}\ {\color{mycolor3} 22}\ {\color{mycolor3} 19}\ {\color{mycolor4} 1}\ {\color{mycolor4} 20}\ {\color{mycolor4} 19}\ {\color{mycolor4} 16}\ {\color{mycolor4} 11}\ {\color{mycolor4} 6}$ computed before is not, since ${\color{mycolor1}c'(11)} > {\color{mycolor1}c'(10)}$ and ${\color{mycolor4}c'(6)} < {\color{mycolor4}c'(5)}$.
\end{example}

Notice that the action of an element $\sigma \in \SG_\nu \times \SG_{\mu^{\rev}}$ only \emph{permutes} the values of a configuration $c$ of $G_{\mu,\nu}^{(k)}$. Thus, if $c$ is recurrent $\sigma \cdot c$ is recurrent as well and given a suitable toppling sequence $\tau \in \SG_n$ for $c$, the sequence $\sigma \circ \tau$ works for $\sigma \cdot c$.
This observation motivates the following definition.

\begin{definition}\label{def:sortrec}
    The \emph{sorted recurrent configurations} of a graph $G_{\mu,\nu}^{(k)}$ are the sorted configurations which are also recurrent. We will denote by $\SortRec_k(\mu;\nu)$ the set of such configurations.
\end{definition}

\begin{example}
    The configuration $c$ from Example~\ref{exa:toppling} is a sorted recurrent configuration of $G_{({\color{mycolor1}4},{\color{mycolor2}1}),({\color{mycolor4}6},{\color{mycolor3}2})}^{(2)}$.
\end{example}

\subsection{The statistic $\lev$}

The first statistic measures the ``degree'' of the configuration (i.e.~the sum of its values), up to a constant term depending only on the graph structure.

\begin{definition}\label{def:level}
    Consider a sorted recurrent configuration $c$ of $G_{\mu,\nu}^{(k)}$. Its \emph{level statistic} is:
    \begin{equation*}\label{eq:level}
        \lev(c) := -\big|E_0(G_{\mu,\nu}^{(k)})\big| + \sum_{i=1}^{n}c(i)
    \end{equation*}
    where $E_0\big(G_{\mu,\nu}^{(k)}\big)$ is the multiset of all edges (counted with multiplicity) not incident to the sink $0$.
\end{definition}

\begin{example}\label{exa:level}
    Consider the configuration $c = {\color{mycolor1} 6}\ {\color{mycolor1} 17}\ {\color{mycolor1} 18}\ {\color{mycolor1} 24}\ {\color{mycolor2} 2}\ {\color{mycolor3} 17}\ {\color{mycolor3} 14}\ {\color{mycolor4} 18}\ {\color{mycolor4} 18}\ {\color{mycolor4} 15}\ {\color{mycolor4} 12}\ {\color{mycolor4} 7}\ {\color{mycolor4} 2}$ from Example~\ref{exa:toppling}. In this case, $E_0(G_{({\color{mycolor1}4},{\color{mycolor2}1}),({\color{mycolor4}6},{\color{mycolor3}2})}^{(2)}) = 140$ so $\lev(c) = 170 - 140 = 30$.
\end{example}
It is well known that the level of a recurrent configuration of a connected graph is nonnegative, and that there is always at least one recurrent configuration of level $0$.

\subsection{The statistic $\del$}

One of the main contributions of the present article is the correct definition of the $\del$ statistic for sorted recurrent configurations on $G_{\mu,\nu}^{(k)}$ for $k\geq 2$. Indeed, it extends the one defined in~\cite{DDILLV} to graphs with multiedges. This definition is the one that inspired our definition of the $\pmaj$.
\smallskip

The key idea is to ``slow'' the release of the grains of sand when a vertex is toppled.

\begin{definition}\label{def:release}
    Consider a configuration $c$ on a graph $G$. Fix a non-negative integer $m \in \N$.
    A \emph{grain release} on vertex $v \in V$ with threshold $m$ sends $c$ into a new configuration $c'$ defined by:
    \begin{equation*}\label{eq:release}
        c'(w) := \begin{cases}
            c(v) - \#\{w \in V \setminus \{v\} \ | \ \epsilon(v,w) \geq m\} & \text{if $w = v$}\\
            c(w) - 1 & \text{if $w \neq v$ and $\epsilon(v,w) \geq m$}\\
            c(w) & \text{if $w \neq v$ and $\epsilon(v,w) < m$}
        \end{cases}
    \end{equation*}
    where $\epsilon(v,w)$ is the multiplicity of edge $vw$ in the graph $G$. In other words, we send one grain from $v$ through its incident edges that have multiplicity greater than or equal to the threshold $m$.\\
    To keep track the threshold and vertex of a grain release, we denote $c' = \psi_v^{(m)}(c)$.
\end{definition}

\begin{example}\label{exa:release}
    Take the configuration $\phi_0(c)$ on the graph $G_{({\color{mycolor1}4},{\color{mycolor2}1}),({\color{mycolor4}6},{\color{mycolor3}2})}^{(2)}$ from Example~\ref{exa:toppling}. By construction, the edges of vertex $10$ have all multiplicity $2$ except for the one connecting it to the sink. Therefore:
    \begin{align*}
        \phi_0(c) &= {\color{mycolor1} 7}\ {\color{mycolor1} 18}\ {\color{mycolor1} 19}\ {\color{mycolor1} 25}\ {\color{mycolor2} 3}\ {\color{mycolor3} 18}\ {\color{mycolor3} 15}\ {\color{mycolor4} 19}\ {\color{mycolor4} 19}\ {\color{mycolor4} 16}\ {\color{mycolor4} 13}\ {\color{mycolor4} 8}\ {\color{mycolor4} 3}\\
        \psi_{10}^{(2)}\phi_{0}(c) &= {\color{mycolor1} 8}\ {\color{mycolor1} 19}\ {\color{mycolor1} 20}\ {\color{mycolor1} 12}\ {\color{mycolor2} 4}\ {\color{mycolor3} 19}\ {\color{mycolor3} 16}\ {\color{mycolor4} 20}\ {\color{mycolor4} 20}\ {\color{mycolor4} 17}\ {\color{mycolor4} 14}\ {\color{mycolor4} 9}\ {\color{mycolor4} 4}\\
        \psi_{10}^{(1)}\psi_{10}^{(2)}\phi_{0}(c) &= {\color{mycolor1} 9}\ {\color{mycolor1} 20}\ {\color{mycolor1} 21}\ {\color{mycolor1} 0}\ {\color{mycolor2} 5}\ {\color{mycolor3} 20}\ {\color{mycolor3} 17}\ {\color{mycolor4} 21}\ {\color{mycolor4} 21}\ {\color{mycolor4} 18}\ {\color{mycolor4} 15}\ {\color{mycolor4} 10}\ {\color{mycolor4} 5} = \phi_{10}\phi_0(c)
    \end{align*}
    and notice that at the grain release $\psi_{10}^{(1)}$ we reduce the grains on $10$ by $12$ and not $13$, because no grain is sent to the sink.\\
    It is more interesting to look at grain releases on the vertex $6$, since it is connected with multiplicity $1$ to vertices $0$, $1$, $2$, $3$, $4$ and $5$ and with multiplicity $2$ to the remaining. In this case:
    \begin{align*}
        \psi_6^{(2)}\phi_{10}\phi_{0}(c) &= {\color{mycolor1} 10}\ {\color{mycolor1} 21}\ {\color{mycolor1} 22}\ {\color{mycolor1} 1}\ {\color{mycolor2} 6}\ {\color{mycolor3} 21}\ {\color{mycolor3} 18}\ {\color{mycolor4} 8}\ {\color{mycolor4} 22}\ {\color{mycolor4} 19}\ {\color{mycolor4} 16}\ {\color{mycolor4} 11}\ {\color{mycolor4} 6}\\
        \psi_6^{(1)}\psi_6^{(2)}\phi_{10}\phi_{0}(c) &= {\color{mycolor1} 11}\ {\color{mycolor1} 22}\ {\color{mycolor1} 23}\ {\color{mycolor1} 2}\ {\color{mycolor2} 7}\ {\color{mycolor3} 22}\ {\color{mycolor3} 19}\ {\color{mycolor4} 1}\ {\color{mycolor4} 22}\ {\color{mycolor4} 19}\ {\color{mycolor4} 16}\ {\color{mycolor4} 11}\ {\color{mycolor4} 6} = \phi_6\phi_{10}\phi_0(c).
    \end{align*}
    By comparing with the results from Example~\ref{exa:toppling}, we see that applying $\psi_v^{(1)}\psi_v^{(2)}$ in these cases is equivalent to toppling the vertex $v$. In the following remark we state this general observation.
\end{example}

\begin{remark}\label{rem:rel_toppl}
    For any configuration $c$ on a graph $G$, consider a vertex $v$ and let $h$ be the maximum of the multiplicity of its outgoing edges. We can rewrite the toppling of $v$ in terms of grain releases as:
    \[
        \phi_v(c) = \psi_v^{(1)} \circ \psi_v^{(2)} \circ \dots \psi_v^{(h-1)} \circ \psi_v^{(h)} (c).
    \]
\end{remark}

The grain release allows to precisely control the movement of grains in Algorithm~\ref{alg:toppl}.

\begin{algorithm}
    \setstretch{1.1}        
    \RestyleAlgo{ruled}     
    \caption{Toppling algorithm with ``slow-release''}\label{alg:toppl}
    \KwIn{A recurrent configuration $c$ of $G_{\mu,\nu}^{(k)}$}
    \KwOut{A word $\wtopp = w_1w_2\dots w_{nk}$ of length $nk$ recording each grain release and an integer value $d$}
    \smallskip
    Initialize the word $\wtopp$ as empty\\
    Initialize the integer values $d = 0$ and $n_{\text{cycles}} = 0$\\
    Initialize a $n$-tuple, the release record $\mathbf{r} = (r_1,\dots,r_n)$ with all zero entries\\
    $c \gets \phi_0(c)$\Comment{Topple the sink $0$}\\
    \While{$\mathbf{r} \neq (k,k,\dots,k)$}{
        \For(\Comment{Call each iteration \emph{visit} of vertex $i$}){$i$ from $n$ to $1$ in decreasing order}{
            \uIf(\Comment{First release from vertex $i$}){$r_i = 0$ and $i$ is unstable}{
                $d \gets d + n_{\text{cycles}}$\Comment{Register the number of cycles before first release on $i$}\\
                $c \gets \psi_i^{(k-r_i)}(c)$\Comment{Perform a grain release on $i$ with threshold $0$}\\
                $r_i \gets r_i + 1$\\
                Append $i$ to the word $\wtopp$
            }
            \ElseIf(\Comment{Another release from vertex $i$}){$0 < r_i < k$}{
                $c \gets \psi_i^{((k-1)-r_i)}(c)$\Comment{Perform a grain release on $i$ with threshold $r_i$}\\
                $r_i \gets r_i + 1$\\
                Append $i$ to the word $\wtopp$
            }
        }
        $n_{\text{cycles}} \gets n_{\text{cycles}} + 1$
    }
\end{algorithm}

\begin{example}\label{exa:alg_toppl}
    We refer to Appendix~\ref{app:comp_algo} to show record tables for each slow-release performed by Algorithm~\ref{alg:toppl} on the configurations $c$ and $\tilde{c}$ that will be defined in Example~\ref{exa:tildemap}.\\
    In both tables, the first column records the vertex on which we apply a slow-release and the other ones keep track of the values of the configuration on all $13$ vertices.
\end{example}

Some properties of this algorithm need to be discussed.

\begin{proposition}\label{prop:alg_toppl}
    For each recurrent configuration $c$ of $G_{\mu,\nu}^{(k)}$:
    \begin{enumerate}[label=(\alph*)]
        \item Algorithm~\ref{alg:toppl} terminates,
        \item $\wtopp$ is a word of length $nk$ and each letter $i \in [n]$ occurs exactly $k$ times,
        \item\label{it:algo3} each completion of the \textbf{while} cycle (other than the last one) is represented in $\wtopp$ by a weak ascent\footnote{Weak ascents are indexes $h \in [nk - 1]$ such that $w_h \leq w_{h+1}$.}. 
    \end{enumerate}
\end{proposition}
\begin{proof}[Proof of Proposition~\ref{prop:alg_toppl}]
    Notice that at most $k$ releases are performed on each non-sink vertex by design. Moreover, releases are only performed once the vertex becomes unstable, i.e.~it can be toppled once. Recall that by Remark~\ref{rem:rel_toppl} this is equivalent to performing $k$ releases with thresholds $k,k-1,\dots,2,1$. Thus, after each release the configuration remains non-negative.\smallskip\\
    All the desired properties of Algorithm~\ref{alg:toppl} follow from the next observation:
    \begin{center}
        \ul{Claim}: if during the execution of the algorithm we have $n$ consecutive visits without release, then $\mathbf{r} = (k,k,\dots,k)$.
    \end{center}
    Suppose that the $n$ consecutive visits without releases, say to vertices $i, i-1, \dots,1, n, n-1\dots, i+1$, happen while the current configuration is $c'$.
    This means that the current release record $\mathbf{r}$ has only values $0$ or $k$ because if $0 < r_j < k$ for some $j \in [n]$ then a release would have been performed when $j$ is visited.\\
    By Remark~\ref{rem:rel_toppl}, this means that $c'$ was obtained from $c$ by toppling the sink and those vertices $j \in \tilde{V}$ such that $r_j = k$.\\
    Moreover, all vertices $j$ such that $r_j = 0$ must be stable since there is no release. If $\mathbf{r} \neq (k,k,\dots,k)$ this would contradict the fact that $c$ is a recurrent configuration (see Remark~\ref{rem:reconf}), since we would have that after toppling a strict subset of all vertices, the remaining ones are all stable.\smallskip\\
    From this claim, we deduce the desired properties.
    \begin{enumerate}[label=(\alph*)]
        \item Suppose that the algorithm does not terminate: this happens only if the \textbf{while} loop is never exited, thus infinitely many visits are performed on the vertices of the graph.\\
        By design there must be at most $nk$ visits with release. Thus there will eventually be $n$ consecutive visits without any release. By the previous \ul{claim}, this means that $\mathbf{r} = (k,k,\dots,k)$ which contradicts the fact that the algorithm never leaves the \textbf{while} loop.
        \item Recall that the algorithm exits the \textbf{while} loop when $\textbf{r} = (k,\dots,k)$. By design, the entry $r_i$ increments by $+1$ if and only if a new letter $i$ is added to the word $\wtopp$. Therefore $\wtopp$ has $r_1 + r_2 + \dots + r_n = nk$ letters and each $i \in [n]$ appears exactly $r_i = k$ times.
        \item When a \textbf{while} cycle other than the last is completed, $\mathbf{r} \neq (k,k,\dots,k)$. By \ul{claim}, this means that in the new \textbf{while} cycle some vertex visit will trigger a release.\\
        Let $w_1\dots w_m$ be the state of the word $\wtopp$ at the start of the new cycle: the first vertex visited with release in the new iteration will be $w_{m+1}$.\\
        Suppose $w_{m} > w_{m+1}$: this means that between the corresponding visits with release there were at least $n$ consecutive visits without release (recall that we visit vertices in decreasing order), which is a contradiction. Therefore each \textbf{while} cycle completion corresponds to a weak ascent.\\
        On the other hand, if there is a weak ascent in the word $\wtopp$, this means that a \textbf{while} cycle has been completed between the two visits with release because at each iteration vertices are visited in decreasing order.
    \end{enumerate}
    This concludes the proof.
\end{proof}

\begin{remark}
    For $k = 1$ a release with threshold $m = 1$ coincides with the toppling move on the same vertex, since the multiplicity of edges is always $1$. Thus in this case our algorithm coincides with~\cite[Algorithm 1]{DDILLV}.
\end{remark}

\begin{remark}\label{rem:test_rec}
    By design of the algorithm, if we run it on a (not necessarily recurrent) configuration $c$ and it eventually stops (in other words, it does not get stuck in a \textbf{while} loop without slow-releases) then $c$ is recurrent.
\end{remark}

We are finally ready to define the $\del$ statistic.

\begin{definition}\label{def:delay}
    The \emph{delay} associated to a sorted recurrent configuration $c$ of $G_{\mu,\nu}^{(k)}$ is the value $d$ obtained by running Algorithm~\ref{alg:toppl}.
\end{definition}

By Proposition~\ref{prop:alg_toppl} Point~\ref{it:algo3}, one can easily recover $\del(c)$ from the \emph{toppling word} $\wtopp(c)$.\\
In fact, let $d_i$ the number of weak ascents preceding the first occurrence of $i$ in $\wtopp(c)$ for all $i \in [n]$. Then:
\begin{equation*}\label{eq:equiv_delay}
    \del(c) = \sum_{i \in [n]}d_i(c).
\end{equation*}

\begin{example}\label{exa:delay}
    Consider the configurations $c$ and $\tilde{c}$ in Example~\ref{exa:toppling} and the execution of Algorithm~\ref{alg:toppl} on them. The resulting toppling word in both cases is:
    \begin{align*}
        \wtopp(c) &= \wtopp(\tilde{c})\\
        &= \textbf{10},\ \textbf{6},\ \textbf{5},\ 10,\ 6,\ 5,\ \textbf{4},\ \textbf{12},\ \textbf{11},\ \textbf{8},\ \textbf{7},\ 4,\ \textbf{3},\ 12,\ 11,\ 8,\ 7,\ 3,\ \textbf{2},\ \textbf{13},\ 2,\ 13,\ \textbf{9},\ \textbf{1},\ 9,\ 1
    \end{align*}
    where we use the bold style to highlight the first occurrence of each vertex in the word.\\
    Keeping track of the ascents in $\wtopp(c) = \wtopp(\tilde{c})$ and removing the latter occurrences of each vertex in the word, we obtain
    \[
        \def\arraystretch{1.5}
        \begin{array}{l||ccc|c|ccccc|c|c|cc}
            i & 10 & 6 & 5 & 4 & 12 & 11 & 8 & 7 & 3 & 2 & 13 & 9 & 1 \\
            d_i & 0 & 0 & 0 & 1 & 2 & 2 & 2 & 2 & 2 & 3 & 4 & 5 & 5
        \end{array}
    \]
    so one can sum up the contributes and get
    \[
        \del(c) = \del(\tilde{c}) = 28.
    \]
\end{example}

\section{A bijection between labelled Dyck paths and sorted sandpiles}\label{sec:pfandsand}

The main result of this section is the construction of a bijection between the labelled Dyck paths in $\PF_{n,kn}(\mu;\nu)$ and sorted recurrent configurations on the family of graphs $G_{\mu,\nu}^{(k)}$. The argument is a direct extension of the proof given in~\cite{DDILLV} for the case $k = 1$.
\medskip

It is useful to introduce a new notation for the vertices of $G_{\mu,\nu}^{(k)}$ that highlights which clique component each vertex lies in:
\begin{align*}
    v_i^{(\mu_s)} &:= \text{the $i^{\text{th}}$ vertex of the $k$-clique associated to part $\mu_s$},\\
    v_j^{(\nu_s)} &:= \text{the $j^{\text{th}}$ vertex of the $(k-1)$-clique associated to part $\nu_s$}.
\end{align*}
Here, the order of vertices in every clique is given by the natural labelling of the vertices.

\subsection{A bijection between sets of configurations}

The first step in creating the bijection is to inject the set of recurrent configurations on $G_{\mu,\nu}^{(k)}$ in a modified graph where we add edges to make $(k-1)$-clique components into $k$-cliques. We construct this injection in such a way that the statistics $\lev$ and $\del$ are preserved. Denote $G_{\mu \sqcup \nu, \varnothing}^{(k)} := G_{\mu \nu^{\rev}, \varnothing}^{(k)}$ the modified graph. We define the \emph{embedding map}:
\def\arraystretch{1.5}
\begin{equation*}
    \begin{array}{cccc}
        \tilde{\cdot}: & \Sort_k(\mu;\nu) & \longrightarrow & \Conf(G_{\mu \sqcup \nu, \varnothing}^{(k)})\\
         & c & \longmapsto & \tilde{c}
    \end{array}
\end{equation*}
where for every possible $i$, $j$ and $s$:
\begin{equation*}
    \tilde{c}(v_i^{(\mu_s)}) := c(v_i^{(\mu_s)})
    \qquad \text{and} \qquad 
    \tilde{c}(v_j^{(\nu_s)}) := c(v_j^{(\nu_s)}) + (j - 1).
\end{equation*}

\begin{example}\label{exa:tildemap}
    Consider the graph $G_{({\color{mycolor1}4},{\color{mycolor2}1}),({\color{mycolor4}6},{\color{mycolor3}2})}^{(2)}$ and its sorted recurrent configuration\\ $c = {\color{mycolor1} 6}\ {\color{mycolor1} 17}\ {\color{mycolor1} 18}\ {\color{mycolor1} 24}\ {\color{mycolor2} 2}\ {\color{mycolor3} 17}\ {\color{mycolor3} 14}\ {\color{mycolor4} 18}\ {\color{mycolor4} 18}\ {\color{mycolor4} 15}\ {\color{mycolor4} 12}\ {\color{mycolor4} 7}\ {\color{mycolor4} 2}$ from Example~\ref{exa:toppling}.\\
    The map $\tilde{\cdot}$ sends $c$ to the configuration $\tilde{c}$ of $G_{({\color{mycolor1}4},{\color{mycolor2}1},{\color{mycolor3}2},{\color{mycolor4}6}), \varnothing}^{(2)}$ where we modify the following values:
    \[\begin{array}{ll}
        \tilde{c}({\color{mycolor4}1}) = \tilde{c}({\color{mycolor4}v_{1}^{\nu_1}}) = c({\color{mycolor4}1}) + {\color{mycolor4}0} = 2 \qquad &
        \tilde{c}({\color{mycolor4}5}) = \tilde{c}({\color{mycolor4}v_{5}^{\nu_1}}) = c({\color{mycolor4}5}) + {\color{mycolor4}4} = 22\\
        \tilde{c}({\color{mycolor4}2}) = \tilde{c}({\color{mycolor4}v_{2}^{\nu_1}}) = c({\color{mycolor4}2}) + {\color{mycolor4}1} = 8 \qquad &
        \tilde{c}({\color{mycolor4}6}) = \tilde{c}({\color{mycolor4}v_{6}^{\nu_1}}) = c({\color{mycolor4}6}) + {\color{mycolor4}5} = 23\\
        \tilde{c}({\color{mycolor4}3}) = \tilde{c}({\color{mycolor4}v_{3}^{\nu_1}}) = c({\color{mycolor4}3}) + {\color{mycolor4}2} = 14 \qquad &
        \tilde{c}({\color{mycolor3}7}) = \tilde{c}({\color{mycolor3}v_{1}^{\nu_2}}) = c({\color{mycolor3}7}) + {\color{mycolor3}0} = 14\\
        \tilde{c}({\color{mycolor4}4}) = \tilde{c}({\color{mycolor4}v_{4}^{\nu_1}}) = c({\color{mycolor4}4}) + {\color{mycolor4}3} = 18 \qquad &
        \tilde{c}({\color{mycolor3}8}) = \tilde{c}({\color{mycolor3}v_{2}^{\nu_2}}) = c({\color{mycolor3}8}) + {\color{mycolor3}1} = 18.
    \end{array}\]
    Thus, we have $\tilde{c} = {\color{mycolor1} 6}\ {\color{mycolor1} 17}\ {\color{mycolor1} 18}\ {\color{mycolor1} 24}\ {\color{mycolor2} 2}\ {\color{mycolor3} 18}\ {\color{mycolor3} 14}\ {\color{mycolor4} 23}\ {\color{mycolor4} 22}\ {\color{mycolor4} 18}\ {\color{mycolor4} 14}\ {\color{mycolor4} 8}\ {\color{mycolor4} 2}$.
\end{example}

The idea behind map $\tilde{\cdot}$ is to counteract the fact that we are embedding the configuration in a graph with more edges. One can imagine that the grains added to each vertex are exactly the ones sent through the new edges during a toppling or a slow-release on the vertex.

Recall that the values $c(v_j^{(\nu_s)})$ are weakly increasing with respect to $j$ when $c$ is a sorted configuration of $G_{\mu,\nu}^{(k)}$. This implies that the values $\tilde{c}(v_j^{(\nu_s)})$ will be \textit{strictly} increasing, as we can see in Example~\ref{exa:tildemap}.

\begin{definition}\label{def:sortsqcup}
    Let $\Sort_k(\mu \sqcup \nu; \varnothing)$ be the set of all representatives for the sorted non-negative and stable configurations on $G_{\mu \sqcup \nu,\varnothing}^{(k)}$ which are \textbf{strictly} increasing in the $k$-clique components associated to parts $\nu_s$ and weakly decreasing in $k$-cliques associated to $\mu_s$.\\
    The subset of those configurations which are also recurrent is denoted $\SortRec_k(\mu \sqcup \nu; \varnothing)$.
\end{definition}

The condition of being strictly increasing in each clique actually characterizes the image of the embedding map, as we will show with the next technical result.

\begin{proposition}\label{prop:tildemap}
    Consider the embedding map $c \mapsto \tilde{c}$.
    \begin{enumerate}[label=(\alph*)]
        \item\label{it:tilde1} It is a bijection $\Sort_k(\mu;\nu) \to \Sort_k(\mu \sqcup \nu;\varnothing)$.
        \item\label{it:tilde2} The restriction $\SortRec_k(\mu;\nu) \to \SortRec_k(\mu \sqcup \nu;\varnothing)$ is a bijection.
        \item\label{it:tilde3} The configurations $c \in \SortRec_k(\mu,\nu)$ and $\tilde{c}$ produce the same $\wtopp$ using Algorithm~\ref{alg:toppl}.
    \end{enumerate}
\end{proposition}
\begin{proof}
    To simplify notation, let $G := G_{\mu,\nu}^{(k)}$ and $\tilde{G} := G_{\mu\sqcup\nu,\varnothing}^{(k)}$.\\
    In order to show Point~\ref{it:tilde1}, we need to check that the embedding map and its natural inverse $\tilde{c} \mapsto c$ (which subtracts the right number of grains on each vertex) preserve the defining properties of their domain and codomain.
    \begin{itemize}
        \item \ul{Sortedness}: since the embedding map (and its inverse) does not change values of vertices $v_i^{(\mu_s)}$, we only need to check sortedness in the cliques associated to parts of $\nu$. Let $v_{j}^{(\nu_s)}$ and $v_{j+1}^{(\nu_s)}$ be generic consecutive vertices in such a clique. Consider $c \in \Sort_k(\mu;\nu)$, then $c(v_j^{(\nu_s)}) \leq c(v_{j+1}^{(\nu_s)})$ by definition. This implies that:
        \[
            \tilde{c}(v_j^{(\nu_s)}) = c(v_j^{(\nu_s)}) + j \leq c(v_{j+1}^{(\nu_s)}) + j < c(v_{j+1}^{(\nu_s)}) + j + 1 = \tilde{c}(v_{j+1}^{(\nu_s)}).
        \]
        Conversely, for any $\tilde{c} \in \Sort_k(\mu\sqcup\nu;\varnothing)$ we have $\tilde{c}(v_j^{(\nu_s)}) \leq \tilde{c}(v_{j+1}^{(\nu_s)}) - 1$ thus
        \[
            c(v_j^{(\nu_s)}) = \tilde{c}(v_j^{(\nu_s)}) - j \leq \tilde{c}(v_{j+1}^{(\nu_s)}) - j - 1 = c(v_{j+1}^{(\nu_s)}).
        \]
        \item \ul{Non-negativity and stability}: looking at the degree of vertices, by construction we see that:
        \begin{align*}
            &\deg_{\tilde{G}}(v_i^{(\mu_s)}) = \deg_G(v_i^{(\mu_s)})\\
            &\deg_{\tilde{G}}(v_j^{(\nu_s)}) = \deg_G(v_j^{(\nu_s)}) + (\nu_s - 1).
        \end{align*}
        Consider a configuration $c \in \Sort_k(\mu;\nu)$, by definition $0 \leq c(v_j^{(\nu_s)}) < \deg_G(v_j^{(\nu_s)})$ so it is also true that
        \[
            0 \leq \tilde{c}(v_j^{(\nu_s)}) = c(v_j^{(\nu_s)}) + (j - 1) \leq \deg_G(v_j^{(\nu_s)}) + (\nu_s - 1) = \deg_{\tilde{G}}(v_j^{(\nu_s)}).
        \]
        On the other hand, given a configuration $\tilde{c} \in \Sort_k(\mu\sqcup\nu; \varnothing)$ we have by definition that $j - 1 \leq \tilde{c}(v_j^{(\nu_s)}) < \deg_{\tilde{G}}(v_j^{(\nu_s)})$. Actually, since the values are \emph{strictly} increasing, the condition that $\tilde{c}(v_{\nu_s}^{(\nu_s)})$ must be stable implies that $\tilde{c}(v_j^{(\nu_s)}) < \deg_{\tilde{G}}(v_j^{(\nu_s)}) - (\nu_s - j)$.\\
        Thus
        \[
            0 \leq \tilde{c}(v_j^{(\nu_s)}) - (j-1) = c(v_j^{(\nu_s)}) \leq \deg_{\tilde{G}}(v_j^{(\nu_s)}) - (\nu_s - j) - (j-1) = \deg_G(v_j^{(\nu_s)}).
        \]
    \end{itemize}
    Since both maps are well defined and trivially injective, the embedding map $\Sort_k(\mu;\nu) \to \Sort_k(\mu \sqcup \nu;\varnothing)$ is a bijection.
    \medskip\\
    To prove both Points~\ref{it:tilde2} and~\ref{it:tilde3}, consider the configuration $c \in \Sort_k(\mu;\nu)$ and its corresponding $\tilde{c} \in \Sort_k(\mu\sqcup\nu;\varnothing)$. Suppose that one between $c$ and $\tilde{c}$ is recurrent. We will prove that both $c$ and $\tilde{c}$ are recurrent and the toppling words $\wtopp(c)$ and $\wtopp(\tilde{c})$ will be the same, by showing that Algorithm~\ref{alg:toppl} starts slow-releases at the same exact steps. In other words...
    \begin{center}
        \ul{Claim}: the toppling algorithm on configurations $c$ and $\tilde{c}$ terminates when at least one of them is recurrent, and the actions on the corresponding graphs $G$ and $\tilde{G}$ (i.e.~the start or continuation of a slow-release or the skip of the vertex) coincide at each step.
    \end{center}
    To keep track of each step of the algorithm, suppose that Algorithm~\ref{alg:toppl} performs $N > 0$ iterations on both configurations $c$ and $\tilde{c}$. We define words $\mathbf{u} := u_1u_2\dots u_N$ and $\mathbf{u}' := u_1'u_2'\dots u_N'$ such that if at step $1 \leq i \leq N$ vertex $v \in [n]$ is read\footnote{The reading order of the vertices on $G$ and $\tilde{G}$ is the same.}:
    \[
        u_i := \begin{cases}
            v & \text{if at step $i$ a slow-release is performed on $v$}\\
            \bullet & \text{if at step $i$ $v$ is skipped}
        \end{cases}
    \]
    \[
        u_i' := \begin{cases}
            v & \text{if at step $i$ a slow-release is performed on $v$}\\
            \bullet & \text{if at step $i$ $v$ is skipped}.
        \end{cases}
    \]
    Then, we will often use the following values for $i = 1,\dots,N$:
    \[
        r_i(v) := \#\{j \leq i \ | \ u_j = v\}
        \qquad
        r_i'(v) := \#\{j \leq i \ | \ u_j' = v\}
    \]
    and the number of slow-releases after $i$ steps:
    \[
        M_i := \#\{j \leq N \ | \ u_j \neq \bullet\}.
    \]
    Using the lists of moves $\mathbf{u}$ and $\mathbf{u}'$, we construct the sequence of all configurations the algorithm goes through. Define $\psi_{\bullet}$ to be the identity map on $\Sort_k(\mu\sqcup\nu;\varnothing)$, then
    \[  
        \begin{array}{l}
            c_0 := \phi_0(c)\\
            c_h := \psi_{u_h}^{(r_h(u_h))}(c_{h-1})
        \end{array}
        \qquad
        \begin{array}{l}
            c_0' := \phi_0(c)\\
            c_h' := \psi_{u_h'}^{(r_h'(u_h'))}(c_{h-1}').
        \end{array}
    \]
    for all $0 < h \leq N$.\\
    Now we prove that $\mathbf{u} = \mathbf{u}'$ by induction on $N$. The base case $N = 0$ is trivially true. Suppose the claim is true for $N > 0$, we want to show that the algorithm make the same next step (i.e.~step $N+1$). Let us compute the configurations $c_N$ and $c_N'$ for all vertices $v \in [n]$. We have several cases.
    \begin{enumerate}[label=\arabic*)]

        \item\label{it:case_mu} \ul{Suppose $v = v_i^{(\mu_s)}$}: recall that $c(v) = c(v_i^{(\mu_s)}) = \tilde{c}(v_i^{(\mu_s)}) = \tilde{c}(v)$.
        \begin{enumerate}[label=\alph*)]
            \item\label{it:subcase1_mu} \ul{If $r_N(v) = r_N'(v) = 0$}, then in the $N$ steps no slow-release was performed on vertex $v$, both on $c$ and $c'$.\\
            In both configurations $v$ receives $M_N + 1$ grains from slow-releases of the sink and other vertices, thus
            \[
                c_N(v) = c(v) + M_N + 1 = \tilde{c}(v) + M_N + 1 = \tilde{c}_N(v).
            \]
            \item\label{it:subcase2_mu} \ul{If $0 < r_N(v) = r_N'(v) < k$}, then in both configurations $v$ receives $M_N + 1 - r_N$ grains from the sink and other vertices and gives $(n-1)r_N(v) + 1$ grains, one to the sink (at the first slow-release) and $r_N(v)$ to each other vertex. Overall:
            \begin{align*}
                c_N(v) &= c(v) + M_N + 1 - r_N(v) - ((n-1)r_N(v) + 1) = c(v) + M_N - nr_N(v)\\
                &= \tilde{c}(v) + M_N - nr_N'(v) = \tilde{c}(v) + M_N + 1 - r_N'(v) - ((n-1)r_N'(v) + 1) = \tilde{c}_N(v).
            \end{align*}
            \item\label{it:subcase3_mu} \ul{If $r_N(v) = r_N'(v) = k$}: it satisfies the same formula of Subcase (1b).
        \end{enumerate}

        \item\label{it:case_nu} \ul{Suppose $v = v_j^{(\nu_s)}$}: recall that in this case $\tilde{c}(v) = \tilde{c}(v_j^{(\nu_s)}) = c(v_j^{(\nu_s)}) + (j-1) = c(v) + (j-1)$.\\
        To simplify notation, let $b_N := \#\{w = v_h^{(\nu_s)} \ | \ r_N(w) > 0\text{ and }w \neq v\}$ be the number of other vertices in the same clique of $v$ (the one corresponding to part $\nu_s$) on which the algorithm has already performed at least one slow-release. It tracks the number of slow-releases of type $\psi_{w}^{(k)}$ in which $v$ did not receive a grain from $w$ in the case of configuration $c$, since $G$ is the underlying graph\footnote{If $w$ is in the same $(k-1)$-clique component of $v$, at the first slow-release on $w$ no grain is sent to $v$ since the multiplicity of the edge $vw$ is $k-1 \not\geq k$.}.
        \begin{enumerate}[label=\alph*)]
            \item\label{it:subcase1_nu} \ul{If $r_N(v) = r_N'(v) = 0$}: then no slow-release is performed on $v$. This means that in $c$ the vertex $v$ received $M_N + 1 - b_N$ grains, while in $\tilde{c}$ it received $M_N + 1$ grains. In other words:
            \begin{align*}
                &c_N(v) = c(v) + M_N + 1 - b_N\\
                &\tilde{c}_N(v) = \tilde{c}(v) + M_N + 1 = c(v) + M_N + 1 + (j-1).
            \end{align*}
            \item\label{it:subcase2_nu} \ul{If $0 < r_N(v) = r_N'(v) < k$}: then at least one slow-release has been performed on $v$. Additionally, at the first slow-release grains are not sent to the other vertices of the clique $\nu_s$ of $v$. Overall:
            \begin{align*}
                c_N(v) &= c(v) + M_N + 1 - b_N - ((n-1)r_N(v) + 1 - (\nu_s - 1))\\
                &= c(v) + M_N - (n-1)r_N(v) + (\nu_s - 1) - b_N \\
                \tilde{c}_N(v) &= \tilde{c}(v) + M_N + 1 - ((n-1)r_N'(v) + 1)\\
                &= c(v) + M_N - (n-1)r_N'(v) + (j-1).
            \end{align*}
            \item\label{it:subcase3_nu} \ul{If $r_N(v) = r_N'(v) = k$}: the formula is analogous to the previous case.
        \end{enumerate}

    \end{enumerate}
    Finally, we are able to determine the next step of the algorithm for both configurations $c$ and $\tilde{c}$. We will consider the subdivision like in the previous argument and fix $v$ to be the vertex read at step $N+1$.
    \begin{enumerate}
        
        \item[\ref{it:case_mu}] Recall that in this case $\deg_G(v) = \deg_{\tilde{G}}(v) = (n-1)k + 1$. Thus:
        \begin{enumerate}
            \item[\ref{it:subcase1_mu}] since $c_N(v) = \tilde{c}_N(v)$ and the degrees coincide:
            \[
                u_{N+1} = u_{N+1}' = \begin{cases}
                    \bullet & \text{if $c_N(v) = \tilde{c}_N(v) < (n-1)k + 1$}\\
                    v & \text{otherwise}.
                \end{cases}
            \]
            \item[\ref{it:subcase2_mu}] in this case, since $0 < r_N(v) = r_N'(v) < k$ by design the algorithm performs another slow-release. Therefore:
            \[
                u_{N+1} = u_{N+1}' = v.
            \]
            \item[\ref{it:subcase3_mu}] since on every non-sink vertex we have applied at most $k$ slow-releases, we trivially have $M_N - nr_N(v) = M_N - nr_N'(v) = M_N - nk \leq 0$.\\
            This implies that $c_N(v) = \tilde{c}_N(v) = c(v) + M_N - nk \leq c(v)$. Since $c$ and $\tilde{c}$ were stable, this implies $v$ is stable in both $c_N$ and $\tilde{c}_N$. Thus:
            \[
                u_{N+1} = u_{N+1}' = \bullet.
            \]
        \end{enumerate}
        
        \item[\ref{it:case_nu}] We have that $\deg_{\tilde{G}}(v) = (n-1)k + 1$ but, since in $G$ we have $v$ is in a $(k-1)$-clique, $\deg_G(v) = (n-1)k + 1 - (\nu_s - 1) = (n-1)k - \nu_s + 2$.\\
        Before looking at the subcases, an important observation is that the vertices of the clique associated to $\nu_s$ are read in decreasing order with respect to their value in $c$ and $\tilde{c}$. Therefore, for $j_1 < j_2$ we perform\footnote{Recall that we ordered values of $c$ and $\tilde{c}$ so that $c(v_{j_1}^{(\nu_s)}) < c(v_{j_2}^{(\nu_s)})$ and $\tilde{c}(v_{j_1}^{(\nu_s)}) \leq \tilde{c}(v_{j_2}^{(\nu_s)})$.} the first slow-release of $v_{j_2}^{(\nu_s)}$ before the one of $v_{j_1}^{(\nu_s)}$.\\
        This implies that $r_N(v_{j_1}^{(\nu_s)}) \leq r_N(v_{j_2}^{(\nu_s)})$.
        \begin{enumerate}
            \item[\ref{it:subcase1_nu}] We look at the possible values of $u_{N+1}$.
            \begin{itemize}[label=$\triangleright$] 
                \item \ul{$u_{N+1} = v$}: in other words we are performing a slow-release on $v$ for the first time. By the observation above, all vertices $v_h^{(\nu_s)}$ with $h > j$ (which are exactly $\nu_s - j$) have been slow-released at least once. Therefore $b_N = \nu_s - j$.\\
                Moreover, to perform a slow-release on $v$ we must have that $v$ is unstable in $c_N$. These two observations combined imply:
                \begin{align*}
                    c(v) + M_N + 1 - b_N = c_N(v) &\geq \deg_G(v) = (n-1)k - \nu_s + 2\\
                    c(v) + M_N + 1 - (\nu_s - j) &\geq (n-1)k - \nu_s + 2\\
                    c(v) + M_N + 1 + (j-1) &\geq (n-1)k + 1\\
                    \tilde{c}(v) &\geq \deg_{\tilde{G}}(v)
                \end{align*}
                and the last inequality means that $v$ is unstable also in $\tilde{c}$ and therefore $u_{N+1}' = v$.
                \item \ul{$u_{N+1} = \bullet$}: since no slow-release has been performed on $v$, we have $b_N \leq \nu_s - j$.\\
                On the other hand, it also means that $v$ is stable in $c_N$ hence:
                \begin{align*}
                    c(v) + M_N + 1 - (\nu_s - j) \leq c(v) + M_N + 1 - b_N = c_N(v) &< (n-1)k - \nu_s + 2\\
                    c(v) + M_N + 1 + j &< (n-1)k + 2\\
                    c(v) + M_N + 1 + (j-1) &< (n-1)k + 1
                \end{align*}
                and the last relation represents exactly the stability of $v$ in $\tilde{c}_N$, so $u_{N+1}' = \bullet$.
            \end{itemize}
            \item[\ref{it:subcase2_nu}] By design of Algorithm~\ref{alg:toppl}, like in Subcase (1b) we have that the algorithm will perform another slow-release of vertex $v$. Therefore:
            \[
                u_{N+1} = u_{N+1}' = v.
            \]
            \item[\ref{it:subcase3_nu}] By the construction of the algorithm we have
            \[
                u_{N+1} = u_{N+1}' = \bullet.
            \]
        \end{enumerate}
    \end{enumerate}
    In every case we can check that $u_{N+1} = u_{N+1}'$. Therefore, by induction we have that for every $N$ (up to the stopping point of the algorithm) $\mathbf{u} = \mathbf{u}'$. Since one between $c$ and $\tilde{c}$ was recurrent, the algorithm terminates in both cases and by Remark~\ref{rem:test_rec} $c$ and $\tilde{c}$ are both recurrent.\\
    In particular, since $\wtopp(c)$ and $\wtopp(\tilde{c})$ are obtained from $\mathbf{u}$ and $\mathbf{u}'$ by erasing the ``$\bullet$'' symbols, we get $\wtopp(c) = \wtopp(\tilde{c})$.
\end{proof}

\begin{example}\label{exa:props_tildemap}
    We include in Appendix~\ref{app:comp_algo} two tables with the complete description of the $\del$ algorithm for the configurations $c$ and $\tilde{c}$ from Example~\ref{exa:tildemap}.
\end{example}

From this proposition we can show that the embedding map preserves both statistics $\lev$ and $\del$.

\begin{corollary}\label{cor:delaytilde}
    For all $c \in \SortRec_k(\mu;\nu)$ we have $\del(c) = \del(\tilde{c})$.
\end{corollary}
\begin{proof}
    The statistic $\del(c)$ depends only on $\wtopp(c)$. Thus, by Proposition~\ref{prop:tildemap} Point~\ref{it:tilde3} we have $\del(c) = \del(\tilde{c})$.
\end{proof}

\begin{lemma}\label{lem:leveltilde}
    For all $c \in \SortRec_k(\mu;\nu)$ we have $\lev(c) = \lev(\tilde{c})$.
\end{lemma}
\begin{proof}
    By construction, we have that the new edges in $G_{\mu\sqcup\nu, \varnothing}^{(k)}$ are exactly:
    \[
        \sum_{\nu_s \in \nu}\sum_{j = 0}^{\nu_s - 1}j = \sum_{\nu_s \in \nu}{\nu_s \choose 2}.
    \]
    We can compute:
    \begin{align*}
        \lev(\tilde{c}) &= -|E_0(G_{\mu\sqcup\nu,\varnothing}^{(k)})| + \sum_{i = 1}^{n}\tilde{c}(i)\\
        &= -\left(|E_0(G_{\mu,\nu}^{(k)})| + \sum_{\nu_s \in \nu}{\nu_s \choose 2}\right) + \left(\sum_{\mu_s \in \mu}\sum_{i = 1}^{\mu_s}\tilde{c}(v_i^{(\mu_s)}) + \sum_{\nu_s \in \nu}\sum_{j = 1}^{\nu_s}\tilde{c}(v_j^{(\nu_s)})\right)\\
        &= -\left(|E_0(G_{\mu,\nu}^{(k)})| + \sum_{\nu_s \in \nu}{\nu_s \choose 2}\right) + \left(\sum_{\mu_s \in \mu}\sum_{i = 1}^{\mu_s}c(v_i^{(\mu_s)}) + \sum_{\nu_s \in \nu}\sum_{j = 1}^{\nu_s}\big(c(v_j^{(\nu_s)}) + j - 1\big)\right)\\
        &= -\left(|E_0(G_{\mu,\nu}^{(k)})| + \sum_{\nu_s \in \nu}{\nu_s \choose 2}\right) + \left(\sum_{\mu_s \in \mu}\sum_{i = 1}^{\mu_s}c(v_i^{(\mu_s)}) + \sum_{\nu_s \in \nu}\sum_{j = 1}^{\nu_s}c(v_j^{(\nu_s)}) + \sum_{\nu_s \in \nu}{\nu_s \choose 2}\right)\\
        &= -\left(|E_0(G_{\mu,\nu}^{(k)})| + \sum_{\nu_s \in \nu}{\nu_s \choose 2}\right) + \left(\sum_{i = 1}^n c(i) + \sum_{\nu_s \in \nu}{\nu_s \choose 2}\right) = \lev(c).
    \end{align*}
    This concludes the proof.
\end{proof}

\section{The bijection with labelled Dyck paths}\label{sec:bijSortedDyck}

In this section we will define the bijection between sorted recurrent configurations on $G_{\mu,\nu}^{(k)}$ and labelled Dyck paths in $\PF_{n,kn}(\mu;\nu)$. Our construction can be subdivided into 3 steps:
\begin{enumerate}[label=\roman*)]
    \item define a bijection $\Gamma_0: \SortRec_k((1^n);\varnothing) \to \PF_{n,kn}((1^n);\varnothing)$ that sends statistics $(\lev, \del)$ into $(\area,\pmaj)$,
    \item show that the bijection $\Gamma_0$ restricts to the subsets $\SortRec_k(\mu\sqcup\nu;\varnothing)$ and $\PF_{n,kn}(\mu;\nu)$,
    \item compose $\Gamma_0$ with the embedding map to obtain a bijection $\Gamma: \SortRec_k(\mu;\nu) \to \PF_{n,kn}(\mu;\nu)$ defined by $\Gamma(c) := \Gamma_0(\tilde{c})$.
\end{enumerate}

The first two steps follow from the next result.
\begin{proposition}\label{prop:gamma0map}
    The map given by
    \def\arraystretch{1.5}
    \[\begin{array}[t]{cccc}
        \Gamma_0: & \SortRec_k((1^n);\varnothing) & \longrightarrow & \PF_{n,kn}((1^n);\varnothing)\\
         & c & \longmapsto &    {\def\arraystretch{1.2}
                                \begin{array}[t]{cccc}
                                    \pi: & [n] & \to & [n]\\
                                        & i & \mapsto & (n-1)k - c(i) + 1
                                \end{array}}
    \end{array}\]
    is well defined and bijective.\\
    Moreover, the map $\Gamma_0$ restricts to a bijection $\SortRec_k(\mu\sqcup\nu;\varnothing) \to \PF_{n,kn}(\mu;\nu)$ for all compositions $\mu$ and $\nu$ with $|\mu| + |\nu| = n$.
\end{proposition}

\begin{example}\label{exa:gamma0}
    Consider the configuration $\tilde{c} = {\color{mycolor1} 6}\ {\color{mycolor1} 17}\ {\color{mycolor1} 18}\ {\color{mycolor1} 24}\ {\color{mycolor2} 2}\ {\color{mycolor3} 18}\ {\color{mycolor3} 14}\ {\color{mycolor4} 23}\ {\color{mycolor4} 22}\ {\color{mycolor4} 18}\ {\color{mycolor4} 14}\ {\color{mycolor4} 8}\ {\color{mycolor4} 2} \in \SortRec_2(({\color{mycolor1}4},{\color{mycolor2}1})\sqcup({\color{mycolor4}6}, {\color{mycolor3}2}); \varnothing)$ from Example~\ref{exa:tildemap}. We can compute the image $\pi := \Gamma_0(\tilde{c})$ to be:
    \[\begin{array}{lll}
        \pi({\color{mycolor4} 1})  = 24 - \tilde{c}({\color{mycolor4}  1}) + 1 = 23\quad &
        \pi({\color{mycolor4} 2})  = 24 - \tilde{c}({\color{mycolor4}  2}) + 1 = 17\quad &
        \pi({\color{mycolor4} 3})  = 24 - \tilde{c}({\color{mycolor4}  3}) + 1 = 11 \\
        \pi({\color{mycolor4} 4})  = 24 - \tilde{c}({\color{mycolor4}  4}) + 1 =  7\quad &
        \pi({\color{mycolor4} 5})  = 24 - \tilde{c}({\color{mycolor4}  5}) + 1 =  3\quad &
        \pi({\color{mycolor4} 6})  = 24 - \tilde{c}({\color{mycolor4}  6}) + 1 =  2 \\
        \pi({\color{mycolor3} 7})  = 24 - \tilde{c}({\color{mycolor3}  7}) + 1 = 11\quad &
        \pi({\color{mycolor3} 8})  = 24 - \tilde{c}({\color{mycolor3}  8}) + 1 =  7\quad &
        \pi({\color{mycolor2} 9})  = 24 - \tilde{c}({\color{mycolor2}  9}) + 1 = 23 \\
        \pi({\color{mycolor1} 10}) = 24 - \tilde{c}({\color{mycolor1} 10}) + 1 =  1\quad &
        \pi({\color{mycolor1} 11}) = 24 - \tilde{c}({\color{mycolor1} 11}) + 1 =  7\quad &
        \pi({\color{mycolor1} 12}) = 24 - \tilde{c}({\color{mycolor1} 12}) + 1 =  8 \\ &
        \pi({\color{mycolor1} 13}) = 24 - \tilde{c}({\color{mycolor1} 13}) + 1 = 19.\quad &
    \end{array}\]
    If we represent $\pi$ as a path in the $12 \times 24$ grid, we actually obtain the labelled Dyck path in Figure~\ref{fig:gamma0}.
    \begin{figure}[bhp]
        \begin{tikzpicture}
            \parkfunc{13}{2}{26,25,24,20,20,20,19,16,16,10,8,4,4}{{\color{mycolor1}10}, {\color{mycolor4}6}, {\color{mycolor4}5}, {\color{mycolor4}4}, {\color{mycolor3}8}, {\color{mycolor1}11}, {\color{mycolor1}12}, {\color{mycolor4}3}, {\color{mycolor3}7}, {\color{mycolor4}2}, {\color{mycolor1}13}, {\color{mycolor4}1}, {\color{mycolor2}9}}{.6}
        \end{tikzpicture}
        \caption{The diagram associated to $\pi := \Gamma_0(\tilde{c})$.}\label{fig:gamma0}
    \end{figure}
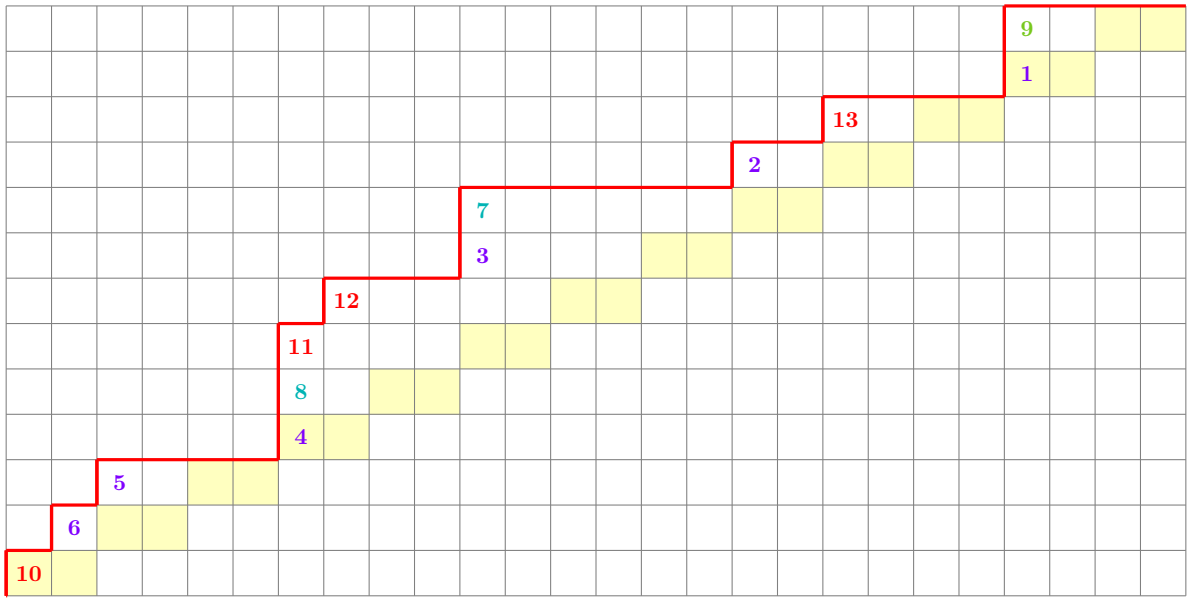\\
    Also, observe that the labelled Dyck path belongs to the subset determined by the same compositions of the configuration $\tilde{c}$: $\pi \in \PF_{13,26}(({\color{mycolor1}4},{\color{mycolor2}1}),({\color{mycolor4}6}, {\color{mycolor3}2}))$.
\end{example}

\begin{proof}[Proof Proposition~\ref{prop:gamma0map}]
    Fix $c \in \SortRec_k((1^n);\varnothing)$.

    To show that $\Gamma_0$ is well defined, we prove that $\Gamma_0(c)$ is a labelled Dyck path.\\
    Since $c$ is non-negative, $\big(\Gamma_0(c)\big)(v) = (n-1)k - c(v) \leq (n-1)k$ for all $v \in [n]$ and using that $c$ is stable we have $c(v) < (n-1)k + 1$, thus $\big(\Gamma_0(c)\big)(v) \geq 1$. The two inequalities combined imply that $\Gamma_0(c)$ has image in $[nk] = \{1,2,\dots,nk\}$.\\
    Moreover, $\Gamma_0(c)$ lies ``above the diagonal''. This property follows from Lemma~\ref{lem:rec_kcomplete}: suppose $i_1,i_2,\dots,i_n$ is a toppling order for $c$ recurrent, then for $j \in [n]$
    \begin{align*}
        c(i_j) \geq (n-j)k + 1 \qquad \Longrightarrow \qquad \big(\Gamma_0(c)\big)(i_j) = (n-1)k - c(i_j) + 1 \leq (j-1)k + 1
    \end{align*}
    and thus $\#\big(\Gamma_0(c)\big)^{-1}\big([(j-1)k+1]\big) \geq \#\{i_1,\dots,i_j\} = j$.
    \smallskip

    To show bijectivity, we can explicitly construct the inverse map.\\
    Fix a parking function $\pi \in \PF_{n,kn}((1^n);\varnothing)$, we associate it to the configuration $c \in \Conf(G_{(1^n),\varnothing}^{(k)})$ defined by $c(i) := (n-1)k - \pi(i) + 1$. One can easily prove that $c$ is a recurrent configuration by applying the same inequalities above and Lemma~\ref{lem:rec_kcomplete}.
    By construction, this map is the inverse of $\Gamma_0$. This proves bijectivity.\medskip

    To show that $\Gamma_0$ restricts to a map $\SortRec_k(\mu\sqcup\nu;\varnothing) \to \PF_{n,kn}(\mu;\nu)$, we prove that for any sorted recurrent configuration $c \in \SortRec_k(\mu,\nu)$:
    \begin{equation}\label{eq:shuffle}
        \wrow(\Gamma_0(c)) \in K_{\mu,1} \shuffle \dots \shuffle K_{\mu,\ell(\mu)} \shuffle I_{\nu,1} \shuffle \dots \shuffle I_{\nu,\ell(\nu)}.
    \end{equation}
    Consider $i_1,i_2 \in K_{\mu,j}$ with $i_1 < i_2$ two vertices in a $k$-clique associated to part $\mu_j$. By definition we have $c(i_1) \geq c(i_2)$, thus
    \[
        \big(\Gamma_0(c)\big)(i_1) = (n-1)k + 1 - c(i_1) \leq (n-1)k + 1 - c(i_2) = \big(\Gamma_0(c)\big)(i_2)
    \]
    which means $i_1$ precedes $i_2$ in $\wrow(\Gamma_0(c))$: either the inequality is strict ($i_1$ appears to the east of $i_2$ in the labelled Dyck path) or the labels are increasingly ordered when equality holds ($i_1$ and $i_2$ are in the same column).\\
    Given $i_1,i_2 \in I_{\nu,j}$ with $i_1 < i_2$ by hypothesis $c(i_1) < c(i_2)$. Therefore
    \[
        \big(\Gamma_0(c)\big)(i_1) = (n-1)k + 1 - c(i_1) > (n-1)k + 1 - c(i_2) = \big(\Gamma_0(c)\big)(i_2)
    \]
    and $i_2$ precedes $i_1$ in $\wrow(\Gamma_0(c))$.\\
    Varying $i_1,i_2$ between all consecutive vertices in the same component, we have shown~\eqref{eq:shuffle}.

    For the inverse map the corresponding result holds by an analogous argument that we omit here.
\end{proof}

This bijection preserves the values of the corresponding statistics.
\begin{proposition}\label{prop:gamma0prop}
    Given a configuration $c \in \SortRec_k(\mu\sqcup\nu; \varnothing)$ we have that:
    \begin{equation}\label{eq:prop1gamma0}
        \lev(c) = \area(\Gamma_0(c)),
    \end{equation}
    and $\wtopp(c) = \wpmaj(\Gamma_0(c))$, so in particular
    \begin{equation}\label{eq:prop2gamma0}
        \del(c) = \pmaj(\Gamma_0(c)).
    \end{equation}
\end{proposition}
\begin{proof}
    Fix a sorted recurrent configuration $c \in \SortRec_k(\mu\sqcup\nu; \varnothing)$ and let $\pi := \Gamma_0(c)$.
    \smallskip

    To show Equation~\eqref{eq:prop1gamma0}, recall that the area statistic of a labelled Dyck path is given by the number of squares between the path and the main diagonal. Suppose $\warea(\pi) = w_1w_2\dots w_n$, then for $i \in [n]$:
    \begin{align*}
        w_i &= \#\text{ of squares in $i^{\text{th}}$ row to the east of the NORTH step} - \\
        &\qquad\qquad\qquad - \#\text{ of squares to the east of the main diagonal in row $i$}\\
        &= (nk - \pi(i) + 1) - (n-i+1)k = (i-1)k + 1 - \pi(i).
    \end{align*}
    Therefore we can compute:
    \begin{align*}
        \area(\pi) &= w_1 + \dots + w_n = {n \choose 2}k + n - \sum_{i \in [n]}\pi(i)\\
        &= {n \choose 2}k + n - \sum_{i \in [n]}\big((n-1)k - c(i) + 1\big)\\
        &= {n \choose 2}k + n - n(n-1)k + \sum_{i \in [n]}c(i) - n\\
        &= - \left(n(n-1) - {n \choose 2}\right)k + \sum_{i \in [n]}c(i)\\
        &= - {n \choose 2}k + \sum_{i \in [n]}c(i) = \lev(c)
    \end{align*}
    where recall that in the graph $G_{\mu \sqcup \nu, \varnothing}^{(k)}$ the number of edges non incident to the sink is ${n \choose 2}k$.
    \smallskip

    For Equation~\eqref{eq:prop2gamma0}, we compare the algorithms that compute $\pmaj$ and $\del$ (cf.~respectively Definition~\ref{def:gen_pmaj} and Algorithm~\ref{alg:toppl}).\\
    The $\pmaj$ is computed by constructing the multisets $\{B_i\}_{i \in [nk]}$ recursively, adding at each step $k$ copies of the labels of north steps at $x$-coordinate $i$ and removing one copy of the maximum label lower than the previous one in $\wpmaj(\pi)$.\\
    Analogously, we can define multisets $\{C_i\}_{i \in [nk]}$ for the $\del$ algorithm as follows:
    \[
        C_0 := \big\{\!\!\big\{ \underbrace{v, v,\dots,v}_{\text{$k$ times}} \ |\ v\text{ is unstable in $\phi_0(c)$} \big\}\!\!\big\}
    \] 
    and define $C_i$ by removing one copy of the vertex slow-released at step $i-1$ and adding $k$ copies of the new unstable vertices. In other words, $C_i$ records what vertices are unstable at step $i-1$ and how many slow-releases still have to be performed on the partially toppled vertices.\\
    We show that the toppling word $\wtopp(c)$ and the parking word $\wpmaj(\pi)$ coincide by proving that for every $i \in [nk]$ we have $B_i = C_i$.\\
    We do an induction of the index $i \in [nk]$:
    \begin{itemize}
        \item by construction $B_0 = C_0$. Indeed, after toppling the sink, the unstable vertices in $c$ are $v \in [n]$ such that $c(v) = (n-1)k$. Applying map $\Gamma_0$, these are the labels $v \in [n]$ such that $\pi(v) = 1$, exactly the ones that get read at the first iteration of the $\pmaj$ algorithm.
        \item suppose that $B_{i-1} = C_{i-1}$. A vertex $v \in [n]$ gets added to the multiset $C_i$ if and only if $v$ becomes unstable after the $(i-1)^{\text{th}}$ slow release. In other words:
        \[
            c(v) + 1 + (i-1) = \deg_G(v) = (n-1)k + 1.
        \]
        Applying $\Gamma_0$, this means that $\pi(v) = (n-1)k - c(v) + 1 = i$, so the vertex $v$ gets added to $B_i$.\\
        Moreover, the same vertex $v$ is removed from both $B_{i-1}$ and $C_{i-1}$. Let $w$ be the vertex removed in the previous step from $B_{i-2} = C_{i-2}$ (or let $w = 1$ if $i = 1$).\\
        In Algorithm~\ref{alg:toppl} the removed vertex $v$ is the first element of $C_{i-1}$ found by reading in decreasing cyclical order the vertices $[n]$ starting from $w$. An equivalent method to find $v$ is to pick the highest value in $C_{i-1} = B_{i-1}$ lower than $w$ or, if there is none such value, the maximal value of $C_{i-1} = B_{i-1}$. This is the same criteria used to pick the element that gets removed from $B_{i-1}$ in the definition of the $\pmaj$ statistic.\\
        We therefore conclude that $B_i = C_i$, since they are constructed by adding and removing from $B_{i-1} = C_{i-1}$ the same elements.
    \end{itemize}
    This concludes the proof.
\end{proof}

Finally, we can construct the desired statistic-preserving bijection.
\begin{definition}\label{def:main_bij}
    For any fixed $\mu,\nu$ compositions such that $|\mu| + |\nu| = n$ let:
    \[
        \begin{array}{cccc}
            \Gamma: & \SortRec_k(\mu;\nu) & \longrightarrow & \PF_{n,kn}(\mu;\nu)\\
                & c & \longmapsto & \Gamma_0(\tilde{c})
        \end{array}.
    \]
\end{definition}

The main properties of this map have already been discussed. We collect them in the following theorem.
\begin{theorem}\label{thm:bijPFSort}
    Consider $n,k \in \mathbb{N} \setminus \{0\}$ and two compositions $\mu,\nu$ such that $|\mu| + |\nu| = n$. Then the map:
    \[
        \Gamma: \SortRec_k(\mu;\nu) \longrightarrow \PF_{n,kn}(\mu;\nu)
    \]
    is a bijection and it preserves the corresponding statistics:
    \[
        \lev(c) = \area(\Gamma(c)) \qquad \text{and} \qquad \del(c) = \pmaj(\Gamma(c))
    \]
    for every $c \in \SortRec_k(\mu;\nu)$.
\end{theorem}
\begin{proof}
    The result follows immediately from the fact that the map $\Gamma$ is the composition of statistics-preserving bijections.\\
    These properties are shown in Proposition~\ref{prop:tildemap}, Corollary~\ref{cor:delaytilde} and Lemma~\ref{lem:leveltilde} for the embedding map and in
    Propositions~\ref{prop:gamma0map} and~\ref{prop:gamma0prop} for map $\Gamma_0$.
\end{proof}

\begin{example}\label{exa:gamma}
    Consider the sorted recurrent configuration $c \in \SortRec_2(({\color{mycolor1}4},{\color{mycolor2}1});({\color{mycolor4}6}, {\color{mycolor3}2}))$ from the running Example~\ref{exa:toppling}.\\
    In Example~\ref{exa:tildemap} we have showed that, through the embedding map, $c$ gets mapped to the configuration:
    \[
        \tilde{c} = {\color{mycolor1} 6}\ {\color{mycolor1} 17}\ {\color{mycolor1} 18}\ {\color{mycolor1} 24}\ {\color{mycolor2} 2}\ {\color{mycolor3} 18}\ {\color{mycolor3} 14}\ {\color{mycolor4} 23}\ {\color{mycolor4} 22}\ {\color{mycolor4} 18}\ {\color{mycolor4} 14}\ {\color{mycolor4} 8}\ {\color{mycolor4} 2} \in \Sort(({\color{mycolor1}4},{\color{mycolor2}1})\sqcup({\color{mycolor4}6}, {\color{mycolor3}2}),\varnothing).
    \]
    Let $\pi := \Gamma(c) = \Gamma_0(\tilde{c})$ be the labelled Dyck path obtained in Example~\ref{exa:gamma0}, Figure~\ref{fig:gamma0}: it belongs to $\PF_{13,26}(({\color{mycolor1}4},{\color{mycolor2}1});({\color{mycolor4}6}, {\color{mycolor3}2}))$.\\
    One can also compute the respective statistics, obtaining:
    \[
        \area(\pi) = \lev(c) = 30 \qquad \text{and} \qquad \dinv(\pi) = \del(c) = 28.
    \]
\end{example}

This result, applied to Corollary~\ref{cor:nablaken_eh_pmaj}, gives a new interpretation of the scalar product $\langle \nabla^k e_n, e_\mu h_\nu \rangle$ in terms of sorted recurrent configurations. It is a direct generalization of the main result obtained in~\cite{DDILLV}.

\begin{corollary}\label{cor:identitydelay}
    Consider $n,k \in \mathbb{N} \setminus \{0\}$ and two compositions $\mu,\nu$ such that $|\mu| + |\nu| = n$. Then:
    \[
        \langle \nabla^k e_n, e_\mu h_\nu \rangle = \sum_{c \in \SortRec_k(\mu;\nu)}q^{\lev(c)}t^{\del(c)}.
    \]
\end{corollary}

\bibliographystyle{amsalpha}
\bibliography{refs}

\newpage

\appendix
\section{Possible diagonal inversion contributes between pairs of labels}\label{app:dinvtable}

Consider $\lambda,\mu \in [n]$ distinct labels of a labelled Dyck path $D \in \LDyck_{n,kn}$. Fix the area contribute difference $m:=a_{\lambda}(D) - a_{\mu}(D) > 0$.\smallskip\\
We collect the possible diagonal inversion contributes of respectively $\dinvcorr$ and $\tdinv$ in the following tables. They are the key to comprehend what happens to the values of $\phi_n^{(k)}{f_D}$ when we modify the values of $f_D$, since the bijection itself relies on the $\dinv$ contribute of each label in the labelled Dyck path.

\medskip

In the first table, the contributes of $\dinvcorr$ of two north steps labelled $\lambda$ and $\mu$ are compared when the only thing changing is the relative east/west position of the steps.

\begin{table}[htbp!]
	\centering
	\begin{tabular}[t]{lcc}
		\hline
		\multicolumn{3}{c}{\textbf{\mathversion{bold}Dinv Correction ($\dinvcorr$)}}\\
		\hline
		\textbf{Area difference} & \textbf{\mathversion{bold}Case $\lambda$ left of $\mu$} & \textbf{\mathversion{bold}Case $\mu$ left of $\lambda$}\\
		$m:=a_\lambda(D)-a_{\mu}(D)$& $f_D(\lambda) < f_D(\mu)$ & $f_D(\lambda) > f_D(\mu)$ \\
		\hline
		\raisebox{2cm}{
			\begin{tabular}{l}(A) $m = 0$\\ \\ Example: $\begin{array}{l}k = 5\\ m = 0\end{array}$\end{tabular}} &
		\begin{tikzpicture}[scale = .7]               
			\draw[green!60!black, dotted] (0,0) -- (6,3);
			\draw[green!60!black, dashed, fill=green!60!black, fill opacity=.2] (0,0.2) -- (6,3.2) -- (6,3.4) -- (0,0.4) -- cycle;
			\draw[green!60!black, dashed, fill=green!60!black, fill opacity=.2] (0,0.4) -- (6,3.4) -- (6,3.6) -- (0,0.6) -- cycle;
			\draw[green!60!black, dashed, fill=green!60!black, fill opacity=.2] (0,0.6) -- (6,3.6) -- (6,3.8) -- (0,0.8) -- cycle;
			\draw[green!60!black, dashed, fill=green!60!black, fill opacity=.2] (0,0.8) -- (6,3.8) -- (6,4) -- (0,1) -- cycle;
			\draw[green!60!black, dashed, opacity=0, fill=orange!60!black, fill opacity=.15] (0,0.2) -- (6,3.2) -- (6,3) -- (0,0) -- cycle;
			\draw[green!60!black] (0,1) -- (6,4);
			\draw[red, very thick] (6,3) -- (6,4) node[at end, above left] {$\mu$};
			\draw[red, very thick] (2,1) -- (2,2) node[at end, above left] {$\lambda$};     
			\node[scale=.8] at (5,1) {$k-1\ \dinvcorr$};
		\end{tikzpicture}
		&
		\begin{tikzpicture}[scale = .7]
			\draw[green!60!black, dotted] (0,0) -- (6,3);
			\draw[green!60!black, dashed, fill=green!60!black, fill opacity=.2] (0,0.2) -- (6,3.2) -- (6,3.4) -- (0,0.4) -- cycle;
			\draw[green!60!black, dashed, fill=green!60!black, fill opacity=.2] (0,0.4) -- (6,3.4) -- (6,3.6) -- (0,0.6) -- cycle;
			\draw[green!60!black, dashed, fill=green!60!black, fill opacity=.2] (0,0.6) -- (6,3.6) -- (6,3.8) -- (0,0.8) -- cycle;
			\draw[green!60!black, dashed, fill=green!60!black, fill opacity=.2] (0,0.8) -- (6,3.8) -- (6,4) -- (0,1) -- cycle;
			\draw[green!60!black, dashed, opacity=0, fill=orange!60!black, fill opacity=.15] (0,0.2) -- (6,3.2) -- (6,3) -- (0,0) -- cycle;
			\draw[green!60!black] (0,1) -- (6,4);
			\draw[red, very thick] (6,3) -- (6,4) node[at end, above left] {$\lambda$};
			\draw[red, very thick] (2,1) -- (2,2) node[at end, above left] {$\mu$};     
			\node[scale=.8] at (5,1) {$k-1\ \dinvcorr$};
		\end{tikzpicture}\\ \hline
		
		\raisebox{2cm}{
			\begin{tabular}{l}(B) $0 < m < k$\\ \\ Example: $\begin{array}{l}k = 5\\ m = 2\end{array}$\end{tabular}} &
		\begin{tikzpicture}[scale = .7]               
			\draw[green!60!black, dotted] (0,0) -- (6,3);
			\draw[green!60!black, dashed, fill=green!60!black, fill opacity=0] (0,0.2) -- (6,3.2) -- (6,3.4) -- (0,0.4) -- cycle;
			\draw[green!60!black, dashed, fill=green!60!black, fill opacity=.2] (0,0.4) -- (6,3.4) -- (6,3.6) -- (0,0.6) -- cycle;
			\draw[green!60!black, dashed, fill=green!60!black, fill opacity=.2] (0,0.6) -- (6,3.6) -- (6,3.8) -- (0,0.8) -- cycle;
			\draw[green!60!black, dashed, fill=green!60!black, fill opacity=.2] (0,0.8) -- (6,3.8) -- (6,4) -- (0,1) -- cycle;
			\draw[green!60!black] (0,1) -- (6,4);
			\draw[red, very thick] (6,3) -- (6,4) node[at end, above left] {$\mu$};
			\draw[red, very thick] (2,1.4) -- (2,2.4) node[at end, above left] {$\lambda$};     
			\node[scale=.8] at (5,1) {$k-m\ \dinvcorr$};
		\end{tikzpicture} & 
		\begin{tikzpicture}[scale = .7]
			\draw[green!60!black, dotted] (0,0) -- (6,3);
			\draw[green!60!black, dashed, fill=green!60!black, fill opacity=.2] (0,0.2) -- (6,3.2) -- (6,3.4) -- (0,0.4) -- cycle;
			\draw[green!60!black, dashed, fill=green!60!black, fill opacity=.2] (0,0.4) -- (6,3.4) -- (6,3.6) -- (0,0.6) -- cycle;
			\draw[green!60!black, dashed, fill=green!60!black, fill opacity=0] (0,0.6) -- (6,3.6) -- (6,3.8) -- (0,0.8) -- cycle;
			\draw[green!60!black, dashed, fill=green!60!black, fill opacity=0] (0,0.8) -- (6,3.8) -- (6,4) -- (0,1) -- cycle;
			\draw[green!60!black, dashed, opacity=0, fill=orange!60!black, fill opacity=.15] (0,0.2) -- (6,3.2) -- (6,3) -- (0,0) -- cycle;
			\draw[green!60!black] (0,1) -- (6,4);
			\draw[red, very thick] (6,3) -- (6,4) node[at end, above left] {$\lambda$};
			\draw[red, very thick] (2,.6) -- (2,1.6) node[at end, above left] {$\mu$};       
			\node[scale=.8] at (5,1) {$k-m-1\ \dinvcorr$};
		\end{tikzpicture}\\ \hline
		
		\raisebox{2cm}{
			\begin{tabular}{l}(C + D) $m \geq k$\\ \\ Example: $\begin{array}{l}k = 5\\ m = 5\end{array}$\end{tabular}} &
		\begin{tikzpicture}[scale = .7]               
			\draw[green!60!black, dotted] (0,0) -- (6,3);
			\draw[green!60!black, dashed, fill=green!60!black, fill opacity=0] (0,0.2) -- (6,3.2) -- (6,3.4) -- (0,0.4) -- cycle;
			\draw[green!60!black, dashed, fill=green!60!black, fill opacity=0] (0,0.4) -- (6,3.4) -- (6,3.6) -- (0,0.6) -- cycle;
			\draw[green!60!black, dashed, fill=green!60!black, fill opacity=0] (0,0.6) -- (6,3.6) -- (6,3.8) -- (0,0.8) -- cycle;
			\draw[green!60!black, dashed, fill=green!60!black, fill opacity=0] (0,0.8) -- (6,3.8) -- (6,4) -- (0,1) -- cycle;
			\draw[green!60!black] (0,1) -- (6,4);
			\draw[red, very thick] (6,3) -- (6,4) node[at end, above left] {$\mu$};
			\draw[red, very thick] (2,2) -- (2,3) node[at end, above left] {$\lambda$};       
			\node[scale=.8] at (5,1) {$0\ \dinvcorr$};
		\end{tikzpicture} &
		\begin{tikzpicture}[scale = .7]
			\draw[green!60!black, dotted] (0,0) -- (6,3);
			\draw[green!60!black, dashed, fill=green!60!black, fill opacity=0] (0,0.2) -- (6,3.2) -- (6,3.4) -- (0,0.4) -- cycle;
			\draw[green!60!black, dashed, fill=green!60!black, fill opacity=0] (0,0.4) -- (6,3.4) -- (6,3.6) -- (0,0.6) -- cycle;
			\draw[green!60!black, dashed, fill=green!60!black, fill opacity=0] (0,0.6) -- (6,3.6) -- (6,3.8) -- (0,0.8) -- cycle;
			\draw[green!60!black, dashed, fill=green!60!black, fill opacity=0] (0,0.8) -- (6,3.8) -- (6,4) -- (0,1) -- cycle;
			\draw[green!60!black] (0,1) -- (6,4);
			\draw[red, very thick] (6,3) -- (6,4) node[at end, above left] {$\lambda$};
			\draw[red, very thick] (2,0) -- (2,1) node[at end, above left] {$\mu$};        
			\node[scale=.8] at (5,1) {$0\ \dinvcorr$};
		\end{tikzpicture}\\
		\hline
	\end{tabular}
	\caption{Possible diagonal inversion correction contributes of two labels.}\label{tab:poss_dinvcorrs}
\end{table}

\newpage

In Table~\ref{tab:poss_tdinvs} the values of $\tdinv$ are compared with a similar criterion. However, in this case the statistic depends on the value of the labels (whether $\lambda < \mu$ or $\mu < \lambda$) so there is an additional subdivision in cases.

\begin{table}[htbp!]
	\centering
	\begin{tabular}{lcc}
		\hline
		\multicolumn{3}{c}{\textbf{\mathversion{bold}Temporary Dinv ($\tdinv$)}}\\
		\hline
		\textbf{Area difference} & \textbf{\mathversion{bold}Case $\lambda$ left of $\mu$} & \textbf{\mathversion{bold}Case $\mu$ left of $\lambda$}\\
		$m:=a_\lambda(D)-a_{\mu}(D)$& $f_D(\lambda) < f_D(\mu)$ & $f_D(\lambda) > f_D(\mu)$ \\
		\hline
		\raisebox{2cm}{
			(A) $m = 0$} & \begin{tikzpicture}[scale = .65]               
			\node[opacity=0] at (1,5.05) {\textbullet}; 
			\draw[green!60!black, dashed] (0,0) -- (6,3);
			\draw[green!60!black] (0,1) -- (6,4);
			\draw[orange!80!black, dashed] (6,4.05) -- (6,5);
			\draw[orange!80!black] (0,2) -- (6,5);
			\draw[red, very thick] (6,3) -- (6,4) node[midway, right] {$\mu$} node[at end] {\textbullet};
			\draw[red, very thick] (2,1) -- (2,2) node[midway, left] {$\lambda$} node[at end] {\textbullet};
			\node[scale=.8] at (5,0) {$\begin{array}{ccc}
					\text{(A1): }\ \lambda < \mu & \Longrightarrow & 1\ \tdinv\\
					\text{(A2): }\ \lambda > \mu & \Longrightarrow & 0\ \tdinv
				\end{array}$};
		\end{tikzpicture} & \begin{tikzpicture}[scale = .65]
			\node[opacity=0] at (1,5.05) {\textbullet}; 
			\draw[green!60!black, dashed] (0,0) -- (6,3);
			\draw[green!60!black] (0,1) -- (6,4);
			\draw[orange!80!black, dashed] (6,4.05) -- (6,5);
			\draw[orange!80!black] (0,2) -- (6,5);
			\draw[red, very thick] (6,3) -- (6,4) node[midway, right] {$\lambda$} node[at end] {\textbullet};
			\draw[red, very thick] (2,1) -- (2,2) node[midway, left] {$\mu$} node[at end] {\textbullet};
			\node[scale=.8] at (5,0) {$\begin{array}{ccc}
					\text{(A1): }\ \lambda < \mu & \Longrightarrow & 0\ \tdinv\\
					\text{(A2): }\ \lambda > \mu & \Longrightarrow & 1\ \tdinv
				\end{array}$};
		\end{tikzpicture}\\ \hline
		
		\raisebox{2cm}{
			(B) $0 < m < k$} & \begin{tikzpicture}[scale = .65]               
			\node[opacity=0] at (1,5.05) {\textbullet}; 
			\draw[green!60!black, dashed] (0,0) -- (6,3);
			\draw[green!60!black] (0,1) -- (6,4);
			\draw[orange!80!black, dashed] (6,4.05) -- (6,5);
			\draw[orange!80!black] (0,2) -- (6,5);
			\draw[red, very thick] (6,3) -- (6,4) node[midway, right] {$\mu$} node[at end] {\textbullet};
			\draw[red, very thick] (2,1.6) -- (2,2.6) node[midway, left] {$\lambda$} node[at end] {\textbullet};
			\node[scale=.8] at (5,0) {$\begin{array}{ccc}
					\text{(B1): }\ \lambda < \mu & \Longrightarrow & 0\ \tdinv\\
					\text{(B2): }\ \lambda > \mu & \Longrightarrow & 1\ \tdinv
				\end{array}$}; 
		\end{tikzpicture} & \begin{tikzpicture}[scale = .65]
			\node[opacity=0] at (1,5.05) {\textbullet}; 
			\draw[green!60!black, dashed] (0,0) -- (6,3);
			\draw[green!60!black] (0,1) -- (6,4);
			\draw[orange!80!black, dashed] (6,4.05) -- (6,5);
			\draw[orange!80!black] (0,2) -- (6,5);
			\draw[red, very thick] (6,3) -- (6,4) node[midway, right] {$\lambda$} node[at end] {\textbullet};
			\draw[red, very thick] (2,.6) -- (2,1.6) node[midway, left] {$\mu$} node[at end] {\textbullet};    
			\node[scale=.8] at (5,0) {$\begin{array}{ccc}
					\text{(B1): }\ \lambda < \mu & \Longrightarrow & 0\ \tdinv\\
					\text{(B2): }\ \lambda > \mu & \Longrightarrow & 1\ \tdinv
				\end{array}$};
		\end{tikzpicture}\\ \hline
		
		\raisebox{2cm}{
			(C) $m = k$} & \begin{tikzpicture}[scale = .65]               
			\node[opacity=0] at (1,5.05) {\textbullet}; 
			\draw[green!60!black, dashed] (0,0) -- (6,3);
			\draw[green!60!black] (0,1) -- (6,4);
			\draw[orange!80!black, dashed] (6,4.05) -- (6,5);
			\draw[orange!80!black] (0,2) -- (6,5);
			\draw[red, very thick] (6,3) -- (6,4) node[midway, right] {$\mu$} node[at end] {\textbullet};
			\draw[red, very thick] (2,2) -- (2,3) node[midway, left] {$\lambda$} node[at end] {\textbullet};       
			\node[scale=.8] at (5,0) {$\begin{array}{ccc}
					\text{(C1): }\ \lambda < \mu & \Longrightarrow & 0\ \tdinv\\
					\text{(C2): }\ \lambda > \mu & \Longrightarrow & 1\ \tdinv
				\end{array}$};
		\end{tikzpicture} & \begin{tikzpicture}[scale = .65]
			\node[opacity=0] at (1,5.05) {\textbullet}; 
			\draw[green!60!black, dashed] (0,0) -- (6,3);
			\draw[green!60!black] (0,1) -- (6,4);
			\draw[orange!80!black, dashed] (6,4.05) -- (6,5);
			\draw[orange!80!black] (0,2) -- (6,5);
			\draw[red, very thick] (6,3) -- (6,4) node[midway, right] {$\lambda$} node[at end] {\textbullet};
			\draw[red, very thick] (2,0) -- (2,1) node[midway, left] {$\mu$} node[at end] {\textbullet};                
			\node[scale=.8, opacity=0] at (5,0) {$\begin{array}{ccc}
					\text{(C1): }\ \lambda < \mu & \Longrightarrow & 0\ \tdinv\\
					\text{(C2): }\ \lambda > \mu & \Longrightarrow & 1\ \tdinv
				\end{array}$};
			\node[scale=.8] at (5,0) {Not Attacking};
		\end{tikzpicture}\\ \hline
		
		\raisebox{2cm}{
			(D) $m > k$} & \begin{tikzpicture}[scale = .65]               
			\node[opacity=0] at (1,5.05) {\textbullet}; 
			\draw[green!60!black, dashed] (0,0) -- (6,3);
			\draw[green!60!black] (0,1) -- (6,4);
			\draw[orange!80!black, dashed] (6,4.05) -- (6,5);
			\draw[orange!80!black] (0,2) -- (6,5);
			\draw[red, very thick] (6,3) -- (6,4) node[midway, right] {$\mu$} node[at end] {\textbullet};
			\draw[red, very thick] (2,2.6) -- (2,3.6) node[midway, left] {$\lambda$} node[at end] {\textbullet}; 
			\draw[white, very thick] (2,-.6) -- (2,.4);        
			\node[scale=.8] at (5,0) {Not Attacking};
		\end{tikzpicture} & \begin{tikzpicture}[scale = .65]
			\node[opacity=0] at (1,5.05) {\textbullet}; 
			\draw[green!60!black, dashed] (0,0) -- (6,3);
			\draw[green!60!black] (0,1) -- (6,4);
			\draw[orange!80!black, dashed] (6,4.05) -- (6,5);
			\draw[orange!80!black] (0,2) -- (6,5);
			\draw[red, very thick] (6,3) -- (6,4) node[midway, right] {$\lambda$} node[at end] {\textbullet};
			\draw[red, very thick] (2,-.6) -- (2,.4) node[midway, left] {$\mu$} node[at end] {\textbullet};            
			\node[scale=.8] at (5,0) {Not Attacking};
		\end{tikzpicture}\\
		\hline
	\end{tabular}
	\caption{Possible temporary diagonal inversion contributes of two labels.}\label{tab:poss_tdinvs}
\end{table}

\newpage

\section{Examples of computation for Algorithm~\ref{alg:toppl}}\label{app:comp_algo}

Referring to the Example~\ref{exa:tildemap} and the two configurations $c$ and $\tilde{c}$, we recollect here the tables with a record of all the grain releases on the configuration during the $\del$ algorithm.
\medskip

In particular, the unstable vertices after each grain release are shown in red and the vertices that have already been grain released once in bold.\\
In the first column we keep track of the vertex that was grain released to get the configuration in such row. Thus, the next grain release can be found by searching for the next either bold or red value of a vertex of the sequence $13,\ 12,\ \dots,\ 2,\ 1$ read ciclically starting from the last grain released vertex.

\begin{table}[htbp!]
    \def\arraystretch{1}
    \setlength{\tabcolsep}{5pt}
    \centering
    \begin{tabular}{c|ccccccccccccc}
        \toprule
        \multicolumn{14}{c}{\textbf{\mathversion{bold}Example of computation for configuration $c$ from Example~\ref{exa:tildemap}}}\\
        \midrule
        $\wtopp(c)$ & 13 & 12 & 11 & 10 & 9 & 8 & 7 & 6 & 5 & 4 & 3 & 2 & 1 \\
        \midrule
            & 6 & 17 & 18 & 24 & 2 & 17 & 14 & 18 & 18 & 15 & 12 & 7 & 2\\
		0 & 7 & 18 & 19 & \textbf{\color{red}25} & 3 & 18 & 15 & 19 & 19 & 16 & 13 & 8 & 3\\ 
		10 & 8 & 19 & 20 & \textbf{13} & 4 & 19 & 16 & \textbf{\color{red}20} & \textbf{\color{red}20} & 17 & 14 & 9 & 4\\ 
		6 & 9 & 20 & 21 & \textbf{14} & 5 & 20 & 17 & \textbf{13} & \textbf{\color{red}20} & 17 & 14 & 9 & 4\\ 
		5 & 10 & 21 & 22 & \textbf{15} & 6 & 21 & 18 & \textbf{13} & \textbf{13} & 17 & 14 & 9 & 4\\ 
		10 & 11 & 22 & 23 & 2 & 7 & 22 & 19 & \textbf{14} & \textbf{14} & 18 & 15 & 10 & 5\\ 
		6 & 12 & 23 & 24 & 3 & 8 & 23 & 20 & 1 & \textbf{15} & 19 & 16 & 11 & 6\\ 
		5 & 13 & 24 & \textbf{\color{red}25} & 4 & 9 & \textbf{\color{red}24} & 21 & 2 & 2 & \textbf{\color{red}20} & 17 & 12 & 7\\ 
		4 & 14 & \textbf{\color{red}25} & \textbf{\color{red}26} & 5 & 10 & \textbf{\color{red}25} & 22 & 2 & 2 & \textbf{13} & 17 & 12 & 7\\ 
		12 & 15 & \textbf{13} & \textbf{\color{red}27} & 6 & 11 & \textbf{\color{red}26} & 23 & 3 & 3 & \textbf{14} & 18 & 13 & 8\\ 
		11 & 16 & \textbf{14} & \textbf{15} & 7 & 12 & \textbf{\color{red}27} & \textbf{\color{red}24} & 4 & 4 & \textbf{15} & 19 & 14 & 9\\ 
		8 & 17 & \textbf{15} & \textbf{16} & 8 & 13 & \textbf{16} & \textbf{\color{red}24} & 5 & 5 & \textbf{16} & \textbf{\color{red}20} & 15 & 10\\ 
		7 & 18 & \textbf{16} & \textbf{17} & 9 & 14 & \textbf{16} & \textbf{13} & 6 & 6 & \textbf{17} & \textbf{\color{red}21} & 16 & 11\\ 
		4 & 19 & \textbf{17} & \textbf{18} & 10 & 15 & \textbf{17} & \textbf{14} & 7 & 7 & 4 & \textbf{\color{red}22} & 17 & 12\\ 
		3 & 20 & \textbf{18} & \textbf{19} & 11 & 16 & \textbf{18} & \textbf{15} & 7 & 7 & 4 & \textbf{15} & 17 & 12\\ 
		12 & 21 & 5 & \textbf{20} & 12 & 17 & \textbf{19} & \textbf{16} & 8 & 8 & 5 & \textbf{16} & 18 & 13\\ 
		11 & 22 & 6 & 7 & 13 & 18 & \textbf{20} & \textbf{17} & 9 & 9 & 6 & \textbf{17} & 19 & 14\\ 
		8 & 23 & 7 & 8 & 14 & 19 & 7 & \textbf{18} & 10 & 10 & 7 & \textbf{18} & \textbf{\color{red}20} & 15\\ 
		7 & 24 & 8 & 9 & 15 & 20 & 8 & 5 & 11 & 11 & 8 & \textbf{19} & \textbf{\color{red}21} & 16\\ 
		3 & \textbf{\color{red}25} & 9 & 10 & 16 & 21 & 9 & 6 & 12 & 12 & 9 & 6 & \textbf{\color{red}22} & 17\\ 
		2 & \textbf{\color{red}26} & 10 & 11 & 17 & 22 & 10 & 7 & 12 & 12 & 9 & 6 & \textbf{15} & 17\\ 
		13 & \textbf{14} & 11 & 12 & 18 & 23 & 11 & 8 & 13 & 13 & 10 & 7 & \textbf{16} & 18\\ 
		2 & \textbf{15} & 12 & 13 & 19 & 24 & 12 & 9 & 14 & 14 & 11 & 8 & 3 & 19\\ 
		13 & 2 & 13 & 14 & 20 & \textbf{\color{red}25} & 13 & 10 & 15 & 15 & 12 & 9 & 4 & \textbf{\color{red}20}\\ 
		9 & 3 & 14 & 15 & 21 & \textbf{13} & 14 & 11 & 16 & 16 & 13 & 10 & 5 & \textbf{\color{red}21}\\ 
		1 & 4 & 15 & 16 & 22 & \textbf{14} & 15 & 12 & 16 & 16 & 13 & 10 & 5 & \textbf{14}\\ 
		9 & 5 & 16 & 17 & 23 & 1 & 16 & 13 & 17 & 17 & 14 & 11 & 6 & \textbf{15}\\ 
		1 & 6 & 17 & 18 & 24 & 2 & 17 & 14 & 18 & 18 & 15 & 12 & 7 & 2
    \end{tabular}
    \caption{Table recording all the grain releases for configuration $c$ from Example~\ref{exa:tildemap}.}\label{tab:delayc}
\end{table}

\newpage

\begin{table}[tbp!]
    \def\arraystretch{1}
    \setlength{\tabcolsep}{5pt}
    \centering
    \begin{tabular}{c|ccccccccccccc}
        \toprule
        \multicolumn{14}{c}{\textbf{\mathversion{bold}Example of computation for configuration $\tilde{c}$ from Example~\ref{exa:tildemap}}}\\
        \midrule
        $\wtopp(\tilde{c})$ & 13 & 12 & 11 & 10 & 9 & 8 & 7 & 6 & 5 & 4 & 3 & 2 & 1 \\
        \midrule
          & 6 & 17 & 18 & 24 & 2 & 18 & 14 & 23 & 22 & 18 & 14 & 8 & 2\\
        0 & 7 & 18 & 19 & \textbf{\color{red}25} & 3 & 19 & 15 & 24 & 23 & 19 & 15 & 9 & 3\\ 
		10 & 8 & 19 & 20 & \textbf{13} & 4 & 20 & 16 & \textbf{\color{red}25} & 24 & 20 & 16 & 10 & 4\\ 
		6 & 9 & 20 & 21 & \textbf{14} & 5 & 21 & 17 & \textbf{13} & \textbf{\color{red}25} & 21 & 17 & 11 & 5\\ 
		5 & 10 & 21 & 22 & \textbf{15} & 6 & 22 & 18 & \textbf{14} & \textbf{13} & 22 & 18 & 12 & 6\\ 
		10 & 11 & 22 & 23 & 2 & 7 & 23 & 19 & \textbf{15} & \textbf{14} & 23 & 19 & 13 & 7\\ 
		6 & 12 & 23 & 24 & 3 & 8 & 24 & 20 & 2 & \textbf{15} & 24 & 20 & 14 & 8\\ 
		5 & 13 & 24 & \textbf{\color{red}25} & 4 & 9 & \textbf{\color{red}25} & 21 & 3 & 2 & \textbf{\color{red}25} & 21 & 15 & 9\\ 
		4 & 14 & \textbf{\color{red}25} & \textbf{\color{red}26} & 5 & 10 & \textbf{\color{red}26} & 22 & 4 & 3 & \textbf{13} & 22 & 16 & 10\\ 
		12 & 15 & \textbf{13} & \textbf{\color{red}27} & 6 & 11 & \textbf{\color{red}27} & 23 & 5 & 4 & \textbf{14} & 23 & 17 & 11\\ 
		11 & 16 & \textbf{14} & \textbf{15} & 7 & 12 & \textbf{\color{red}28} & 24 & 6 & 5 & \textbf{15} & 24 & 18 & 12\\ 
		8 & 17 & \textbf{15} & \textbf{16} & 8 & 13 & \textbf{16} & \textbf{\color{red}25} & 7 & 6 & \textbf{16} & \textbf{\color{red}25} & 19 & 13\\ 
		7 & 18 & \textbf{16} & \textbf{17} & 9 & 14 & \textbf{17} & \textbf{13} & 8 & 7 & \textbf{17} & \textbf{\color{red}26} & 20 & 14\\ 
		4 & 19 & \textbf{17} & \textbf{18} & 10 & 15 & \textbf{18} & \textbf{14} & 9 & 8 & 4 & \textbf{\color{red}27} & 21 & 15\\ 
		3 & 20 & \textbf{18} & \textbf{19} & 11 & 16 & \textbf{19} & \textbf{15} & 10 & 9 & 5 & \textbf{15} & 22 & 16\\ 
		12 & 21 & 5 & \textbf{20} & 12 & 17 & \textbf{20} & \textbf{16} & 11 & 10 & 6 & \textbf{16} & 23 & 17\\ 
		11 & 22 & 6 & 7 & 13 & 18 & \textbf{21} & \textbf{17} & 12 & 11 & 7 & \textbf{17} & 24 & 18\\ 
		8 & 23 & 7 & 8 & 14 & 19 & 8 & \textbf{18} & 13 & 12 & 8 & \textbf{18} & \textbf{\color{red}25} & 19\\ 
		7 & 24 & 8 & 9 & 15 & 20 & 9 & 5 & 14 & 13 & 9 & \textbf{19} & \textbf{\color{red}26} & 20\\ 
		3 & \textbf{\color{red}25} & 9 & 10 & 16 & 21 & 10 & 6 & 15 & 14 & 10 & 6 & \textbf{\color{red}27} & 21\\ 
		2 & \textbf{\color{red}26} & 10 & 11 & 17 & 22 & 11 & 7 & 16 & 15 & 11 & 7 & \textbf{15} & 22\\ 
		13 & \textbf{14} & 11 & 12 & 18 & 23 & 12 & 8 & 17 & 16 & 12 & 8 & \textbf{16} & 23\\ 
		2 & \textbf{15} & 12 & 13 & 19 & 24 & 13 & 9 & 18 & 17 & 13 & 9 & 3 & 24\\ 
		13 & 2 & 13 & 14 & 20 & \textbf{\color{red}25} & 14 & 10 & 19 & 18 & 14 & 10 & 4 & \textbf{\color{red}25}\\ 
		9 & 3 & 14 & 15 & 21 & \textbf{13} & 15 & 11 & 20 & 19 & 15 & 11 & 5 & \textbf{\color{red}26}\\ 
		1 & 4 & 15 & 16 & 22 & \textbf{14} & 16 & 12 & 21 & 20 & 16 & 12 & 6 & \textbf{14}\\ 
		9 & 5 & 16 & 17 & 23 & 1 & 17 & 13 & 22 & 21 & 17 & 13 & 7 & \textbf{15}\\ 
		1 & 6 & 17 & 18 & 24 & 2 & 18 & 14 & 23 & 22 & 18 & 14 & 8 & 2
    \end{tabular}
    \caption{Table recording all the grain releases for configuration $\tilde{c}$ from Example~\ref{exa:tildemap}.}\label{tab:delaytildec}
\end{table}

Observe that in both Tables~\ref{tab:delayc} and~\ref{tab:delaytildec} the first column represents the order of grain releases $\wtopp$ and that the two coincides, as proved in Proposition~\ref{prop:alg_toppl}.

\end{document}